\documentclass[10pt,reqno,final]{amsart}
\usepackage{amssymb,amsmath,amsthm}
\usepackage{mathrsfs,pifont}
\usepackage{graphicx}
\usepackage[notcite,notref]{showkeys}
\usepackage{url,hyperref}
 \usepackage{color}
\usepackage{float,version}
\usepackage{subfigure}

    \numberwithin{equation}{section}
    \numberwithin{table}{section}
    \numberwithin{figure}{section}

\newtheorem{thm}{Theorem}[section]

\newtheorem{lemma}{Lemma}[section]
\newtheorem{rem}{Remark}[section]
\theoremstyle{definition}
\newtheorem{defn}{\bf Definition}[section]

\newcommand{\re}[1]{(\ref{#1})}

\def \ri {{\rm i}}

\def \vs {\hspace*{0.5cm}}
\def \af {\alpha}

\def \om {\omega}
\def \swf {\psi^{\alpha,n}}
\def\CP{{\mathcal P}}
\def \gi {G^{(\af)}}

\def \bx{\bs x}
\def \bxi{\bs \xi}

\def\i{\mathrm{i}}
\def\e{\mathrm{e}}
\def\d{\mathrm{d}}
\def\RR{\mathbb{R}}

\def\sph{\mathbb{S}}
\def\ball{\mathbb{B}}

\def\CP{{\mathcal P}}
\def\CH{{\mathcal H}}

\def\s{\sigma}
\def\NN{\mathbb{N}}
\def\tr{\mathsf{t}}

\newcommand{\bs}[1]{\boldsymbol{#1}}
\begin{document}

\graphicspath{{figs/}}

\title[Prolate spheroidal wave functions on a ball]
{Ball Prolate spheroidal wave functions in  arbitrary dimensions}
\author[
	J. Zhang,\; H. Li,\; L. Wang\; $\&$\; Z. Zhang
	]{
		\;\; Jing Zhang${}^1$,    \;\; Huiyuan Li${}^{2}$  \;\; Li-Lian Wang${}^{3}$, \;\;  and\;\; Zhimin Zhang${}^{4}$
		}
	\thanks{${}^1$School of Mathematics and Statistics $\&$  Hubei Key Laboratory of Mathematical Sciences,  Central China Normal University, Wuhan 430079, China. The work of this author  is partially supported by the National Natural Science Foundation of China (NSFC 11671166). \\
		\indent ${}^{2}$State Key Laboratory of Computer Science/Laboratory of Parallel Computing,  Institute of Software, Chinese Academy of Sciences, Beijing 100190, China. The research of this author is partially   supported by the National Natural Science Foundation of China (NSFC 11471312, 91430216 and 91130014).\\
		\indent ${}^{3}$Division of Mathematical Sciences, School of Physical
		and Mathematical Sciences, Nanyang Technological University,
		637371, Singapore. The research of this author is partially supported by Singapore MOE AcRF Tier 1 Grant (RG 27/15). \\
		\indent ${}^{4}$Beijing Computational Sciences and Research Center, Beijing 100193, China,  and Department of Mathematics, Wayne State University, MI 48202, USA. The research of this author is supported
in part by the National Natural Science Foundation of China (NSFC 11471031 and 91430216),
the Joint Fund of the National Natural Science Foundation of China and the China Academy of Engineering Physics (NSAF U1530401),  and the U.S.
National Science Foundation (DMS-1419040).
}




\begin{abstract}
In this paper, we introduce the  prolate spheroidal wave functions (PSWFs)  of real order $\af>-1$ on the  unit ball
in arbitrary dimension, termed as  ball PSWFs. They
 are  eigenfunctions of  both a weighted concentration integral operator, and
a Sturm-Liouville differential operator.  Different  from  existing works on multi-dimensional  PSWFs,  the ball PSWFs are
defined as a  generalisation of orthogonal {\em ball polynomials} in primitive variables with a  tuning  parameter $c>0$,
through a ``perturbation" of  the   Sturm-Liouville equation of the ball polynomials.
From this perspective,  we can explore   some  interesting   intrinsic connections between the ball PSWFs and  the finite Fourier  and  Hankel transforms.
%
%
We provide an efficient and accurate algorithm for computing  the ball PSWFs and the associated eigenvalues, and present various numerical results to illustrate the efficiency of the method. Under this uniform framework, we can recover the existing PSWFs by suitable variable substitutions.
\end{abstract}

\maketitle


\section{Introduction}
The PSWFs  are a family of orthogonal
bandlimited functions, originated from the investigation of time-frequency concentration problem in the 1960s (cf. \cite{Landau61,Landau62,Slep61,Slep64}).
In the study of time-frequency concentration problem, Slepian was the first to note that the PSWFs, denoted by $\big\{\psi_n(x;c)\big\}_{n=0}^\infty$,
are the  eigenfunctions of an integral operator  related to the finite Fourier transform:
\begin{equation}\label{pswf0int}
\lambda_n(c)\psi_n(x;c) =\int_{-1}^1 \e^{\ri cxt}\psi_n(t;c) \d t,\quad c>0,\quad  x\in I:=(-1,1),
\end{equation}
where $c > 0$ is the so-called bandwidth parameter determined by the concentration rate and concentration interval, and $\{\lambda_n(c)\}$ are the corresponding eigenvalues.
By a remarkable coincidence, Slepian et al. \cite{Slep61} recognized that the PSWFs
also form the eigen-system  of the second-order singular Sturm-Liouville differential equation, 
\begin{equation}\label{Pswfzero}
\partial_x \big((1-x^2)\partial_x \psi_n(x;c)\big)+\big(\chi_n(c)-c^2x^2\big)\psi_n(x;c)=0,\quad c>0,\quad x\in I,
\end{equation}
which appears in separation of variables for solving the Helmholtz equation in spheroidal coordinates.
The Sturm-Liouville  equation links up the PSWFs with  orthogonal polynomials,  and  this connection  plays a key role in the study of the PSWFs.

The properties inherent to these functions have subsequently
attracted many attentions for decades. Within the last few years,
there has been a growing research interest in various aspects of the
PSWFs including analytic and asymptotic studies \cite{Xiao.R03,Boyd.AMC03,RokXiao07,Botezatu2016}, approximation with PSWFs \cite{Sengupta2012,Bonami2015,XiaoH.R01,XiaT01,Osipov2014},
numerical evaluations \cite{Bouwkamp1947On,Boyd.acm,Walter.S05,AndLV07,Karo.M08,Alici.Shen16,Lederman17}, development of numerical methods using this bandlimited basis \cite{Chen2006,Kong2012A,Wang.ZZ14,Karnik.Wa16}.  In particular,   we refer to the monographs    \cite{Hogan12,Osipov13} and the recent review paper \cite{Wang2017A} for many references therein.

The extensions of the time-frequency concentration problems on a finite interval
to other geometries have been considered in e.g.,  \cite{Slep64,Bey.K07,F.Sim06,Karo.M16,Khalid.K16,Shkolnisky07,Zhang.WLZ17}. In \cite{Slep64}, D. Slepian   extended the finite Fourier transform \eqref{pswf0int} to a bounded multidimensional domain $\Omega\subset \RR^d$,
\begin{align}
\label{F2d}
\lambda \psi(\bx) = \int_{\Omega} \psi({\bs \tau})\e^{-{\rm{i}} c \langle \bx,  {\bs \tau}\rangle}\d{\bs \tau}, \quad \bx \in \Omega,
\end{align}
and  then investigated the time-frequency concentration  on the unit disk $\mathbb{B}^2$.
Their effort stimulated researchers'  interest to the discussion of \emph{generalized prolate spheroidal wave functions} in two dimensions.
 Beylkin et al. \cite{Bey.K07} explored some interesting properties of band-limited functions on a disk.
In \cite{Shkolnisky07,Landa2016}, the authors  studied  the integration and approximation of the PSWFs on a disk.
As usual, these \emph{generalized PSWFs} on  the disc satisfy
 the Sturm-Liouville differential equation and the integral equation at the same time. We also note that
Taylor \cite{M.Tay06}  generalized the PSWFs to the triangle by defining  a special type of Sturm-Liouville equation.

In contrast,  time-frequency concentration problem over a bounded  domain in   higher dimension
has received very limited  attention.
The works \cite{Miranian2004,F.Sim06,Bates2016} studied the time-frequency concentration problem on a sphere.   Khalid et al \cite{Khalid.K16} formulated and solved the analog of Slepian spatial-spectral concentration problem on the three-dimensional ball,
 and Michel et al  \cite{Michel} extended it   to  vectorial case.   We note that  the  time-frequency/spatial-spectral concentration in both cases  is applicable  for  ``bandlimited functions"  with a finite (spherical harmonic or Bessel-spherical harmonic) expansion instead of those whose Fourier transform have a bounded  support. More importantly,
 many properties, in particular those relating to orthogonal polynomials, are still unknown without a Strum-Liouville differential equation.

In this paper, we
propose a generalization of PSWFs of real order $\af>-1$ on the  unit ball  $\mathbb{B}^d:=\{ \bx\in\RR^d:  \|\bx\|<1\}$ of an arbitrary dimension $d$.   The ball  PSWFs in the current paper  inherit  the merit of  PSWFs in one dimension such that they are eigenfunctions of an integral operator and a differential  operator simultaneously.

In the first place,  we  introduce a Sturm-Liouville  differential equation  and then define the ball PSWFs as   eigenfunctions of the eigen-problem:
 \begin{equation}\label{ball_pswf}
\big[-(1-\|\bx\|^2)^{-\af}\nabla\cdot({\bs {\rm I}}-\bx\bx^{\tr})(1-\|\bx\|^2)^{\af}\nabla+ c^2 \|\bx\|^2\big]\psi(\bx; c)=\chi\, \psi(\bx; c),\;\; \bx\in\ball^d, \  \alpha>-1.
\end{equation}
Hereafter, composite differential operators  are  understood in the convention of right associativity, for instance,
$$\nabla\cdot({\bs {\rm I}}-\bx\bx^{\tr})(1-\|\bx\|^2)^{\af}\nabla
=\nabla\cdot[({\bs {\rm I}}-\bx\bx^{\tr})(1-\|\bx\|^2)^{\af}\nabla].$$
In distinction to   \cite{Slep64} and other related works,  the Sturm-Liouville  differential equation  \eqref{ball_pswf} here  is defined in primitive variables instead of  the radial variable.
It extends the one-dimensional Sturm-Liouville differential equation \eqref{Pswfzero}  intuitively  while  preserves the key features: symmetry, self-adjointness and  form of the bandwidth term  $c^2 \|\bx\|^2$.
More importantly,  \eqref{ball_pswf} extends the orthogonal ball polynomials \cite{Dai2013}  (the case with  $c=0$) to ball PSWFs with a tuning parameter $c>0$. The implication  is twofold.
This not only  provides  a tool  to derive  analytic and asymptotic formulae for the PSWFs on an arbitrary unit ball and the associated eigenvalues, but also   offers an optimal Bouwkamp  spectral-algorithm for the computation of PSWFs just as in one dimension \cite{Bouw50}: expand them in the basis of the orthogonal ball polynomials, and reduce the problem to  an generalized algebraic eigenvalue problem with a tri-digonal matrix.

The second purpose of this paper is  to make an investigation of the integral transforms
behind the ball PSWFs,  and explore their  connections with existing works.
More specifically, we can  show  that  the commutativity  of
 the Sturm-Liouville  differential operator  in \eqref{ball_pswf}  with the integral operator of the finite Fourier transform.  As a result,  the ball PSWFs are also eigenfunctions of the finite Fourier transform:
\begin{equation}\label{psi00}
\lambda\psi(\bx;c) =\int_{\ball^d}\e^{-\ri
c \langle \bx, \bs \tau\rangle}\psi(\bs \tau; c)(1-\|\bx\|^2)^{\af}\d\bs \tau
:= [{\mathscr F}_{c}^{(\af)}\psi](\bx;c),\quad \bx\in\ball^d,\  c>0,\;\af>-1.
\end{equation}
Morover, it has been demonstrated that the $(d-1)$-dimensional  spherical harmonics ($Y^n_{\ell}, \, 1\le \ell\le a_n^d,\, n\ge 0;  $ see  \S \ref{SH} and refer to \cite{Dai2013}) are  eigenfunctions of the  Fourier transform on the unit sphere $\mathbb{S}^{d-1}$ {\cite[Lemma 9.10.2]{Askey}}. 
Thus,  by writing
$$\psi(\bx;c)=\|\bx\|^{\frac{1-d}{2}}\phi(\|\bx\|;c)Y^n_{\ell}(\bx /\|\bx\|),$$
the finite Fourier transform \eqref{psi00} is reduced to the equivalent (symmetric) finite Hankel transform in radial direction (also refer to \cite[Eq.\,(i)]{Slep64} for the case $d=2$ and $\alpha=0$),
\begin{align}
(2\pi)^{-\frac{d}{2}}c^{\frac{d-1}{2}}\lambda\,  \i^n \, \phi(\rho;c)= \int_{0}^{1}  J_{n+\frac{d-2}{2}}(c\rho r)\phi(r;c)\sqrt{c\rho r} (1-r^2)^{\alpha} \d r, \quad 0<\rho<1
.\label{hphi}
\end{align}
The eigenfunctions $\phi(r;c) $ of \eqref{hphi}, which are also referred to as  {\em generalized prolate spheroidal wave functions} in \cite{Slep64}, are further shown to be the bounded solutions of the following  Sturm-Liouville differential equation:
\begin{align}\label{SLphi}
\begin{split}
&\Big[-(1-r^2)^{-\alpha}\partial_r(1-r^2)^{\alpha+1}\partial_r
+\frac{(2n+d-1)(2n+d-3)}{4r^2}+c^2r^2\Big]
\phi(r;c)
\\
  =&\Big[\chi+\frac{(d-1)(4\alpha+d+1)}4\Big]\phi(r;c) .
  \end{split}
\end{align}
One  can also refer to \cite[Eq.\,(ii)]{Slep64} for the case $\alpha=0$ and $d=2$, and refer to \eqref{Pswfzero} for the case $\alpha=0$ and $d=1$ in which $n\in\{0,1\}$.
In such a way,  \eqref{ball_pswf},  \eqref{psi00},  \eqref{hphi} and \eqref{SLphi}  reveal the intrinsic connections among the finite  Fourier transform, finite Hankel transform
and the Sturm-Liouville differential operator behind the   ball PSWFs.

%
%

The rest of the paper is organized as follows. In Section \ref{Sect:2}, we introduce some
of the special functions and orthogonal polynomials,
and collect their relevant properties  to be used throughout the paper.
 In Section \ref{secG}, we propose the Sturm-Liouville  differential equation on an arbitrary unit ball in primitive variables, define  the  ball PSWFs and study their
analytic properties.   In Section \ref{sect4}, we  study the ball PSWFs as eigenfunctions of  the  integral operators,
make investigations of their (finite) Fourier transform and  (finite) Hankel transform, and
present other important features of  ball PSWFs.
An efficient method for computing the ball PSWFs  using the differential operator \eqref{ball_pswf}
together  with the connection with existing works is descibed in Section \ref{sect:5}. Numerical results
are provided  to justify our theory  and  to demonstrate the efficiency of our algorithm.

\section{Special functions: spherical  harmonics and ball polynomials}\label{Sect:2}

In this section, we review some relevant special functions which especially  include the spherical harmonics and ball polynomials.
More importantly, we derive some new formulations and properties to facilitate the discussions in the forthcoming sections.

\subsection{Some related orthogonal polynomials and special functions}
We briefly  review the relevant  properties  of some orthogonal  polynomials and related special functions to be used throughout this paper, which can be found in various resources,
   see    e.g., \cite{Abra72, Dai2013, Xu2001,STW11}.

For real $\af,\beta>-1$, the normalized Jacobi polynomials, denoted by
$\{P_{n}^{(\af,\beta)}(\eta)\}_{n\ge 0},$   satisfy  the
three-term recurrence relation:
\begin{equation}\label{Jacobi}
\begin{split}
&\eta {P}_n^{(\af,\beta)}(\eta)=a_n^{(\af,\beta)} {P}_{n+1}^{(\af,\beta)}(\eta)+b_n^{(\af,\beta)} {P}_{n}^{(\af,\beta)}(\eta)+a_{n-1}^{(\af,\beta)} {P}_{n-1}^{(\af,\beta)}(\eta),
\\
& P_{0}^{(\af,\beta)}(\eta)=\frac{1}{h^{(\alpha,\beta)}_0},
\quad
P_{1}^{(\af,\beta)}(\eta)=  \frac{1}{2h^{(\alpha,\beta)}_1}
  \big( (\af+\beta+2)\eta+(\af-\beta)\big),
\end{split}
\end{equation}
where $\eta\in I:=(-1,1)$, and
\begin{align*}
&a_n^{(\af,\beta)}=\sqrt{ \frac{4(n+1)(n+\alpha+1)(n+\beta+1)(n+\alpha+\beta+1)}{(2n+\alpha+\beta+1)(2n+\alpha+\beta+2)^2(2n+\alpha+\beta+3) } }
,
\\
&b_n^{(\af,\beta)} = \frac{\beta^2-\alpha^2}{(2n+\alpha+\beta)(2n+\alpha+\beta+2)},
  \\
  & h_n^{(\af,\beta)} = \sqrt{\frac{\Gamma(n+\alpha+1)\Gamma(n+\beta+1)}{2(2n+\alpha+\beta+1) \Gamma(n+1) \Gamma(n+\alpha+\beta+1)}   }.
\end{align*}
Let $\omega^{\af,\beta}(\eta)=(1-\eta)^{\af}(1+\eta)^{\beta}$ be the Jacobi weight function. The normalized Jacobi polynomials are orthonormal in the sense that
\begin{equation}\label{Jacobiorth}
\int_{-1}^{1} {P}_{n}^{(\af,\beta)}(\eta){P}_{m}^{(\af,\beta)}(\eta)\omega_{\af,\beta} (\eta)
\d{\eta}=2^{\alpha+\beta+2}\delta_{nm}.
\end{equation}
The leading coefficient of $P_n^{(\af,\beta)}(\eta)$ is
\begin{equation}\label{knkn}
\kappa_n^{(\af,\beta)}
=\frac{1}{2^{n} h_n^{(\alpha,\beta)}} \binom{2n+\alpha+\beta}{n}.
\end{equation}
The Jacobi polynomials are the eigenfunctions of the
Sturm-Liouville problem
\begin{equation}\label{JacobiSL}
\mathscr{L}_{\eta}^{(\af,\beta)} P_{n}^{(\af,\beta)}(\eta):=
-\frac{1}{\omega_{\af,\beta}(\eta)}
\partial_{\eta}\big(\omega_{\af+1,\beta+1}(\eta)\partial_{\eta}P_{n}^{(\af,\beta)}(\eta)\big)=
 \lambda_n^{(\af,\beta)}P_n^{(\af,\beta)}(\eta),\quad  \eta\in I,
\end{equation}
with the corresponding  eigenvalues $\lambda_n^{(\af,\beta)}=n(n+\af+\beta+1).$

In this paper, we shall also use  the Bessel function of the first kind of order $\nu>-1/2$, denoted by $J_{\nu}(z)$. It satisfies the Bessel's  equation:
\begin{equation*}
 z^2 \partial_z^2 J_{\nu}(z) + z  \partial_z J_{\nu}(z)+(z^{2}-\nu ^{2}) J_{\nu}(z)=0,
 \quad {z\ge 0}.
\end{equation*}
and has the  Poisson integral representation:
\begin{equation}\label{JPoisson}
\begin{split}
J_{\nu}(z)=\frac{z^{\nu}}{2^\nu \sqrt \pi \Gamma(\nu+\frac{1}{2})}
\int_{-1}^1e^{\ri zt}(1-t^2)^{\nu-\frac{1}{2}}\d{t},\quad { z\ge 0},\; \nu >-\frac 1 2.
\end{split}
\end{equation}
Moreover, we have
\begin{align}
\label{Jexpan}
J_{\nu}(z)=\sum _{m=0}^{\infty }{\frac {(-1)^{m}}{m!\,\Gamma (m+\nu +1)}}
{\left({\frac {z}{2}}\right)}^{2m+\nu },\quad \nu\geq 0.
\end{align}
and  (cf. \cite{Watson44}):
\begin{align}
\label{DJz}
\partial_z \left(\frac{J_\nu(z)}{z^{\nu}}\right)= - \frac{J_{\nu+1}(z)}{z^{\nu}},\quad z>0,\;\nu >-\frac 1 2.
\end{align}

\subsection{Spherical harmonics}
\label{SH}We first introduce some notation. Let $\mathbb{R}^d$ be the  $d$-dimensional Euclidean space. For $\bx\in \mathbb{R}^d$, we write $\bx =
(x_1,\cdots,x_d)^{\tr}$ as a column vector, where $(\cdot)^{\tr}$ denotes matrix or vector transpose. The inner product of $\bx,\bs y\in \mathbb{R}^d$ is denoted by $\bx\cdot\bs y$ or $\langle\bx,\bs y\rangle:=   \bx^{\tr} \bs y =\sum^d_{i=1} x_iy_i$, and the
norm of $\bx$ is denoted by $\|\bx\| := \sqrt{ \langle\bx, \bx\rangle}=\sqrt{\bx^{\tr}\bx}$.
The unit sphere $\mathbb{S}^{d-1}$ and the unit ball $\ball^d$ of $\mathbb{R}^d$ are respectively defined by
\begin{equation*}
\mathbb{S}^{d-1}:=\big\{\hat \bx\in \RR^d: \|\hat \bx\|=1\big\},\quad \ball^d:=\big\{\bx\in \RR^d: r=\|\bx\|{ \leq}1\big\}.
\end{equation*}
For   each $ \bx\in \RR^d$, we introduce its polar-spherical coordinates $(r,\hat \bx)$ such that $r=\|\bx\|$ and $ \bx =r\hat\bx,\;\hat\bx \in \mathbb{S}^{d-1}.$ Define the inner product of $L^2(\mathbb{S}^{d-1})$ as
\begin{equation}\label{fginner}
        ( f, g )_{\mathbb{S}^{d-1}}: = \int_{\mathbb{S}^{d-1}} f(\hat \bx ) g(\hat \bx) \d\s(\hat \bx),
\end{equation}
where $d \s$ is the surface measure.
Define the differential operator 
\begin{align}
\label{Dij}
D_{ij} = x_j \partial_{x_i}-x_i \partial_{x_j}  =\partial_{\theta_{ij}},\quad 1\le i\neq j \le d,
\end{align}
where $\theta_{ij}$ is the angle of polar coordinates in the $(x_i,x_j)$-plane by $(x_i,x_j)= r_{ij} ( \cos \theta_{ij} , \sin \theta_{ij} )$ with $r_{ij}\ge 0$ and $0\le \theta_{ij}\le 2\pi$.
Then the Laplace-Beltrami operator $\Delta_0$ (i.e., the spherical part of $\Delta$) is  defined by
\cite{Dai2013}
\begin{align}
\label{Delta0}
\Delta_0 = \sum_{1\le j  < i\le d} D_{ij}^2.
\end{align}

Let $\CP_n^d$ be the space of homogeneous polynomials of degree
$n$ in $d$ variables, i.e.,
$$
\CP_n^d={\rm span}\big\{ \bx^{\bs k} = x_1^{k_1}x_2^{k_2}\dots x_d^{k_d}\,:\, |\bs k|=k_1+k_2+\cdots+ k_d=n\big\}.
$$
 Define the space of all harmonic polynomials of degree $n$ as
\begin{equation*}
\mathcal{H}_n^d:=\big\{p\in \CP_n^d:\Delta p=0 \big\}.
\end{equation*}
It is seen that  a harmonic polynomial of degree $n$ is a homogeneous polynomial degree $n$ that satisfies the Laplace equation.

Spherical harmonics are the restriction of harmonic polynomials on the unit sphere.
Note that for any $Y\in \CH_n^d$, we have
\begin{equation}\label{Ybxcase}
Y(\bx) = r^n Y(\hat \bx),\quad \bx = r \hat \bx, \;\; r=\|\bx\|,\;\;\hat{\bx}\in \mathbb{S}^{d-1},
\end{equation}
in the spherical polar coordinates. 
It is evident that $Y(\bx)$ is uniquely determined by its restriction $Y(\hat \bx)$ on the sphere.
With a little abuse of notation, we still use $\CH_n^d$ to denote the set of all spherical
harmonics of degree $n$ on the unit sphere $\mathbb{S}^{d-1}$. Here,  we understand that the variable
is $\hat \bx$, i.e.,
$$
\CH_n^d=\{ Y(\hat {\bx}): \hat {\bx}\in \mathbb{S}^{d-1},\,  Y \in \CP_n^d, \, \Delta Y =0\}.
$$

In spherical polar coordinates, the Laplace operator can be written as
\begin{equation} \label{eq:Delta}
   \Delta = \frac{d^2}{d r^2} + \frac{d-1}{r} \frac{d}{dr} + \frac{1}{r^2} \Delta_0,
\end{equation}
so for  any $Y\in \CP_n^d$,
$$
   \Delta Y(\bx) = \Delta [r^nY(\hat \bx)] = n(n+d-2)\, r^{n-2} Y(\hat \bx) + r^{n-2} \Delta_0Y(\hat \bx).
$$
Thus, the spherical
harmonics are  eigenfunctions of the Laplace-Beltrami operator,
\begin{equation} \label{eq:LaplaceBeltrami}
        \Delta_0  Y(\hat \bx) = - n(n+d-2) Y(\hat \bx), \quad Y \in \CH_n^d,\quad\hat \bx \in \mathbb{S}^{d-1},
\end{equation}
As a result, the  spherical harmonics of different degree $n$ are orthogonal with respect to the inner product
$( \cdot, \cdot )_{\mathbb{S}^{d-1}}$.

It is known that (cf. \cite{Dai2013})
\begin{equation}\label{dimenCP}
    \dim \CP_n^d  = \binom{n+d-1}{n},  \quad
      a_n^d: = \dim \mathcal{H}_n^d = \binom{n+d-1}{n} - \binom{n+d-3}{n-2}.
\end{equation}
In what follows, for fixed $n\in \NN_0$,  we always denote by $\{Y_\ell^n: 1 \le \ell \le a_n^d\}$ the   (real) orthonormal  basis of $\CH_n^d$. In view of \eqref{eq:LaplaceBeltrami}, we have the orthogonality:
\begin{align}
\label{Yorth}
 ( Y_\ell^n, Y_\iota^m)_{\mathbb{S}^{d-1}}=\delta_{nm} \delta_{\ell\iota},
 \quad  \ell \in \Upsilon_{\!n}^d, \;\;    \iota\in \Upsilon_{\!m}^d,
 \end{align}
 where for notational convenience,  we introduce the index set
 \begin{equation}\label{indexsetA}
 \Upsilon_{\!n}^d=\{l\,:\, 1\le l\le a_n^d\},\quad d, n\in {\mathbb N}.
 \end{equation}

\begin{rem}\label{exaples} {~}
\begin{itemize}
\item For $d=1$,  there exist only two  orthonormal harmonic polynomials: $Y^0_1=\frac{1}{\sqrt{2}}$ and $Y^1_1=\frac{x}{\sqrt{2}}$. 
\item For $d=2$,  the space  $\mathcal{H}_n^2$ has dimension  $a_n^2=2-\delta_{n0}$ and the orthogonal basis of $\mathcal{H}_n^2$ can be  given by the real
and imaginary parts of $(x_1+\i x_2)^n$. Thus, in polar coordinates $\bx = (r\cos \theta, r\sin \theta)^{\tr}\in \RR^2$,
we simply take
$$
Y_1^0(\bx) = \frac{1}{\sqrt{2\pi}}, \quad  Y_1^n(\bx) =  \frac{ r^n}{\sqrt{ \pi}} \cos n \theta, \quad  Y_2^n(\bx) =  \frac{ r^n}{\sqrt{\pi}}  \sin n \theta, \quad n\ge 1.
$$
\item For $d=3$, the dimensionality of the harmonic polynomial space of degree $n$ 
is  $a_n^3=2n+1$. In spherical coordinates
$\bx=(r\sin \theta \cos \phi, r\sin \theta \cos \phi, r\cos \theta)^{\tr}\in \RR^3$, the orthonormal basis can be taken as
\begin{align*}
 Y^n_1(\bx) = \frac{1}{\sqrt{8\pi}}P^{(0,0)}_{n} (\cos \theta ), \quad  &Y^n_{2k}(\bx) = \frac{ r^n}{2^{k+1} \sqrt{ \pi}} (\sin\theta)^{k}  P^{(k,k)}_{n-k} (\cos \theta ) \cos k \phi, \quad   1\le k \le n,
 \\
 &Y^n_{2k+1}(\bx) =  \frac{ r^n}{2^{k+1} \sqrt{ \pi}}  (\sin\theta)^{k}  P^{(k,k)}_{n-k} (\cos \theta ) \sin k \phi, \quad 1\le  k \le  n.
\end{align*}
\end{itemize}
\end{rem}

The spherical harmonics satisfy  the following explicit integral relation.
\begin{lemma}[{\cite[Lemma 9.10.2]{Askey}}]\label{lm:Ffx} For any $\bs{\hat x},  \bs{\hat \xi} \in \sph^{d-1}$ and $w> 0$, we have
\begin{equation}
\label{eq:Ffx}
\int_{\sph^{d-1}}  \e^{-\i w \langle\bs{\hat \xi}, \bs{\hat x\rangle}} Y^n_{\ell}(\bs{\hat x}) \d\s(\bs{\hat x})
 = \frac{ {(2\pi)}^{\frac{d}{2}} (-\i)^n }{{w}^{\frac{d-2}2}}  J_{n+\frac{d-2}2}(w)   Y_{\ell}^{n}(\bs{\hat \xi}).
\end{equation}
\end{lemma}

For any  function $f\in L^2(\RR^d),$ we expand it in spherical harmonic series:
\begin{equation}\label{fbxdefn}
   f(\bx) =\sum_{n=0}^{\infty} \sum_{\ell=1}^{a_n^d}   f^n_{\ell}(r)   Y^n_{\ell}(\bs {\hat x}),\quad f_{\ell}^n(r) = \int_{\sph^{d-1}} f(r \hat \bx) Y^n_{\ell}(\hat \bx) \d\s(\hat \bx).
\end{equation}
Then its Fourier transform
\begin{align*}
\mathscr{F}[f](\bxi):=\int_{\mathbb{R}^d} f(\bx)\e^{-\rm{i}\langle \bxi,  \bx \rangle}\d{\bx},
\end{align*}
 can be represented in spherical harmonic series with the coefficients being  the Hankel  transform of its original spherical harmonic coefficients.
\begin{thm}\label{th:Ffx}
For any  function $f(\bs x)\in L^2(\RR^d),$
we have
\begin{equation}\label{Fourierf}
\begin{split}
\mathscr{F}[f](\bxi)
= \sum_{n=0}^{\infty}  \frac{ {(2\pi)}^{\frac{d}{2}} (-\i)^n }{{\rho}^{\frac{d-2}2}}  \sum_{\ell=1}^{a_n^d}  Y_{\ell}^{n} (\bs{\hat \xi}) \mathscr{H}_{n+\frac{d-2}{2}}^d [f^n_{\ell}](\rho), \quad
\end{split}
\end{equation}
where $\bxi = \rho \hat \bxi,\, \hat\bxi \in \sph^{d-1},\, \rho\ge 0,$ and the Hankel transform is defined by
\begin{equation}\label{HankelTrans}
\mathscr{H}_{\nu}^d[f](\rho)\equiv\int_0^\infty J_{\nu}(\rho r) f(r)r^{\frac{d}{2}}\d{r},\quad \rho\ge 0,\,\nu> -\frac{1}{2},\;r>0.
\end{equation}
\end{thm}
\begin{proof}
Denote by $(r,\bs{\hat x})$ and $(\rho, \bs {\hat\xi})$ the polar-spherical coordinates of $\bs x$ and $\bs \xi$, respectively.
Then applying the Fourier transform to the series \eqref{fbxdefn}, we obtain
\begin{equation*}
\begin{split}
\mathscr{F}[f](\bs \xi)&=  \int_{\RR^d} f(\bx) \e^{-\i \langle \bs \xi,  \bx\rangle} \d{x} = \sum_{n=0}^{\infty} \sum_{\ell=1}^{a_n^d}    \int_{0}^{\infty} f^n_{\ell}(r) r^{d-1} \d r \int_{\sph^{d-1}}
  Y^n_{\ell}(\bs {\hat x})  \e^{-\i \rho r   \langle \bs{\hat\xi},  \bs{\hat x}\rangle}  \d\s(\hat \bx).
\end{split}
\end{equation*}
Further,  using Lemma \ref{lm:Ffx} leads to
\begin{equation*}
\begin{split}
\mathscr{F}[f](\bs \xi)&= \sum_{n=0}^{\infty}  \sum_{\ell=1}^{a_n^d}   \int_{0}^{\infty} f^n_{\ell}(r) r^{d-1} \d r
\frac{ {(2\pi)}^{\frac{d}{2}}(- \i)^n }{({\rho r})^{\frac{d-2}2}}   J_{n+\frac{d-2}2}(\rho r)
 Y_{\ell}^{n}(\hat \bxi)
   \\  &
= \sum_{n=0}^{\infty}  \frac{ {(2\pi)}^{\frac{d}{2}} (-\i)^n }{\rho^{\frac{d-2}2}}  \sum_{\ell=1}^{a_n^d}  Y_{\ell}^{n} (\hat \bxi)\int_{0}^{\infty} f^n_{\ell}(r) J_{n+\frac{d-2}2}(\rho r)   r^{\frac{d}{2}} \d r \\
 &
= \sum_{n=0}^{\infty}  \frac{ {(2\pi)}^{\frac{d}{2}} (-\i)^n }{\rho^{\frac{d-2}2}}  \sum_{\ell=1}^{a_n^d}  Y_{\ell}^{n} (\hat \bxi)\mathscr{H}_{n+\frac{d-2}{2}}^d [f^n_{\ell}](\rho).
  \end{split}
  \end{equation*}
  This ends the proof.
\end{proof}
\subsection{Ball polynomials: \!orthogonal polynomials on $\mathbb B^d$}\label{secP}
For any $\alpha>-1$, we define the ball polynomials as
\begin{equation}\label{orthnPkl}
P_{k,\ell}^{\af,n}(\bx)={P}_{k}^{(\af,n+\frac{d}2-1)}(2\|\bx\|^2-1) Y_{\ell}^{n}(\bx), \quad \bx\in \ball^d, \;\;  \ell \in \Upsilon_n^d,\;\; k,n\in {\mathbb N}_0.
\end{equation}
Note that the total degree of  $P_{k,\ell}^{\af,n}(\bs x)$ is $n+2k$ for any $\ell\in \Upsilon_n^d$.
The ball polynomials are mutually orthogonal
with respect to the weight function $\varpi_{\af}(\bx):=(1-\|\bx\|^2)^{\af}$ (cf. {\cite[Propostion 11.1.13]{Dai2013}}):
\begin{equation}
\label{orthP}
(P_{k,\ell}^{\af,n},P_{j,\iota}^{\af,m} )_{\varpi_{\af}} =\delta_{nm}\delta_{kj}\delta_{\ell \iota}, \quad \ell \in \Upsilon_n^d,\;\; \iota \in \Upsilon_m^d,\;\; k,j,m,n\in {\mathbb N}_0,
\end{equation}
where the inner product $(\cdot,\cdot)_{\varpi_{\af}} $ is defined by
\begin{equation*}
(f,g)_{\varpi_{\af}}:=\int_{\ball^d}f(\bx)g(\bx)\varpi_{\af}(\bx)\,\d\bx.
\end{equation*}


\begin{lemma}[{\cite[Theorem 11.1.5]{Dai2013}}]\label{dxbop} The ball orthogonal polynomials 
are the eigenfunctions of the differential operator:
\begin{equation}
\label{Ldef}
\mathscr{L}_{\bx}^{(\af)}P_{k,\ell}^{\af,n}(\bs x):=\left(-\Delta+\nabla \cdot \bx(2\af+\bx\cdot\nabla)-2\af d\right) P_{k,\ell}^{\af,n}(\bs x)=\gamma_{n+2k}^{(\af)}P_{k,\ell}^{\af,n}(\bs x),
\end{equation}
where $\gamma_m^{(\af)}:=m(m+2\af+d).$
\end{lemma}

The Sturm-Liouville operator $\mathscr{L}_{\bs x}^{(\af)}$ takes different
forms, which find more appropriate for the forthcoming derivations.
\begin{thm}\label{thmdx}
For $\alpha>-1$, it holds that
\begin{align}
\label{Ldef2}
\mathscr{L}_{\bx}^{(\af)} =& -(1-\|\bx\|^2)^{-\af}\nabla\cdot({\bs {\rm I}}-\bx\bx^{\tr})(1-\|\bx\|^2)^{\af}\nabla
\\
\label{Ldef3}
=& -(1-\|\bx\|^2)^{-\af} \nabla\cdot  (1-\|\bx\|^2)^{\af+1}\nabla - \Delta_0
\\
\label{Ldefr}
=&-(1-r^2)\partial^2r-\frac{d-1}{r}\partial r+(2\af+d+1)r\partial r-\frac{1}{r^2} \Delta_0,
\end{align}
where $\Delta_0$ is the spherical part of $\Delta$ and involves only derivatives in $\hat \bx.$
\end{thm}
\begin{proof} Using the Leibniz rule for  gradient and divergence, one derives
\begin{equation}\label{Lxaf}
\begin{split}
 -(1-&\|\bx\|^2)^{-\af}\nabla\cdot(\bs{\rm{I}}-\bx\bx^{\tr})(1-\|\bx\|^2)^{\af}\nabla\\
&=-(1-\|\bx\|^2)^{-\af} \left[ (1-\|\bx\|^2)^{\af} \nabla\cdot  (\bs{\rm I}-\bx\bx^{\tr})\nabla-2\af (1-\|\bx\|^2)^{\af-1} \bx^{\tr}   (\bs{\rm I}-\bx\bx^{\tr})\nabla  \right]\\
&= -\nabla\cdot(\bs{\rm I}-\bx\bx^{\tr})\nabla+2\af   \bx\cdot  \nabla
= -\nabla\cdot(\bs{\rm I}-\bx\bx^{\tr})\nabla + 2\af (\nabla\cdot \bx - d)\\
&=-\Delta+\nabla \cdot \bx(2\af+\bx\cdot\nabla)-2\af d,
\end{split}
\end{equation}
which exactly gives \eqref{Ldef2}.

Next, a component by component reduction yields
\begin{align*}
  -(1-&\|\bx\|^2)^{-\af}\nabla\cdot({\bs {\rm I}}-\bx\bx^{\tr})(1-\|\bx\|^2)^{\af}\nabla
  \\
  &= -(1-\|\bx\|^2)^{-\af}\sum_{1\le i\le d} \partial_{x_i}
  \Big[       (1-x_i^2) (1-\|\bx\|^2)^{\af} \partial_{x_i}    - \sum_{1\le j\neq i\le d}  x_ix_j (1-\|\bx\|^2)^{\af}   \partial_{x_j} \Big]
  \\
  &= -(1-\|\bx\|^2)^{-\af} \Big[\sum_{1\le i\le d} \partial_{x_i}
   (1-\|\bx\|^2)^{\af+1} \partial_{x_i}     +\sum_{1\le i\le d}   \sum_{1\le j \neq i\le d}x_j  \partial_{x_i}  (1-\|\bx\|^2)^{\af}    D_{ij}\Big]
  \\
  &= -(1-\|\bx\|^2)^{-\af} \Big[\sum_{1\le i\le d} \partial_{x_i}
   (1-\|\bx\|^2)^{\af+1} \partial_{x_i}     +   \sum_{1\le j < i\le d}  D_{ij} (1-\|\bx\|^2)^{\af}    D_{ij}\Big]
   \\
  & = -(1-\|\bx\|^2)^{-\af} \nabla  \cdot (1-\|\bx\|^2)^{\af+1} \nabla -  \Delta_0,
\end{align*}
where the commutativity
of $D_{ij}$ and $r$ is used in the last step. This verifies \eqref{Ldef3}.

Finally, applying the Leibniz rule once again, one gets
\begin{align*}
\mathscr{L}_{\bx}^{(\af)}  =& -(1-\|\bx\|^2)^{-\af} \Big[  (1-\|\bx\|^2)^{\af+1}  \nabla  \cdot    \nabla   -  2(\alpha+1) (1-\|\bx\|^2)^{\alpha} \bx\cdot\nabla  \Big]  -  \Delta_0
\\
=& -(1-\|x\|^2) \Delta + 2 (\alpha+1) \bx\cdot  \nabla   - \Delta_0
\\
=& -(1-r^2) \Big[ \partial_r^2 + \frac{d-1}{r}\partial_r  +\frac{1}{r^2} \Delta_0 \Big] + 2(\alpha+1) r\partial_r  - \Delta_0
\\
=& -(1-r^2)\partial^2r-\frac{d-1}{r}\partial r+(2\af+d+1)r\partial r-\frac{1}{r^2} \Delta_0,
\end{align*}
where we used the \eqref{eq:Delta} and identity $\bx\cdot  \nabla = r \bs{\hat x} \cdot \nabla =r\partial_r$.
\end{proof}

\def\I{\mathrm{I}}

Thanks to \eqref{eq:LaplaceBeltrami}, we use the form \eqref{Ldefr} of the operator $\mathscr{L}_{\bx}^{(\af)},$ and derive that in $r$-direction,
\begin{equation}\label{Opr}
\begin{split}
\mathscr{L}_{r}^{(\af)}
\big(r^n   P_{k}^{\alpha,n+\frac{d}2-1}(2r^2-1) \big)  =\gamma_{n+2k}^{(\af)}\big(r^nP_{k}^{\alpha,n+\frac{d}2-1}(2r^2-1)  \big),
\end{split}
\end{equation}
where we denote
\begin{equation}\label{LrLrA}
\mathscr{L}_{r}^{(\af)}:=-(1-r^2)\partial_r^2-\frac{d-1}{r}\partial_r+(2\af+d+1)r\partial_r+\frac{n(n+d-2)}{r^2}.
\end{equation}
With a change of variable
 $\eta=2r^2-1$ and denoting  $\beta_n=n+{d}/{2}-1$, we can rewrite \eqref{Opr} as
\begin{equation}\label{jacobieta}
\begin{split}
\mathscr{L}_{\eta}^{(\af,\beta_n)} P_{n}^{(\af,\beta_n)}(\eta) &=-\frac{1}{\omega_{\af,\beta_n}(\eta)}
\partial_{\eta}\big(\omega_{\af+1,\beta_n+1}(\eta)\partial_{\eta}P_{k}^{(\af,\beta_n)}(\eta)\big)
\\
&=\frac{1}{4}(\gamma_{n+2k}^{(\af)}-\gamma_{n}^{(\af)}) P_{k}^{\alpha,\beta_n}(\eta)  =\lambda_k^{(\af,\beta_n)}P_k^{(\af,\beta_n)}(\eta),\quad  \eta\in (-1,1),
\end{split}
\end{equation}
which is exactly  \eqref{JacobiSL}. This indicates a close relation between the $r$-component of a ball polynomial and Jacobi polynomials in $x\in (-1,1)$ with parameter varying with $n.$

\section{Ball PSWFs as  eigenfunctions of a Sturm-Liouville operator}\label{secG}

The PSWFs to be introduced can be defined as eigenfunctions of a differential operator or an integral operator. In this section,  we focus on the former approach, and present some important properties from this perspective.


\subsection{Definition of ball PSWFs on $\ball^d$}  For $\alpha>-1,$ we define the second-order differential operator:
\begin{equation}\label{opDxB}
\mathscr{D}_{c,\bx}^{(\af)}:=\mathscr{L}_{\bx}^{(\af)}+ c^2 \|\bx\|^2=-(1-\|\bx\|^2)^{-\af}\nabla\cdot({\bs {\rm I}}-\bx\bx^{\tr})(1-\|\bx\|^2)^{\af}\nabla+ c^2 \|\bx\|^2,
\end{equation}
for $\bx\in \ball^d,$ and real $c\ge 0,$ where the operator $\mathscr{L}_{\bx}^{(\af)}$ is defined  in Lemma \ref{dxbop} with various equivalent forms stated in Theorem \ref{thmdx}.
It is clear that  $\mathscr{D}_{c,\bx}^{(\af)}$ is  a strictly positive self-adjoint operator in the sense that for any $u,v$ in the domain of $\mathscr{D}_{c,\bx}^{(\af)},$ we have
\begin{equation}\label{Dself1}
\big(\mathscr{D}_{c,\bx}^{(\af)} u, v\big)_{\varpi_\af}= \big(u, \mathscr{D}_{c,\bx}^{(\af)}  v\big)_{\varpi_\af},
\end{equation}
and for all $u\not=0,$
\begin{align}\label{Dself}
\big(\mathscr{D}_{c,\bx}^{(\af)} u, u\big)_{\varpi_\af}=\|\nabla u\|_{\varpi_{\af+1}}^2
+ \sum_{1\le i<j\le d} \|D_{ij}u\|^2_{\varpi_\af}   +c^2( \|u\|^2_{\varpi_\af}
-  \|u\|^2_{\varpi_{\af+1}}) >0.
\end{align}
Hence, by the  Sturm-Louville theory (cf.  \cite{Al-Gwaiz07,Codd55}), the operator  $\mathscr{D}_{c,\bx}^{(\af)}$ admits  a countable and infinite set of bounded, analytical eigenfunctions $\{\psi(\bx)\}$ which forms a  complete orthogonal system of $L^2_{\varpi_\af}(\ball^d).$
%
%
In other words, we have
\begin{equation}\label{varphichi}
\mathscr{D}_{c,\bx}^{(\af)}[\psi](\bx)=\chi\, \psi(\bx),\quad  \bx\in \ball^d,
\end{equation}
where $\{\chi:=\chi(c)\}$ are the corresponding eigenvalues.

In view of \eqref{Ldefr}, we can rewrite the operator $\mathscr{D}_{c,\bx}^{(\af)}$ in the spherical-polar coordinates as
\begin{equation*}
\mathscr{D}_{c,\bx}^{(\af)}=\mathscr{L}_{\bs x}^{(\af)}+c^2r^2 = -(1-r^2)\partial^2_r-\frac{d-1}{r}\partial_r+(2\af+d+1)r\partial_r-\frac{1}{r^2} \Delta_0+c^2r^2.
\end{equation*}
We infer from \eqref{orthnPkl} and Lemma \ref{dxbop} that the eigenfunction in \eqref{varphichi} takes the form:
\begin{equation}\label{psi}
\psi(\bx)=r^n \phi_{k}^{\alpha,n}(2r^2-1; c) Y^n_{\ell}(\hat\bx),\quad  \ell\in \Upsilon_n^d,\;\; k,n\in {\mathbb N}.
\end{equation}
In analogy to \eqref{Opr}-\eqref{LrLrA}, the eigen-value problem \eqref{varphichi} in $r$-direction
takes the  equivalent form:
\begin{equation}\label{Oprpsi}
\big(\mathscr{L}_{r}^{(\af)}+c^2r^2\big)
\big(r^n   \phi_{k}^{\alpha,n}(2r^2-1; c) \big)  =\chi_{n,k}^{(\af)}(c) \big(r^n   \phi_{k}^{\alpha,n}(2r^2-1; c) \big).
\end{equation}
Similar to \eqref{jacobieta}, we make  a change of variable $\eta=2r^2-1,$ and find from the above  that
\begin{equation}\label{phieig}
\mathscr{D}_{c,\eta}^{(\af)}\phi_k^{\af,n}(\eta; c)=\frac{1}{4}\big(\chi_{n,k}^{(\af)} (c)-\gamma_{n}^{(\af)}\big)\phi_k^{\af,n}(\eta; c),
\end{equation}
where $\mathscr{D}_{c,\eta}^{(\af)}$ is the second-order differential operator:
\begin{equation}\label{newMbopt}
\mathscr{D}_{c,\eta}^{(\af)}:=\mathscr{L}_{\eta}^{(\af,\beta_n)} +\frac{c^2(\eta+1)}{8}=-\frac{1}{\omega_{\af,\beta_n}(\eta)}
\partial_{\eta}\big(\omega_{\af+1,\beta_n+1}(\eta)\partial_{\eta}\cdot\big)+\frac{c^2(\eta+1)}{8},
\end{equation}
with  $\af>-1, \beta_n=n+d/2-1, \;\eta\in I. $ Note  that $\mathscr{D}_{c,\eta}^{(\af)}$ is  a  symmetric and strictly positive operator. 
According to the general theory of  Sturm-Liouville problems  (cf.  \cite{Al-Gwaiz07,Codd55}),
$\big\{\phi_k^{\af,n}(\eta; c)\big\}_{k=0}^\infty$ forms a  complete orthogonal
system of $L^2_{\omega_{\af,\beta_n}}(I).$
In view of \eqref{psi} and \eqref{phieig}, we can define the  PSWFs of interest as follows.
\vskip 3pt

\begin{defn}\label{BPSWFs} {\bf (Ball PSWFs on ${\mathbb B}^d$).} For real $\alpha>-1$ and real $c\ge 0,$
the prolate spheroidal wave functions
on a $d$-dimensional unit  ball $\ball^d,$ denoted by $\big\{\swf_{k,\ell}(\bx; c)\big\}_{\ell\in \Upsilon_n^d}^{k,n\in {\mathbb N}},$ are  eigenfunctions of the differential operator defined in  $\mathscr{D}_{c,\bx}^{(\af)}$ defined in \eqref{opDxB}, that is,
\begin{equation}\label{varphichi2}
\mathscr{D}_{c,\bx}^{(\af)}\swf_{k,\ell}(\bx; c)=\chi_{n,k}^{(\af)}(c)\, \swf_{k,\ell}(\bx; c),\quad  \bx\in \ball^d,
\end{equation}
where $\big\{\chi_{n,k}^{(\af)}(c)\big\}_{\ell\in \Upsilon_n^d}^{k,n\in {\mathbb N}}$ are the corresponding eigen-values, and $c$ is the bandwidth parameter.
\end{defn}

 We summarize two points  in order.
  In  the spherical-polar coordinates, $\swf_{k,\ell}(\bx;c)$ has a  separated form given by \eqref{psi}, i.e.,
     \begin{equation}\label{psinew}
\swf_{k,\ell}(\bx;c)=r^n \phi_{k}^{\alpha,n}(2r^2-1; c) Y^n_{\ell}(\hat\bx),\quad  \ell\in \Upsilon_n^d,\;\; k,n\in {\mathbb N},
\end{equation}
where $\phi_{k}^{\alpha,n}(\cdot; c)$ satisfies \eqref{Oprpsi}-\eqref{phieig}. On the other hand, if $c=0,$ we find readily from the previous discussions that
\begin{equation}\label{gtog}
\swf_{k,\ell}(\bx; 0)=P_{k,\ell}^{\af,n}(\bx), \quad \phi_{k}^{\alpha,n}(\eta; 0)=P_{k}^{(\alpha,\beta_n)}(\eta),\quad \chi_{n,k}^{(\af)}(0)=\gamma_{n+2k}^{(\af)}.
\end{equation}
 Thus,  the ball PSWF $\swf_{k,\ell}(\bx; c)$ on $\ball^d$ can be viewed as a generalization of the ball polynomial  $P_{k,\ell}^{\af,n}(\bx)$ (cf. Subsection \ref{secP}) with a tuning parameter $c$.



\subsection{Important properties} We present below  some basic properties of  $\swf_{k,\ell}(\bx; c)$
that follows from the Sturm-Louville theory (cf. \cite{Al-Gwaiz07,Codd55}).
\begin{thm}\label{GpswfProp}  For any $c>0$ and $\af>-1$,
\begin{itemize}
\item[(i)] $\big\{\swf_{k,\ell}(\bx; c)\big\}_{\ell\in \Upsilon_n^d}^{k,n\in {\mathbb N}}$ are all real, smooth, and form a complete orthonormal system of
$L^2_{\varpi_{\af}}(\ball^d), $ namely,
\begin{equation}\label{orthnswf}
\int_{\ball^d} \swf_{k,\ell}(\bx; c) \psi^{\af,m}_{j,\iota}(\bx; c)\varpi_{\af} (\bx)\d\bx=\delta_{k,j}\delta_{\ell ,\iota}\delta_{n,m}\,.
\end{equation}

\item[(ii)] $\big\{\chi_{n,k}^{(\af)}(c)\big\}_{k,n\in {\mathbb N}}$ are all real, positive, and
ordered  for fixed $n$ as follows
\begin{equation}\label{eqn_incr}
0<\chi_{n,0}^{(\af)}(c)<\chi_{n,1}^{(\af)}(c)<\cdots<
\chi_{n,k}^{(\af)}(c)<\cdots.
\end{equation}
\item[(iii)] $\big\{\swf_{k,\ell}(\bx; c)\big\}_{\ell\in \Upsilon_n^d}^{k,n\in {\mathbb N}}$  with even $n$ are even functions of $\bx,$ and those with
odd $n$ are odd, namely,
\begin{equation}\label{parity}
\psi_{k,{\ell}}^{\af,n}(-\bx; c)=(-1)^n \psi_{k,{\ell}}^{\af,n}(\bx; c),\quad \forall \bs \;\bx\in \ball^d.
\end{equation}
%
\end{itemize}
\end{thm}
We have the following bounds for the eigen-values $\big\{\chi_{n,k}^{(\af)}(c)\big\}_{k,n\in {\mathbb N}}$.
\begin{thm}
 \label{lm:chi}
 For any $\af>-1$ and $c>0$,
\begin{equation}\label{lambda}
(n+2k)(n+2k+2\alpha+d)<\chi_{n,k}^{(\alpha)}(c)<(n+2k)(n+2k+2\alpha+d)+c^2, \quad
n\geq{0}.
\end{equation}
\end{thm}
\begin{proof}
Differentiating the equation \eqref{varphichi2}  with
respect to $c$ yields
\begin{equation*}
\big[  \mathscr{D}_{c,\bx}^{(\af)} -\chi_{n,k}^{(\af)}(c)\big]  \big(\partial_{c}\swf_{k,\ell}(\bx;c)\big) = \big(\partial_{c}\chi_{n,k}^{(\af)}(c)-2c\|\bx\|^2\big)\swf_{k,\ell}(\bx;c).
\end{equation*}
Taking the inner product with $\swf_{k,\ell}$ with respect to $\varpi_{\af}$,
and using \eqref{Dself} and \eqref{varphichi2}, we derive
\begin{equation*}
\begin{split}
 \partial_c\chi_{n,k}^{(\af)}(c)-&  2c\int_{\ball^d}[\swf_{k,\ell}(\bx;c)]^2 \|\bx\|^2 \varpi_{\af}(\bx)\d\bx
= \big( \big[  \mathscr{D}_{c,\bx}^{(\af)} -\chi_{n,k}^{(\af)}(c)\big]  \partial_{c}\swf_{k,\ell},   \swf_{k,\ell}\big)_{\varpi_{\af}}
\\
=&\,\big(  \partial_{c}\swf_{k,\ell},  \big[  \mathscr{D}_{c,\bx}^{(\af)} -\chi_{n,k}^{(\af)}(c)\big]  \swf_{k,\ell}\big)_{\varpi_{\af}} = 0.
\end{split}
\end{equation*}
As a result,
$$0<\partial_c\chi_{n,k}^{(\af)}(c)=2c\int_{\ball^d}[\swf_{k,\ell}(\bx;c)]^2 \|\bx\|^2 \varpi_{\alpha}(\bx)\d\bx < 2c\int_{\ball^d}[\swf_{k,\ell}(\bx;c)]^2  \varpi_{\alpha}(\bx)\d\bx=2c,$$
which implies
$$
0<\chi_{n,k}^{(\af)}(c)-\chi_{n,k}^{(\af)}(0) =\chi_{n,k}^{(\af)}(c) - (n+2k)(n+2k+2\af+d) <c^2.$$
This ends the proof.
\end{proof}

For   $0<c\ll 1,$  the PSWF $\swf_{k,\ell}(\bx;c)$ is a small perturbation of the ball
polynomial $P_{k,\ell}^{\af,n}(\bx).$
\begin{thm}\label{cpertu} For $0<c\ll 1,$ we have
\begin{equation*}
\swf_{k,\ell}(\bx;c)=P_{k,\ell}^{\af,n}(\bx)+O(c^2),\quad  \chi^{(\af)}_{n,k}(c)=\gamma_{2n+k}^{(\af)}+O(c^2),\quad k,n\in {\mathbb N}.
\end{equation*}
\end{thm}
\begin{proof}
Following the  perturbation scheme in \cite{Slep64}, we expand the eigen-pair
$\big\{\chi_{n,k}^{(\af)}(c),\phi^{\af,n}_{k}(\eta;c)\big\}$ in series of $c^2:$
\begin{equation}\label{pertexp}
\begin{split}
&\phi^{\af,n}_{k}(\eta;c)=P_{k}^{(\af,\beta_n)}(\eta)+\sum_{j=1}^\infty c^{2j}Q^{\af,n}_{k,j}(\eta); \quad \chi_{n,k}^{(\af)}(c)=\gamma_{2n+k}^{(\af)}+\sum_{j=1}^\infty c^{2j}d_{k,j}^{\af,n},
\end{split}
\end{equation}
where $\gamma_{2n+k}^{(\af)}=\chi_{n,k}^{(\af)}(0)$ (cf. \eqref{gtog}), and
\begin{equation}\label{Q}
\begin{split}
&Q_{k,j}^{\af,n}(\eta)=\sum_{h=-j}^{j}B_{h,k}^{\af,n}(j)P_{k+h}^{(\af,\beta_n)}(\eta),
\end{split}
\end{equation}
with the convectional choice $B_{0,k}^{\af,n}=0.$
Hence, substituting the expansion \re{pertexp} into the eigen-equation \eqref{phieig},
and  equating to zero the coefficients of distinct powers of
$c^2,$ we find the equation corresponding to the coefficient of $c^2$ is
\begin{equation*}
\big( 8 \mathscr{L}_{\eta}^{(\af,\beta_n)}-2\gamma_{n+2k}^{(\af)}+2\gamma_{n}^{(\af)}\big)Q_{k,1}^{\af,n}(\eta)+\big(\eta+1-2d_{k,1}^{\af,n}\big)P_{k}^{(\af,\beta_n)}(\eta)=0.
\end{equation*}
Hence, using the expansion \eqref{Q},  the eigen equation \eqref{JacobiSL},  and the three-term recurrence \eqref{Jacobi},
we find
\begin{equation*}
\begin{split}
\big[8(\lambda_{k+1}^{(\af,\beta_n)} -&\lambda_{k}^{(\af,\beta_n)} ) B_{1,k}^{\af,n}+ a_k^{(\af,\beta_n)}\big]  P_{k+1}^{(\af,\beta_n)}
+\big[8(\lambda_{k-1}^{(\af,\beta_n)} -\lambda_{k}^{(\af,\beta_n)}) B_{-1,k}^{\af,n} + a_{k-1}^{(\af,\beta_n)}\big] P_{k-1}^{(\af,\beta_n)}
\\
+& (b_k^{(\af,\beta_n)}+1-2d_{k,1}^{\af,n}\big)P_{k}^{(\af,\beta_n)} = 0,
\end{split}
\end{equation*}
which implies
\begin{equation*}
d_{k,1}^{\af,n}=\frac{b_k^{(\af,\beta_n)}+1}{2},  \quad
 B_{1,k}^{\af,n}=- \frac{a_k^{(\af,\beta_n)} }{8(2k+\af+\beta_n+2)  },\quad
  B_{-1,k}^{\af,n}=\frac{a_{k-1}^{(\af,\beta_n)} }{8(2k+\af+\beta_n)  } =-B_{1,k-1}^{\af,n} .
\end{equation*}
Thus  we obtain  
\begin{equation*}
\chi_{n,k}^{(\af)}(c)=\gamma_{2n+k}^{(\af)}+c^2 d_{k,1}^{\af,n}+O(c^4).
\end{equation*}
and
\begin{align*}
&\phi^{\af,n}_{k}(\eta;c) = P_{k}^{(\af,\beta_n)}(\eta) +c^2\big(B_{-1,k}^{\af,n} P_{k-1}^{(\af,\beta_n)}(\eta)+B_{1,k}^{\af,n} P_{k+1}^{(\af,\beta_n)}(\eta) \big)+O(c^4),
\\
&\swf_{k,\ell}(\bx;c)=P_{k,\ell}^{\af,n}(\bx)+c^2\big(B_{-1,k}^{\af,n}  P_{k-1,\ell}^{\af,n}(\bx)+B_{1,k}^{\af,n} P_{k+1,\ell}^{\af,n}(\bx)\big)+O(c^4).
\end{align*}
This ends the proof.
\end{proof}

\section{Ball PSWFs  as  eigenfunctions of  finite Fourier transform}\label{sect4}
In this section, we show that the ball PSWFs  are  eigenfunctions of  a compact (finite) Fourier integral operator.

Define the (weighted) finite Fourier integral operator
 ${\mathscr F}_{c}^{(\af)}:{L^2_{\varpi_{\alpha}}(\ball^d)}\rightarrow
L^2_{\varpi_{\alpha}}(\ball^d)$ by
\begin{equation}\label{bdlim}
{\mathscr F}_{c}^{(\af)}[\phi](\bx)=\int_{\ball^d}\e^{-\ri
c \langle \bx, \bs \tau\rangle}\phi(\bs \tau)\varpi_{\alpha}(\bs \tau)\d\bs \tau,\quad \bx\in\ball^d,\;\; c>0,\;\af>-1,
\end{equation}
where $\varpi_{\af}(\bx)=(1-\|\bx\|^2)^{\af}$  as before.
Note that for $\af= 0$, ${\mathscr F}_{c}^{(\af)}$ is reduced to the finite Fourier transform on the  ball.
From Theorem \ref{th:Ffx}, we have that for $ \bx = r \hat \bx$ with $\hat\bx \in \sph^{d-1},$
\begin{align*}
&{\mathscr F}_{c}^{(\af)}[\phi](\bx)= \sum_{n=0}^{\infty}  \frac{ (2\pi)^{\frac{d}{2}} (-\i)^n }{\rho^{\frac{d-2}2}}  \sum_{\ell=1}^{a_n^d}  Y_{\ell}^{n} (\hat \bx) \widehat {\mathscr{H} } _{n+\frac{d-2}{2}}^d[\phi^n_{\ell}](r),\quad
\end{align*}
where spherical coefficient $\phi_{\ell}^n(r)$ and the finite Hankel transform $\widehat {\mathscr{H} } _{\nu}^d$ are 
\begin{align*}
&\phi_{\ell}^n(r) = \int_{\sph^{d-1}} f(\rho \hat {\bs \tau}) Y^n_{\ell}(\hat {\bs \tau}) \d\s(\hat {\bs \tau}),\quad
\widehat {\mathscr{H} } _{\nu}^d[f](\rho)\equiv\int_0^\infty J_{\nu}(\rho r) f(r)r^{\frac{d}{2}}\d{r},\quad
\end{align*}
for $\rho\ge 0,\,\nu> -\frac{1}{2}$ and $r>0.$

We  introduce an associated integral operator
${\mathcal Q}_c^{(\af)}: {L^2_{\varpi_{\alpha}}(\ball^d)}\rightarrow
L^2_{\varpi_{\alpha}}(\ball^d),$ defined by
\begin{equation}\label{OpQff}
{\mathcal Q}_c^{(\af)}=({{\mathscr F}_c^{(\alpha)}})^{*}\circ {\mathscr F}_c^{(\alpha)},\quad
 c>0,\;\af>-1.
\end{equation}
\begin{thm}
\label{th:Q}
Let $c>0,\af>-1$ and $\phi\in L_{\varpi_{\af}}^2(\ball^d).$ Then we have
\begin{equation}\label{OpQ}
{\mathcal Q}_c^{(\af)}\big[\phi\big](\bx)=\int_{\ball^d} {\mathcal K}_c^{(\af)}(\bx,\bs \tau)\phi(\bs \tau){\varpi_{\af}}(\bs \tau)\d\bs \tau,\quad
\bx\in \ball^d,
\end{equation}
where
\begin{equation}\label{propOpQ}
\begin{split}
{\mathcal K}_c^{(\af)}(\bx,\bs t):&=(2\pi)^{\frac{d}{2}} \frac{\widehat {\mathscr{H}}_{\frac{d-2}{2}}^d
[\omega_{\af,\af}](c\|\bs \tau-\bx\|)}{(c\|\bs \tau-\bx\|)^{\frac{d-2}{2}}}\\
&=\frac{(2\pi)^{\frac{d}{2}}}{(c\|\bs \tau-\bx\|)^{\frac{d-2}{2}}}\int_{0}^1s^{\frac{d}{2}}(1-s^2)^{\af}J_{\frac{d-2}{2}}(cs\|\bs \tau-\bx\|) \d{s}.
\end{split}
\end{equation}
\end{thm}
\begin{proof} By \eqref{bdlim}, we have
\begin{equation}\label{cal1}
\begin{split}
\big(({{\mathscr F}_c^{(\alpha)}})^{*}\circ {\mathscr F}_c^{(\alpha)}\big)\big[\phi\big](\bx)&
=\displaystyle\int_{\ball^d}{\mathcal K}_c^{(\af)}(\bx,\bs \tau)\phi(\bs \tau) \varpi_\alpha(\bs \tau)\d\bs \tau,
\end{split}
\end{equation}
where
\begin{equation*}
{\mathcal K}_c^{(\af)}(\bx,\bs \tau)=\int_{\ball^d}e^{\ri c\langle\bx-\bs\tau,\bs s\rangle} \varpi_\alpha(\bs s) \d\bs s.
\end{equation*}
Using the spherical-polar coordinates $\bs s=s{\bs{ \hat{s}}}, \,  {\bs{ \hat{s}}}\in \sph^{d-1}, s\ge 0$, we derive from \eqref{eq:Ffx} that
\begin{equation*}
\begin{split}
\int_{\ball^d}\e^{\ri c\langle\bx-\bs\tau,\bs s\rangle} \varpi_\alpha(\bs s) \d\bs s
&=\int_{0}^1s^{d-1}(1-s^2)^{\af}\d s \int_{\mathbb{S}^{d-1}}\e^{\ri cs\langle \bx-\bs \tau,\hat {\bs s}\rangle}  \d \sigma (\hat {\bs s})\\
&=\frac{(2\pi)^{\frac{d}{2}}}{(c\|\bs \tau-\bx\|)^{\frac{d-2}{2}}}\int_{0}^1s^{\frac{d}{2}}(1-s^2)^{\af}J_{\frac{d-2}{2}}(cs\|\bs \tau-\bx\|) \d{s}.
\end{split}
\end{equation*}
This ends the proof.
\end{proof}
The following theorem indicates  that the ball PSWFs are  eigenfunctions of both  ${\mathscr F}_c^{(\alpha)}$ and
${\mathcal Q}_c^{(\alpha)}.$

\begin{thm}\label{th:lam} For  $\alpha>-1$ and $c>0,$  the  ball PSWFs are the eigenfunctions of
${\mathscr F}_c^{(\alpha)}:$ 
\begin{equation}\label{eigen1}
{\mathscr F}_c^{(\alpha)}[\swf_{k,{\ell}}](\bx;c)={(-\ri)}^{n+2k}\lambda_{n,k}^{(\af)}(c)\, \swf_{k,{\ell}}(\bx;c),\quad
\bx\in \ball^d,
\end{equation}
and the eigenvalues
 $\big\{\lambda_{n,k}^{(\af)}(c)\big\}_{k,n\in {\mathbb N}}$  are all real and can be arranged for fixed $n$ as
\begin{equation}\label{orderth}
\lambda_{n,0}^{(\af)}(c)>\lambda_{n,1}^{(\af)}(c)>\cdots>\lambda_{n,k}^{(\af)}(c)>\cdots>0.\;
\end{equation}
Moreover, $\big\{\swf_{k,{\ell}}(\bx;c)\big\}_{\ell\in \Upsilon_n^d}^{k,n\in {\mathbb N}}$ are also the eigenfunctions
of ${\mathcal Q}_{c}^{(\af)}:$
\begin{equation}\label{Qc}
{\mathcal Q}_{c}^{(\af)}[\swf_{k,{\ell}}](\bx;c)=\mu_{n,k}^{(\af)}(c)\,\swf_{k,{\ell}}(\bx;c),
\end{equation}
and the eigenvalues have the relation:
\begin{equation}\label{mulambda}
\mu_{n,k}^{(\alpha)}(c)=
|\lambda_{n,k}^{(\af)}(c)|^2\,.
\end{equation}
\end{thm}
\begin{proof} We first prove  \eqref{eigen1}.  Let $\mathscr{D}_{c,\bx}^{(\af)}$ be the
Sturm-Liouville  operator  defined in  \eqref {opDxB}.
 One verifies readily that
\begin{equation}\label{Dexpst}
\begin{split}
\mathscr{D}_{c,\bx}^{(\af)}&\e^{-\ri c \langle \bx, \bs t\rangle}\overset{\re{Ldef2}}=
\big[-\nabla\cdot(\bs{\rm I}-\bx\bx^{\tr})\nabla+2\af  \bx\cdot\nabla + c^2 \|\bx\|^2\big]\e^{-\ri c\langle \bx, \bs t\rangle}
\\
&=\big[c^2\|\bs t\|^2-(2\af+d+1)\ri c   \bx\cdot\bs t -c^2 (\bx\cdot\bs t)^2 + c^2 \|\bx\|^2\big]\e^{-\ri c\langle \bx, \bs t\rangle}=\mathscr{D}_{c,\bs t}^{(\af)}\e^{-\ri c\langle \bx, \bs t\rangle}.
\end{split}
\end{equation}
Thus, we obtain from  \eqref{Dself1}, \eqref{varphichi2}  and \eqref{Dexpst} that
\begin{equation*}
\begin{split}
&\chi_{n,k}^{(\af)} \int_{\ball^d}  \e^{-\ri
c \langle\bx,\bs t\rangle}\psi_{k,{\ell}}^{\alpha,n}(\bs t;c)\varpi_{\af}(\bs t) \d\bs t\overset{\re{varphichi2}}=\int_{\ball^d} \varpi_{\alpha}(\bs t) \e^{-\ri
c\langle\bx,\bs t\rangle}\mathscr{D}_{c,\bs t}^{(\af)}\psi_{k,{\ell}}^{\alpha,n}(\bs t;c)\d\bs t\\
&\qquad \overset{\re{Dself1}} =\int_{\ball^d}\varpi_{\alpha}(\bs t)\psi_{k,{\ell}}^{\alpha,n}(\bs t;c)\mathscr{D}_{c,\bs t}^{(\af)}\e^{-\ri
c\langle\bx,\bs t\rangle}\d\bs t \overset{\re{Dexpst}}=\int_{\ball^d}\varpi_{\alpha}(\bs t)\psi_{k,{\ell}}^{\alpha,n}(\bs t;c)\mathscr{D}_{c,\bx}^{(\af)}\e^{-\ri c\langle\bx,\bs t\rangle}\d\bs t\\
&\qquad =\mathscr{D}_{c,\bx}^{(\af)}\int_{\ball^d}\e^{-\ri c\langle\bx,\bs t\rangle} \psi_{k,{\ell}}^{\alpha,n}(\bs t;c) \varpi_{\alpha}(\bs t) \d\bs t,
\end{split}
\end{equation*}
or equivalently,  
$$
\mathscr{D}_{c,\bx}^{(\af)}\left({\mathscr F}_c^{(\af)}[\swf_{k,{\ell}}]\right)=\chi_{n,k}^{(\af)} {\mathscr F}_c^{(\af)}[\swf_{k,{\ell}}].
$$
This implies ${\mathscr F}_c^{(\af)}[\swf_{k,{\ell}}]$ is an eigenfunctions
of $\mathscr{D}_{c,\bx}^{(\af)}$ corresponding to the eigenvalue $\chi_{n,k}^{(\af)}$.

On the other hand, by resorting to the spherical-polar coordinates
$\bx = r\bs{\hat x}$ and $\bs \tau = \tau \bs{\hat \tau}$
 with $r,\tau\ge 0$ and $\bs{\hat x},\bs{\hat \tau}\in \sph^{d-1}$,  we further deduce that
\begin{equation}
\begin{split}
   {\mathscr F}_{c}^{(\af)} &[\swf_{k,{\ell}} ](\bx)=\int_{\ball^d}\e^{-\ri
c \langle \bx, \bs \tau\rangle}\swf_{k,{\ell}} (\bs \tau)\omega_{\alpha}^{\ball}(\bs \tau)\d\bs \tau
\\
=&\,  \int_0^1 (1-\tau^2)^{\alpha}  \tau^{n+d-1} \phi_{k}^{\alpha,n}(2\tau^2-1; c)  \d{\tau}\
\int_{\sph^{d-1}} \e^{-\ri  c\tau r \langle {\bs{ \hat{x}}}, \hat {\bs \tau} \rangle}  Y^n_{\ell}(\hat{\bs \tau})
 \d\s({\hat {\bs \tau}})
 \\
 \overset{\eqref{eq:Ffx}}=&\, Y_{\ell}^{n}(\hat \bx) \int_0^1 (1-\tau^2)^{\alpha}  \tau^{n+d-1} \phi_{k}^{\alpha,n}(2\tau^2-1; c)   \frac{ (2\pi)^{\frac{d}{2}}\, (-\i) ^n }{(c \tau r)^{\frac{d-2}2}}  J_{n+\frac{d-2}2}(c\tau r)  \d{\tau},
\end{split}\label{Freduc}
\end{equation}
which shows that $ {\mathscr F}_{c}^{(\af)} [\swf_{k,{\ell}} ](\bx)$
has the spherical component $Y_{\ell}^{n}(\hat \bx)$.
Hence, we conclude that ${\mathscr F}_{c}^{(\af)} [\swf_{k,{\ell}} ](\bx)$ is a multiple of
$\swf_{k,{\ell}} (\bx)$ itself. Thus, for certain $\lambda_{n,k,\ell}^{(\af)} $,
\begin{align}
\label{Fmulti}
{\mathscr F}_{c}^{(\af)} &[\swf_{k,{\ell}} ](\bx) = (-\i)^{n}(-1)^k \lambda_{n,k,\ell}^{(\af)} \swf_{k,{\ell}} (\bx) .
\end{align}
Furthermore, a combination of  \eqref{Freduc} and \eqref{Fmulti} yields
\begin{align}
\label{ITphi}
\begin{split}
(2\pi)^{\frac{d}{2}} &c ^n  \int_0^1 (1-\tau^2)^{\alpha}  \tau^{2n+d-1} \phi_{k}^{\alpha,n}(2\tau^2-1; c)     \frac{J_{n+\frac{d-2}2}(c\tau r)}{ (c\tau r)^{n+\frac{d-2}2} }   \d{\tau}
\\
&
= (-1)^k\lambda_{n,k,\ell}^{(\af)}  \phi_{k}^{\alpha,n}(2r^2-1; c)
= (-1)^{k}\lambda_{n,k}^{(\af)} \phi_{k}^{\alpha,n}(2r^2-1; c),
\end{split}
\end{align}
where the second equality sign reveals that $ \lambda_{n,k,\ell}^{(\af)}
= \lambda_{n,k}^{(\af)}= \lambda_{n,k}^{(\af)}(c)$ is  independent of $\ell$.  Thus    \eqref{eigen1} follows
and $\lambda_{n,k}^{(\af)}$ is real.

We now verify \eqref{Qc}.  By \eqref{eigen1}, one readily checks that
\begin{align*}
  ({{\mathscr F}_c^{(\alpha)}})^{*} \big[ \swf_{k,l}\big] (\bx; c) = \i^{n+2k} \lambda_{n,k}^{(\af)}
 \swf_{k,l} (\bx; c).
\end{align*}
Then \eqref{Qc} is a direct consequence of \eqref{eigen1} and the above equation.

 We next verify that $\lambda_{n,k}^{(\af)}(c)>0.$ Applying the differential operator $(\frac1{4r}\partial_r)^l$ on both sides of  \eqref{ITphi},
followed by the recurrence relation \eqref{DJz} of Bessel functions for differentiation 
leads to
\begin{align}\label{dkint}
\begin{split}
{(-1)^l}\frac{(2\pi)^{\frac{d}{2}}c ^{n+2l} }{4^l} &   \int_0^1 (1-\tau^2)^{\alpha}  \tau^{2n+d+2l-1} \phi_{k}^{\alpha,n}(2\tau^2-1; c)     \frac{J_{n+\frac{d-2}2+l}(c\tau r)}{ (c\tau r)^{n+\frac{d-2}2+l} }   \d{\tau}\\
=& {(-1)^k}\lambda_{n,k}^{(\af)}  \Big(\frac{1}{4r}\partial_{r}^l\phi_{k}^{\alpha,n} (2r^2-1; c)\Big)^{l}.
\end{split}
\end{align}
Taking limits as $r\rightarrow 0$ and letting $l=k,$ yields
\begin{align*}
\frac{\pi^{\frac{d}{2}}c ^{n+2k} }{2^{n+3k-1}\Gamma(n+\frac{d}2+k)}    \int_0^1 (1-\tau^2)^{\alpha}  \tau^{2n+d+2k-1} \phi_{k}^{\alpha,n}(2\tau^2-1; c)     \d{\tau}
=\lambda_{n,k}^{(\af)} \partial_{\eta}^k \phi_{k}^{\alpha,n}(-1; c),
\end{align*}
where we used the 
series representation \eqref{Jexpan} of the Bessel function.

Furthermore, changing variables $\eta=2\tau^2-1$ in the above equation  shows that
\begin{align}\label{dkintEqr}
\frac{ \pi^{\frac{d}{2}}c ^{n+2k} }{2^{2n+4k+\frac{d}2+\alpha}\Gamma(n+\frac{d}2+k)}\int_{-1}^1(1+\eta)^k \phi_{k}^{\alpha,n}(\eta; c) \omega_{\alpha,\beta_n}{(\eta)}    \d{\eta}
=\lambda_{n,k}^{(\af)} (c) \partial_{\eta}^k \phi_{k}^{\alpha,n}(-1; c).
\end{align}
Thanks to \eqref{gtog}, we find from \eqref{Jacobiorth} that as $c$ approaches to zero,
\[
\partial_{\eta}^k \phi_{k}^{\alpha,n}(-1; c)\to \partial_{\eta}^k P_k^{(\alpha,\beta_n)}(-1)= k! \, \kappa_k^{(\alpha,\beta_n)} ,
\]
and
\[
\int_{-1}^1(1+\eta)^k\phi_{k}^{\alpha,n}(\eta; c)\omega_{\af,\beta_n}(\eta)\d\eta\to  \int_{-1}^1(1+\eta)^kP_k^{(\alpha,\beta_n)}(\eta)\omega_{\af,\beta_n}(\eta)\d\eta=\frac{2^{\alpha+\beta_n+2}}{\kappa_{k}^{(\alpha,\beta_n)}}.
\]
Hence, a direct calculation by using  \eqref{knkn} and the above two facts leads to
\begin{equation}\label{lambforc}
\lim_{c\rightarrow 0}\frac{\lambda_{n,k}^{(\alpha)}}{c^{n+2k}}=\frac{(\pi)^{\frac{d}{2}}\Gamma(\beta_n+1)h_k^{(\alpha,\beta_n)} }{2^{4k+2n+d+\alpha}\Gamma(k+\beta_n+1)\Gamma(n+\frac{d}2+k)\kappa_{k}^{(\alpha,\beta_n)}}.
\end{equation}
Then, the equation \eqref{lambforc} implies that for sufficient small $c,$
$\lambda_{n,k}^{(\af)}(c)>0$ for all $n, k\ge 0$ and $\af>-1.$  In fact, this property holds for all
$c>0,$ since  if there exists $\tilde c>0$ such that
$\lambda_{n,k}^{(\af)}(\tilde c)<0,$ we are able to find $c_1>0$ such that $\lambda_{n,k}^{(\af)}(c_1)=0, $
which is not possible.

 We are now in a position to justify  \eqref{orderth}. Let  $ \phi_{k}^{\alpha,n}$  and $ \phi_{k+1}^{\alpha,n}$ be the successive eigenfunctions
of \eqref{ITphi}.  Then an immediate consequence of  \eqref{dkint} with $l=1$ gives
\begin{align*}
&\frac{(2\pi)^{\frac{d}{2}} c ^{n+2}}4   \int_0^1 (1-\tau^2)^{\alpha}  \tau^{2n+d+1} \phi_{k}^{\alpha,n}(2\tau^2-1; c)     \frac{J_{n+\frac{d}2}(c\tau r)}{ (c\tau r)^{n+\frac{d}2} }   \d{\tau}
= \lambda_{n,k}^{(\af)}  \big(\phi_{k}^{\alpha,n}\big)'(2r^2-1; c),
\\
&\frac{(2\pi)^{\frac{d}{2}} c ^{n+2}}4    \int_0^1 (1-\tau^2)^{\alpha}  \tau^{2n+d+1} \phi_{k+1}^{\alpha,n}(2\tau^2-1; c)     \frac{J_{n+\frac{d}2}(c\tau r)}{ (c\tau r)^{n+\frac{d}2} }   \d{\tau}
= \lambda_{n,k+1}^{(\af)}  \big(\phi_{k+1}^{\alpha,n}\big)'(2r^2-1; c).
\end{align*}
Multiplying the first equation by
$\phi_{k+1}^{\alpha,n}(2r^2-1;c) \omega_{\alpha}(r^2) r^{2n+d+1}$ and  integrating the resultant equation over $(0,1)$,
we derive from the second  equation above that
\begin{align*}
 \lambda_{n,k}^{(\af)}
 &\int_0^1 \big[  \big(\phi_{k}^{\alpha,n}\big)'(2r^2-1; c)\phi_{k+1}^{\alpha,n} (2r^2-1; c)\omega_{\alpha}(r^2) r^{2n+d+1}\d r
 \\
 =&\frac{(2\pi)^{\frac{d}{2}} c ^{n+2}}4 \int_0^1  \int_0^1  (\tau r)^{2n+d+1} \phi_{k}^{\alpha,n}(2\tau^2-1; c)
  \phi_{k+1}^{\alpha,n} (2r^2-1; c)   \frac{J_{n+\frac{d}2}(c\tau r)}{ (c\tau r)^{n+\frac{d}2} }\omega_{\alpha}(\tau^2) \omega_{\alpha}(r^2)   \d{\tau}   \d{r}
\\
=&\lambda_{n,k+1}^{(\af)}\int_0^1 \big[  \big(\phi_{k+1}^{\alpha,n}\big)'(2r^2-1; c)\phi_{k}^{\alpha,n} (2r^2-1; c)\omega_{\alpha}(r^2) r^{2n+d+1}\d r.
\end{align*}
which gives
\begin{align}
\label{difflam}
  \lambda_{n,k}^{(\af)}- \lambda_{n,k+1}^{(\af)} = \lambda_{n,k}^{(\af)}
  \left( 1 -\frac{   \displaystyle \int_{-1}^1  \big(\phi_{k}^{\alpha,n}\big)'(\eta; c)   \phi_{k+1}^{\alpha,n} (\eta; c)
 \omega^{\alpha,\beta_n+1}(\eta) \d \eta }{   \displaystyle  \int_{-1}^1  \big(\phi_{k+1}^{\alpha,n}\big)'(\eta; c)   \phi_{k}^{\alpha,n} (\eta; c)
 \omega^{\alpha,\beta_n+1}(\eta) \d \eta}   \right).
\end{align}

Now as $c\to 0$, $\phi_{k}^{\alpha,n}(\eta)\to P^{(\alpha,\beta_n)}_k(\eta)$
and $(\phi_{k}^{\alpha,n})'(\eta)\to \partial_{\eta}P^{(\alpha,\beta_n)}_k(\eta)$.
The numerator in \eqref{difflam} approaches
\begin{align*}
\int_{-1}^1 \partial_{\eta} P^{(\alpha,\beta_n)}_k(\eta)(1+\eta) P^{(\alpha,\beta_n)}_{k+1}(\eta)
\omega_{\alpha,\beta_n} \d \eta = 0.
\end{align*}
To estimate the denominator, we resort the following identity,
\begin{align*}
h^{(\alpha,\beta_n)}_{k+1} &\partial_{\eta}P^{(\alpha,\beta_n)}_{k+1}(\eta)(1+\eta)  = \frac{k+\alpha+\beta_n+2}{2} h^{(\alpha+1,\beta_n+1)}_{k} P^{(\alpha+1,\beta_n+1)}_{k}(\eta)(1+\eta)
\\
=&\, \frac{k+\alpha+\beta_n+2}{2k+\alpha+\beta_n+3} \big[ (k+\beta_n+1) h^{(\alpha+1,\beta_n)}_{k}  P^{(\alpha+1,\beta_n)}_{k}(\eta)
+(k+1) h^{(\alpha+1,\beta_n)}_{k+1} P^{(\alpha+1,\beta_n)}_{k+1}(\eta)   \big]
\\
=&\,  \frac{k+\alpha+\beta_n+2}{2k+\alpha+\beta_n+3}
 \Big[ (k+\beta_n+1) \sum_{\nu=0}^k
 \frac{  (\beta_n+\nu+1)_{k-\nu} (\alpha+\beta_n+2\nu+1)   }
 {  (\alpha+\beta_n+\nu+1)_{k+1-\nu} }   h^{(\alpha,\beta_n)}_{\nu} P^{(\alpha,\beta_n)}_{\nu}(\eta)
 \\
 & + (k+1)  \sum_{\nu=0}^{k+1} \frac{  (\beta_n+\nu+1)_{k+1-\nu} (\alpha+\beta_n+2\nu+1)   }
 {  (\alpha+\beta_n+\nu+1)_{k+2-\nu} }   h^{(\alpha,\beta_n)}_{\nu} P^{(\alpha,\beta_n)}_{\nu}(\eta)\Big]
 \\
 =&\,
 \sum_{\nu=0}^{k+1}
 \frac{  (\beta_n+\nu+1)_{k+1-\nu} (\alpha+\beta_n+2\nu+1)   }
 {  (\alpha+\beta_n+\nu+1)_{k+1-\nu} }   h^{(\alpha,\beta_n)}_{\nu} P^{(\alpha,\beta_n)}_{\nu}(\eta)
 ,
\end{align*}
where  the second equality sign is derived from \cite[p.\,71, (4.5.4)]{Szego75} and  the third equality sign is derived from \cite[Theorem 7.1.3]{Askey}. 
As a result, the denominator approches
\begin{align*}
\int_{-1}^1 &\partial_{\eta}P^{(\alpha,\beta_n)}_{k+1}(\eta)(1+\eta) P^{(\alpha,\beta_n)}_{k}(\eta) \omega_{\alpha,\beta_n}(\eta)\d \eta
 \\
&= \int_{-1}^1  \frac{  (\beta_n+k+1) (\alpha+\beta_n+2k+1)   h_{k}^{(\alpha,\beta_n)} }
 {  (\alpha+\beta_n+k+1)   h_{k+1}^{(\alpha,\beta_n)} }
  P^{(\alpha,\beta_n)}_{k}(\eta) P^{(\alpha,\beta_n)}_{k}(\eta) \omega_{\alpha,\beta_n}(\eta)\d \eta
  \\
&  =
2^{\alpha+\beta_n+2}  \sqrt{\frac{(k+1) (k+\beta_n+1)(2k+\alpha+\beta_n+1) (2k+\alpha+\beta_n+3)  }{(k+\alpha+1)(k+\alpha+\beta_n+1)  }}.
\end{align*}
By making $c$ sufficiently small,  the fraction on the right of the \eqref{difflam}
is of absolute values less than unity and
$$
 \lambda_{n,k}^{(\af)}- \lambda_{n,k+1}^{(\af)} = \lambda_{n,k}^{(\af)} (1+ \mathcal{O}(1)) > 0.
$$

Since for $c\neq 0$ and $ \lambda_{n,k}^{(\af)}$ for fixed $n$ are all district and positive,
the ordering in \eqref{orderth} must hold.
\end{proof}

\section{Evaluation of ball PSWFs and connections with some existing PSWFs}\label{sect:5}

In this section, we present an efficient algorithm to evaluate the PSWFs and their associated eigenvalues.  We also illustrate some connections with e.g., circular PSWFs introduced in literature.
\subsection{Spectrally accurate Bouwkamp algorithm}
As with  the Slepian basis,  an efficient approach to evaluate  the PSWFs  is the Bouwkamp-type algorithm (cf. \cite{Bouw50,XiaoH.R01,Boyd.acm}).
We start with  the  differential equation \eqref{varphichi2}  of the PSWFs $\{\psi^{\alpha,n}_{k,\ell}\}$,
which can be regarded as a perturbation of \eqref{Ldefr} for the ball polynomials $\{P^{\alpha,n}_{j,\ell}\}$
 here.
In view of  \eqref{orthnPkl} and \eqref{psinew},  we  can simply expand $\psi^{\alpha,n}_{k,\ell}(\bs x)= \phi^{\af,n}_{k}(2\|\bs x\|^2-1)Y^n_{\ell}(\bs x)$  in  an infinite series in  $\{P^{\alpha,n}_{j,\ell}\}_{j=0}^{\infty}$,
\begin{align}
\label{ExpanPsi}
\begin{split}
\psi^{\alpha,n}_{k,\ell}(\bs x;c)
=  \sum_{j=0}^{\infty}
\beta^{n,k}_{j} P^{\alpha,n}_{j,\ell}(\bs x).
\end{split}
\end{align}
Thanks to  the definition  \eqref{orthnPkl} of the ball polynomials and the  three-term recurrence relation \eqref{Jacobi} of the normalized Jacobi polynomials,
we derive that  for any $\ell \in  \Upsilon_{\!n}^d\, $ and  $n,j\in \NN$,
\begin{align}
\label{TTRBall}
\|\bs x\|^2 P^{\alpha,n}_{j,\ell}(\bs x) = \frac{a^{(\alpha,\beta_n)}_j}{2} P^{\alpha,n}_{j+1,\ell}(\bs x)
+ \frac{1+b^{(\alpha,\beta_n)}_j}{2} P^{\alpha,n}_{j,\ell}(\bs x) + \frac{a^{(\alpha,\beta_n)}_{j-1}}{2} P^{\alpha,n}_{j-1,\ell}(\bs x).
\end{align}
Substituting the  expansion \eqref{ExpanPsi} into \eqref{varphichi2} and using the
three-term recurrence  \eqref{TTRBall}
together with the Sturm-Liouville equation  \eqref{Ldef}, we obtain
\begin{equation*}
\begin{split}
\sum_{j=0}^{\infty}&\bigg[ \Big( \gamma_{n+2j}^{(\af)} +
\frac{ b_{j}^{(\af,\beta_n)} c^2 +c^2}{2} \Big)  \beta_{j}^{n,k}
 + \frac{a_{j-1}^{(\af,\beta_n)}  c^2}{2}  \beta_{j-1}^{n,k}
 + \frac{a_{j}^{(\af,\beta_n)}  c^2}{2}  \beta_{j+1}^{n,k}
- \chi^{(\af)}_{n,k}(c)  \beta_{j}^{n,k}  \bigg]  {P}_{j,\ell}^{\af, n}
= 0.
\end{split}
\end{equation*}
As a result, the expansion coefficients $\{\beta_{j}^{n,k}\}_{j=0}^{\infty}$ in \eqref{ExpanPsi} are determined by
the following three-term recurrence relation:
\begin{equation}
\label{eigsys}
 \Big [  \gamma_{n+2j}^{(\af)} +
 \frac{ (b_{j}^{(\af,\beta_n)}+1)  c^2}{2} - \chi^{(\af)}_{n,k}(c) \Big] \beta_{j}^{n,k}
 + \frac{a_{j-1}^{(\af,\beta_n)}  c^2}{2}  \beta_{j-1}^{n,k}
 + \frac{a_{j}^{(\af,\beta_n)}  c^2}{2}  \beta_{j+1}^{n,k} = 0, \quad j\ge 0.
\end{equation}

\begin{rem}\label{eigenequiv}
The  matrix eigen-problem \eqref{eigsys} can be equivalently deduced  from evaluating the radial component $\phi^{\af,n}_{k}$
of $\psi^{\alpha,n}_{k,\ell}(\bs x)= \phi^{\af,n}_{k}(2\|\bs x\|^2-1)Y^n_{\ell}(\bs x)$  in terms of
 Jacobi polynomials with the unknown coefficients $\{\beta_{j}^{n,k}\}$:
\begin{equation}\label{genexp}
\phi_{k}^{\alpha,n}(\eta; c)=\sum_{j=0}^{\infty}\beta_{j}^{n,k} P_{j}^{(\af,\beta_n)}(\eta).
\end{equation}
Indeed, from \eqref{phieig}, we have
\begin{equation}\label{eigin1dr}
 \bigg[ -\frac{4}{\omega_{\af,\beta_n}(\eta)} \partial_{\eta} \left(\omega_{\af+1,\beta_n+1}(\eta)  \partial_{\eta} \right) +   \frac{c^2(\eta+1)}{2}  +\gamma_{n}^{(\alpha)}  \bigg]\phi_{k}^{\alpha,n}(\eta; c)  =\chi^{(\af)}_{n,k}(c) \,  \phi_{k}^{\alpha,n}(\eta; c).
\end{equation}
Substituting this expansion into \eqref{eigin1dr} and using the
three-term recurrence  \eqref{Jacobi}
together with the Sturm-Liouville equation  \eqref{JacobiSL}, we derive
\begin{equation*}
\begin{split}
\sum_{j=0}^{\infty}&\bigg[ \Big( 4 \lambda_j^{(\af,\beta_n)} +\gamma_{n}^{(\af)} +
\frac{ b_{j}^{(\af,\beta_n)} c^2 +c^2}{2} \Big)  \beta_{j}^{n,k}
 + \frac{a_{j-1}^{(\af,\beta_n)}  c^2}{2}  \beta_{j-1}^{n,k}
 + \frac{a_{j}^{(\af,\beta_n)}  c^2}{2}  \beta_{j+1}^{n,k}
\bigg]  {P}_{j}^{(\af, \beta_n)}(\eta)
\\
&=\chi^{(\af)}_{n,k}(c)  \sum_{j=0}^{\infty}  \beta_{j}^{n,k} P_{j}^{(\af,\beta_n)}(\eta),\quad \eta\in(-1,1).
\end{split}
\end{equation*}
Then we can obtain \eqref{eigsys} from the above.
\end{rem}

Thanks to \eqref{eigsys}, we now use the Bouwkamp-type algorithm to evaluate $\big\{\psi^{\alpha,n}_{k,\ell}, \chi_{n,k}^{(\alpha)}\big\}$ with $2k+n\le N$.
Following  the truncation rule  in \cite{Boyd.acm,wang2009new}, we set $ M=2N+2\alpha+30$ and
suppose   $\big\{\tilde \psi^{\alpha,n}_{k,\ell}, \tilde \chi_{n,k}^{(\alpha)}\big\}$ to be the  approximation of $\big\{{\psi}^{\alpha,n}_{k,\ell}, \chi_{n,k}^{(\alpha)}\big\}$ with
\begin{align*}
 \tilde{\psi}^{\alpha,n}_{k,\ell}(\bs x;c)
= \sum_{j=0}^{\lceil \frac{M-n}{2} \rceil}
\tilde \beta^{n,k}_{j} P^{\alpha,n}_{j,\ell}(\bs x),\quad  2k+n\le N.
\end{align*}
Denote  $K=\lceil \frac{M-n}{2} \rceil$. Then the Bouwkamp-type algorithm gives  the following  finite algebraic eigen-system for $\{ \tilde \beta^{n,k}_{j}\}_{j=0}^K$ and
$\tilde \chi_{n,k}^{(\alpha)}$,
\begin{equation}\label{Feigsys}
({\bs A}-\tilde \chi_{n,k}^{(\alpha)}\cdot{\bs I})\vec{\beta}^{n,k}=\bs 0,
\end{equation}
where
$\vec{\beta}^{n,k}=(\tilde\beta^{n,k}_{0},\tilde\beta^{n,k}_{1},\dots,\tilde\beta^{n,k}_{K})$
and ${\bs A}$ is the $(K+1)\times (K+1)$ symmetric tridiagonal matrix whose nonzero entries are given by
\begin{equation}\label{333term}
\begin{split}
A_{j,j}=\gamma_{n+2j}^{(\af)}+\big(b_{j}^{(\af,\beta_n)}+1\big)\cdot \frac{c^2}{2};\quad  A_{j,j+1}=A_{j+1,j}=a_{j}^{(\af,\beta_n)} \cdot \frac{c^2}{2}, \quad  0\le j \le K.
\end{split}
\end{equation}

We next introduce  a formula to compute  the eigenvalues $\big\{\lambda_{n,k}^{(\alpha)}(c)\big\}$ associated with the integral operator \eqref{bdlim} in very stable manner.
\begin{thm}\label{complamda} For any $\af>-1$ and $c>0,$ we have
\begin{equation}\label{lamcomp}
\lambda_{n,k}^{(\af)}(c)=\frac{\pi^{\frac{d}{2}}c^n \sqrt{\Gamma(\af+1) }}{
   2^{n-\frac{1}{2}}\sqrt{\Gamma(n+\frac{d}{2})\Gamma(\af+n+d/2+1) }}\cdot\frac{\beta_0^{n,k}}{\phi_{k}^{\alpha,n}(-1; c)},
\end{equation}
where $\beta_0^{n,k}$ is given in \eqref{genexp}.
\end{thm}
\begin{proof} We find  from \eqref{phieig} that
\begin{equation*}
-2(\beta_n+1)\partial_{\eta}\phi_{k}^{\af,n}(-1)=\frac{1}{4}(\chi_{n,k}^{(\af)}-\gamma_{n}^{(\af)})\phi_{k}^{\af,n}(-1).
\end{equation*}
If $\phi_{k}^{\af,n}(-1)$ vanishes, then so does $\partial_\eta\phi_{k}^{\af,n}(-1).$ Differentiating  \eqref{phieig} shows that if $\phi_{k}^{\af,n}(-1)$ and $\partial_\eta\phi_{k}^{\af,n}(-1)$ vanish,  so does $\partial_x^2\phi_{k}^{\af,n}(-1).$ Repeated differentiation implies  that if $\phi_{k}^{\af,n}(-1)=0,$ then
$\phi_{k}^{\af,n}(\eta)\equiv 0.$  This  results in the contradiction,
so we have  $\phi_{k}^{\alpha,n}(-1; c)\not=0$ for any $k\geq 0$ and $n\ge 0$.

Next, we obtain from \eqref{dkintEqr} with $k=0$ that
\begin{equation*}
\frac{\pi^{\frac{d}{2}}{ c}^n}{2^{2n+\af+\frac{d}{2}}\Gamma(n+\frac{d}2)}   \int_{-1}^1
\phi_{k}^{\alpha,n}(\eta; c)\omega_{\af,\beta_n}(\eta)\d\eta=  \lambda_{n,k}^{(\af)}\phi_{k}^{\alpha,n}(-1; c).
\end{equation*}
This yields
\begin{equation*}
\begin{split}
\lambda_{n,k}^{(\af)}(c)&=\frac{\pi^{\frac{d}{2}}c^n }{
   2^{2n+\af+\frac{d}{2}}\Gamma(n+\frac{d}2)\phi_{k}^{\alpha,n}(-1; c)}\int_{-1}^1\phi_{k}^{\alpha,n}(\eta; c)\omega_{\af,\beta_n}(\eta)\d\eta\\
&=\frac{\pi^{\frac{d}{2}}c^n }{
   2^{2n+\af+\frac{d}{2}}\Gamma(n+\frac{d}{2})\phi_{k}^{\alpha,n}(-1; c)}\int_{-1}^1\Big(\sum_{j=0}^\infty
\beta_{j}^{n,k} P_{j}^{(\af,\beta_n)}(\eta)\Big)\om_{\af,\beta_n}(\eta)
d\eta\\
&=\frac{\pi^{\frac{d}{2}}c^n \beta_0^{n,k} h_0^{(\af,\beta_n) }}{
   2^{n-1}\Gamma(n+\frac{d}{2})\phi_{k}^{\alpha,n}(-1; c)}\\
&=\frac{\pi^{\frac{d}{2}}c^n \sqrt{\Gamma(\af+1) }}{
   2^{n-\frac{1}{2}}\sqrt{\Gamma(n+\frac{d}{2})\Gamma(\af+n+d/2+1) }}\cdot\frac{\beta_0^{n,k}}{\phi_{k}^{\alpha,n}(-1; c)}.
\end{split}
\end{equation*}
The proof is now completed.
\end{proof}

\subsection{Connection with existing works}

Below, we particularly look at  the ball PSWFs with $d=1,2$ and special parameter $\alpha$, and demonstrate  their connections with existing PSWFs.


For
$d=1$, one has  $\Upsilon_{\!n}^1=\{ 1\}$ (cf.  \eqref{indexsetA}) for $n=0,1 $ and $\Upsilon_{\!n}^d=\emptyset$ for $n\ge 2$.
Recall the formula in \cite[Theorem 4.1]{Szego75} with a  different normalisation for Jacobi polynomials,
\begin{align*}
&P^{(\alpha,\alpha)}_{2k}(\eta) =2^{\alpha+\frac12}P^{(\alpha,-\frac12)}_{k}(2\eta^2-1) ,
\quad P^{(\alpha,\alpha)}_{2k+1}(\eta) =2^{\alpha+\frac12}\eta P^{(\alpha,\frac12)}_{k}(2\eta^2-1) , \quad k\ge 0.
\end{align*}
Then by Remark \ref{exaples},  
 \begin{align*}
     &P^{\alpha,0}_{k,1} (x)= P^{(\alpha,-\frac12)}_{k}(2x^2-1) Y^0_1(x) = 2^{-\alpha-1} P^{(\alpha,\alpha)}_{2k}(x), \quad k\ge 0,
     \\
     &P^{\alpha,1}_{k,1} (x)= P^{(\alpha,\frac12)}_{k}(2x^2-1) Y^1_1(x) = 2^{-\alpha-1} P^{(\alpha,\alpha)}_{2k+1}(x),  \ \ \quad k\ge 0.
 \end{align*}
 The expansion \eqref{ExpanPsi} is then reduced to
 \begin{align*}
&\psi^{\alpha,0}_{k,1}(x;c)= 2^{-\alpha-1} \sum_{j=0}^{\infty}
\beta^{0,k}_{j} P^{(\alpha,\alpha)}_{2j}(x):=\psi^{(\alpha)}_{2k}(x;c),  \quad\ \  k\ge 0,
\\
&\psi^{\alpha,1}_{k,1}(x;c)= 2^{-\alpha-1} \sum_{j=0}^{\infty}
\beta^{1,k}_{j} P^{(\alpha,\alpha)}_{2j+1}(x):=\psi^{(\alpha)}_{2k+1}(x;c), \quad   k\ge 0.
\end{align*}
It  implies that the Bouwkamp algorithm for $d=1$ here is  exactly reduced to  the even/odd decoupled one in one dimension,
see  \cite{Boyd.acm,Slep61} for $\alpha=0$ and \cite{wang2009new} for general $\alpha>-1$  for details.
In particular,  Boyd \cite{Boyd.acm} suggested a cut-off  $M=2N+30$ for  evaluating  the
 Slepian basis $\{\psi_n^{(0)}\}_{n=0}^N$.  In \cite{wang2009new}, we expand $\{\psi_n^{(\af)}(x;c)\}$ in terms of the normalized Gegenbauer polynomials,
\begin{equation}
{\psi_n^{(\af)}(x; c)}=\sum_{k=0}^{\infty}\beta_k^n\ \gi_k(x)
\quad{\rm with}\quad \beta_k^n=\int_{-1}^1\psi_n^{(\af)}(x;c)G_k^{(\af)}(x)\omega_\af(x)\d x,
\end{equation}
where $\gi_k(x)=2^{-\af-1}P_k^{(\af,\af)}(x),\; k\geq 0.$ Here, we use the truncation $M = 2N + 2\af + 30$ for the computations of
$\{\psi_n^{(\af)}\}_{n=0}^N.$ We also notice that $\beta_k^n = 0$ if $n + k$ is odd, which allows us to obtain a symmetric tridiagonal system, and efficient eigen-solvers can be applied.


To explore the connection  in two dimensions,   we  denote
 $$
 \psi_{n,k}^{(\alpha)}(r;c)=r^{n+\frac{d-1}{2}}\phi_{k}^{\alpha,n}(2r^2-1; c),
 $$
and then transform \eqref{Oprpsi} and  \eqref{ITphi} into
\begin{align}\label{psi0}
\begin{split}
&\Big[-(1-r^2)^{-\alpha}\partial_r(1-r^2)^{\alpha+1}\partial_r
+\frac{(2n+d-1)(2n+d-3)}{4r^2}+c^2r^2\Big]
\psi_{n,k}^{(\alpha)}(r;c)
\\
  =&\Big[\chi_{n,k}^{(\alpha)}(c)+\frac{(d-1)(4\alpha+d+1)}4\Big] \psi_{n,k}^{(\alpha)}(r;c) ,
  \end{split}
\end{align}
 and
 \begin{align}
 \label{psi0Int}
\begin{split}
 \int_0^1  (1-\tau^2)^{\alpha} \psi_{n,k}^{(\alpha)}(\tau; c)  J_{n+\frac{d-2}2}(c\tau r) \sqrt{c\tau r}\,   \d{\tau}
= \frac{c ^{\frac{d-1}{2}} (-1)^{k} }{(2\pi)^{\frac{d}{2}} } \lambda_{n,k}^{(\alpha)}(c)     \psi_{n,k}^{(\alpha)}(r; c),
\end{split}
\end{align}
respectively.   In particular, for  $d=2$ and  $\alpha=0$, we have
\begin{align}\label{psi0Dif20}
&\Big[-\partial_r(1-r^2)\partial_r
+\frac{n^2-\frac14}{r^2}+c^2r^2\Big]
\psi_{n,k}^{(0)}(r;c)
  =\Big[\chi_{n,k}^{(0)}(c)+\frac{3}4\Big] \psi_{n,k}^{(0)}(r;c) ,
\end{align}
and
\begin{align}
 \label{psi0Int20}
& \int_0^1  \psi_{n,k}^{(0)}(\tau; c)  J_{n}(c\tau r) \sqrt{c\tau r}\,   \d{\tau}
= \frac{(-1)^{k} \sqrt{c}  }{ 2\pi  }\, \lambda_{n,k}^{(0)}(c)\,     \psi_{n,k}^{(0)}(r; c).
\end{align}
Indeed,  \eqref{psi0Dif20} defines  the {\em generalized prolate spheroidal wave functions}  $\psi_{n,k}^{(0)}(r;c)$ in two dimensions   in \cite[(25)]{Slep64}.
Slepian \cite{Slep64} expanded $\psi_{n,k}^{(0)}(r;c)$ in a series of hypergeometric functions:
$$
\psi_{n,k}^{(0)}(r;c) = \sum_{j=0}^{\infty} d_j^{n,k}  r^{n+\frac12} {}_2F_1(-j, j+n+1; n+1; r^2),
$$
then used  the Bouwkamp algorithm for  solving \eqref{psi0Dif20}.
Actually, by simply setting
$$
  d_j^{n,k} = (-1)^j\, \binom{j+n}{j}^{-1} \sqrt{\frac{2}{2j+n+1}} \, \beta^{n,k}_j,
$$
one can also obtain   the infinite eigen-system \eqref{eigsys} for $d=2$ and $\alpha=0$.
\begin{rem}\label{with25} More precisely,
we can find the relation between $\{\psi_{n,k}(r;c),\chi_{n,k}(c)\}$ {\rm (cf. \cite{Slep64})}  and $\{\psi_{n,k}^{(\af)}(r;c),\chi_{n,k}^{(\af)}(c)\}$ from \eqref{psi0} and \eqref{ITphi} with $d=2$ and $\af=0,$
\begin{equation}\label{relatA}
\psi_{n,k}(r;c)=\sqrt{r}\psi_{ n,k}^{(0)}(r;c),\quad \chi_{n,k}(c)=\chi_{n,k}^{(0)}(c)+\frac{3}{4},\quad \lambda_{n,k}=c\Big(\sqrt{c}\lambda_{n,k}^{(0)}(c)/2\pi\Big)^2.
\end{equation}
It is seen that the eigen-functions therein are singular at $r=0.$
\end{rem}

While for $d=3$ and $\alpha=0$, Slepian considered  the eigenvalue problem  \eqref{Fmulti} of the finite Fourier transform,
 and then reduced it to
  \begin{align}
 \label{psi0Intd0}
\begin{split}
 \int_0^1  \psi_{n,k}^{(0)}(\tau; c)  J_{n+\frac{d-2}2}(c\tau r) \sqrt{c\tau r}\,   \d{\tau}
= \frac{(-1)^{k} \sqrt{c}  }{ 2\pi  }\,  \frac{ \sqrt{c} \lambda_{n,k}^{(0)}(c)}{\sqrt{ 2\pi  }}\,     \psi_{n,k}^{(0)}(r; c).
\end{split}
\end{align}
After a comparison between \eqref{psi0Intd0} with \eqref{psi0Int20},  Slepian finally evaluated the generalized PSWFs $ \psi_{n,k}^{(0)}(r; c)$ for $d=3$  in the absence of  its Sturm-Liouville differential equation   by solving \eqref{psi0Dif20} with $J_n$
and $\lambda_{n,k}^{(0)}$ replaced by $J_{n+\frac{d-2}2}$ and $ \frac{ \sqrt{c} \lambda_{n,k}^{(0)}}{\sqrt{ 2\pi  }}$,
respectively.


\subsection{Numerical results}
We first  present numerical results obtained from the previously described algorithms for $\chi_{n,k}^{(\af)}(c)$ and $\lambda_{n,k}^{(\af)}(c)$ with $d=2.$ In Table \ref{tab1},  we  tabulate  the numerical results  and compare with  {\cite[Table I]{Slep64}} with  $\alpha=0,$ and  for various choices of the parameters $c, n,k$.  Indeed, we are able to provide many more significant digits, and it shows  our formulation and algorithm in this special case are more stable. 
In Tables \ref{tab2}, we report  the values of $\psi_{n,k}^{(\af)}(r;c)\, (\af=0)$  corresponding to the eigenvalues displayed in Table \ref{tab1}.
For various $c, n$ and $k$, these results are accurate to at least $6$ digits with respect to those were given in {\cite[Table II]{Slep64}}.
On the other hand,  compared with our results with those obtained by the scheme in \cite{Amodio13}, we observe that to achieve same accuracy, the approach in \cite{Amodio13} needed about $10000$ points, while only about $2(n+2k)+30$ points are  required for the method herein.

\begin{table}[!th] \centering
\caption{\small The case $d=2$: $\chi_{n,k}^{(\af)}(c)$ and $\lambda_{n,k}^{(\af)}(c)$ with
$\af=0$.}
\label{tab1}
{\small
\begin{tabular}{lclclclclcl}
  \hline
   $c$ & $n$ &  $k$ & $\chi_{n,k}$ {\cite[Table I]{Slep64}}  & $\chi_{n,k}^{(0)}(c)+\frac{3}{4}$ & $\lambda_{n,k}$ {\cite[Table I]{Slep64}}&$c\Big(\sqrt{c}\lambda_{n,k}^{(0)}(c)/2\pi\Big)^2$ \\
  \hline\hline
$0.1$& $0$  &   $0$ & $7.5499895e-01 $ &  $7.549989583334328e-01$&$2.4968775e-03 $   &$ 2.496877494303882e-03$\\
$0.5$& $0$  &   $0$ & $8.7434899e-01 $ & $8.743489971815857e-01$&$ 6.0585348e-02$   &  $ 6.058534466942055e-02$\\
$1$& $0$  &     $0$& $1.2395933e+00 $ &  $1.239593258779101e+00$&$2.2111487e-01 $   &$2.211148636497345e-01$\\
$4$& $0$  &     $0$& $6.5208586e+00 $ &  $ 6.520858597472127e+00$&$ 9.7495117e-01 $   &$   9.749510755184038e-01$\\
$10$& $0$  &    $0$& $1.8690110e+01$ &  $ 1.869010993969090e+01$&$9.9999957e-01 $   &$ 9.999995234517773e-01$\\
\hline
$2$& $1$  &   $0$ & $6.3394615e+00 $ & $  6.339461594016627e+00$&$1.6123183e-01 $   &$1.612318294915764e-01 $\\
$2$& $1$  &   $1$ & $1.7912353e+01 $ & $ 1.791235348206654e+01$&$1.8549511e-04 $   &$1.854950923417457e-04 $\\
$2$& $1$  &   $2$ & $3.7820310e+01 $ & $ 3.782031001324489e+01 $&$1.9082396e-08 $   &$ 1.908239530607290e-08  $\\
$2$& $1$  &   $3$ & $6.5789319e+01 $ & $ 6.578931995056144e+01 $&$4.9988893e-13 $   &$  4.998888383640053e-13 $\\
$2$& $2$  &   $0$ & $1.1710916e+01 $ & $ 1.171091633298800e+01 $&$1.9088335e-02 $   &$ 1.908833481911065e-02 $\\
\hline
\end{tabular} }
\end{table}

\begin{table}[!th] \centering
\caption{\small The case $d=2:$ $\psi_{n,k}^{(\af)}(r;c)\triangleq r^n\phi_{k}^{\af,n}(2r^2-1;c)$ with
$\af=0.$}
\label{tab2}
{\small
\begin{tabular}{llllllll}
  \hline
$r$ &  $c$ & \vs $\sqrt{r}\psi_{0,0}^{(0)}(r;c)$ &\vs\;\; $\psi_{0,0}^{(0)}(r;c)$({\cite[Table II]{Slep64}})&\;\; $T_{0,0}(r;c)$(\cite{Amodio13}) \\
  \hline\hline
$0.1$ & $1$&   $\;\;\; 4.746377794187660e-01$&\vs$\;\;\;4.74638e-01  $& $\;\;\; 4.7463759e-01 $\\
$0.2$ & $1$&  $\;\;\; 6.687764918417400e-01$&\vs$\;\;\;6.68776e-01 $&$\;\;\;6.6877647e-01 $\\
$0.3$ & $1$&  $\;\;\;8.140701934306384e-01 $&\vs$\;\;\;8.14070e-01 $&$\;\;\;8.1407035e-01 $\\
$0.5$ & $1$&   $\;\;\; 1.030440043954435e+00$&\vs$\;\;\;1.03044e+00 $&$\;\;\;1.0304405e+00 $\\
$0.8$ & $1$&   $\;\;\;1.241572788028936e+00 $&\vs$\;\;\;1.24157e+00 $&$\;\;\;1.2415737e+00 $\\
$1$ & $1$  &   $\;\;\; 1.326266154743105e+00$&\vs$\;\;\;1.32627e+00$&$\;\;\;1.3262673e+00 $\\
\hline
$r$ &  $c$ & \vs $\sqrt{r}\psi_{2,3}^{(0)}(r;c)$ &\vs\;\; $\psi_{2,3}^{(0)}(r;c)${\cite[Table II]{Slep64}}&\;\; $T_{2,3}(r;c)$(\cite{Amodio13}) \\
  \hline\hline

$0.4$ & $1$&  $\;\;\;1.222417855043133e+00 $&\vs$\;\;\;1.22242e+00$&$\;\;\;1.2224159e+00$\\
$0.5$ & $1$&  $\;\;\; 5.021247272944478e-01$&\vs$\;\;\;5.02125e-01$&$\;\;\;5.0212393e-01$\\
$0.6$ & $1$&  $-7.286501244358855e-01 $&\vs$-7.28650e-01$&$-7.2864896e-01$\\

$0.8$ & $2$&  $ -9.788937888170204e-02 $&\vs$-9.78895e-02$&$-9.7889226e-02$\\
$0.9$ & $2$&  $\;\;\; 1.731187946953650e+00 $&\vs$\;\;\;1.73119e+00$&$\;\;\;1.7311852e+00$\\
$1$ & $2$  &  $-4.239904747895277e+00 $     &\vs$-4.23990e+00 $&$-4.2398981e+00 $\\
\hline
\end{tabular} }
\end{table}

%

\begin{figure}
  \centering
  \subfigure[Graph of $\chi_{n,k}^{(0)}(c)$ with $c=20$ and $d=2$.]{
    \includegraphics[width=2.7in]{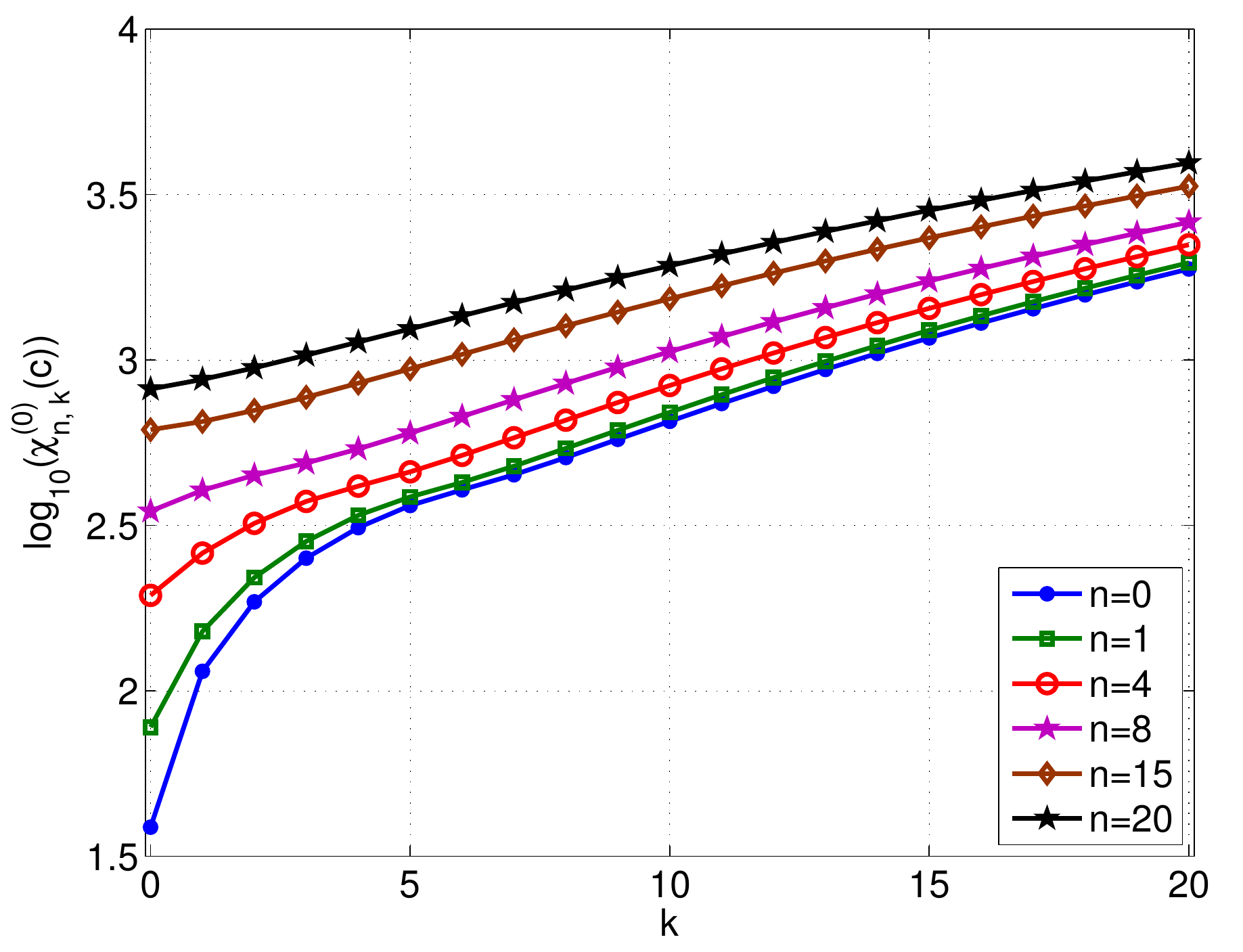}}
  \hspace{0in}
  \subfigure[Graph of $\lambda_{n,k}^{(0)}(c)$ with $c=4$ and $d=2$.]{
    \includegraphics[width=2.7in]{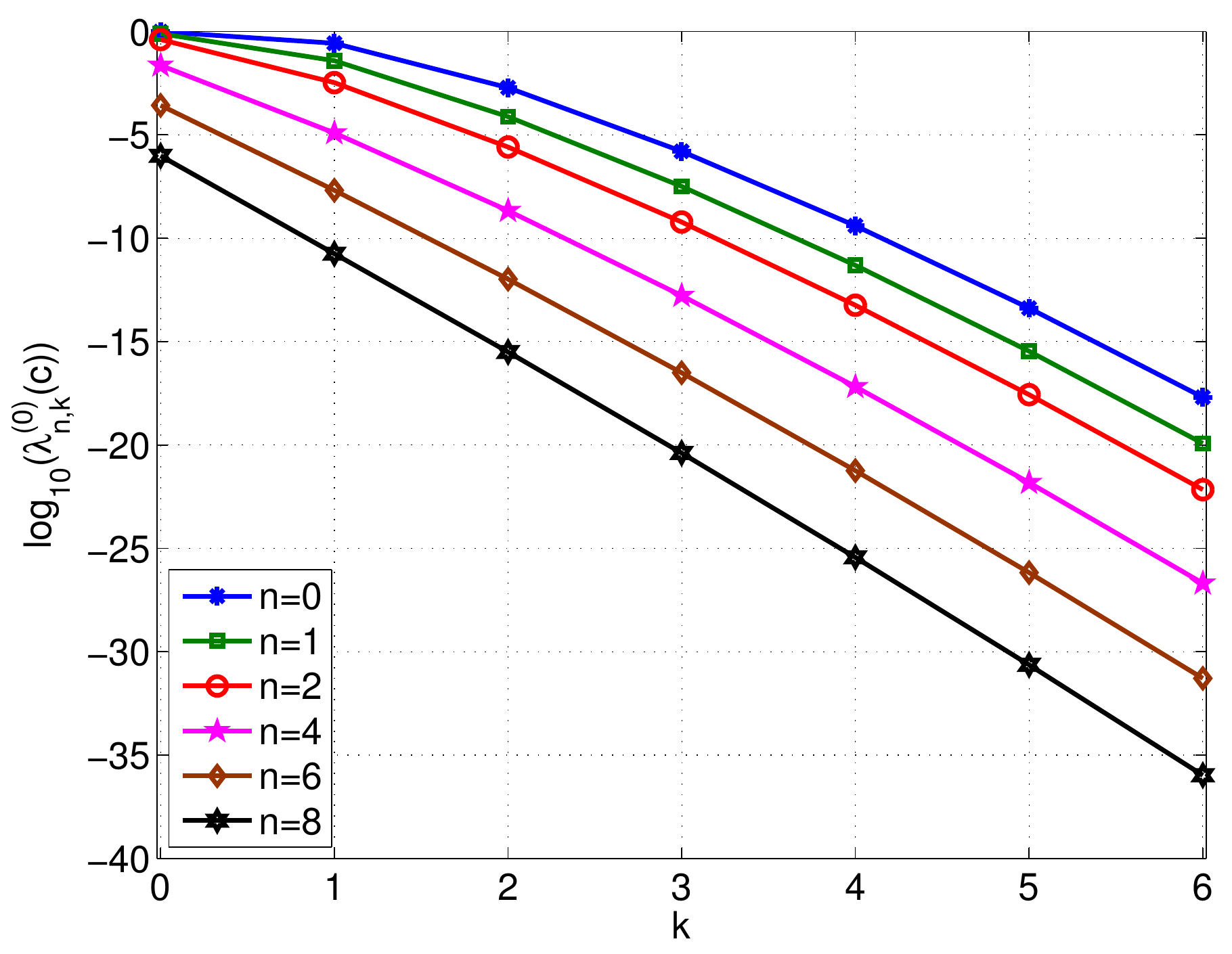}}
  \vfill
  \subfigure[Graph of $\chi_{n,k}^{(1)}(c)$ with $c=10$ and $d=3$.]{
    \includegraphics[width=2.7in]{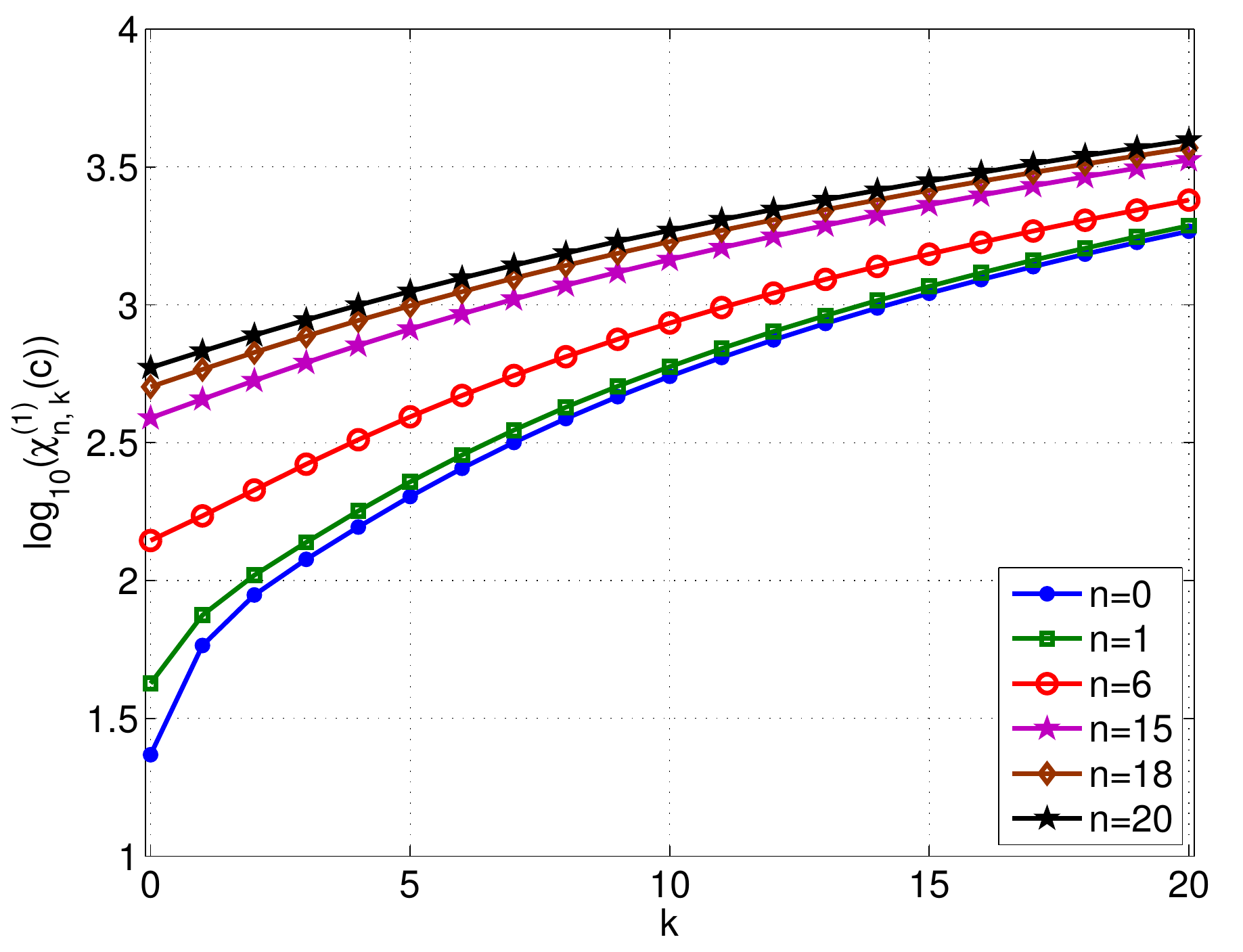}}
  \hspace{0in}
  \subfigure[Graph of $\lambda_{n,k}^{(1)}(c)$ with $c=2$ and $d=3$.]{
    \includegraphics[width=2.7in]{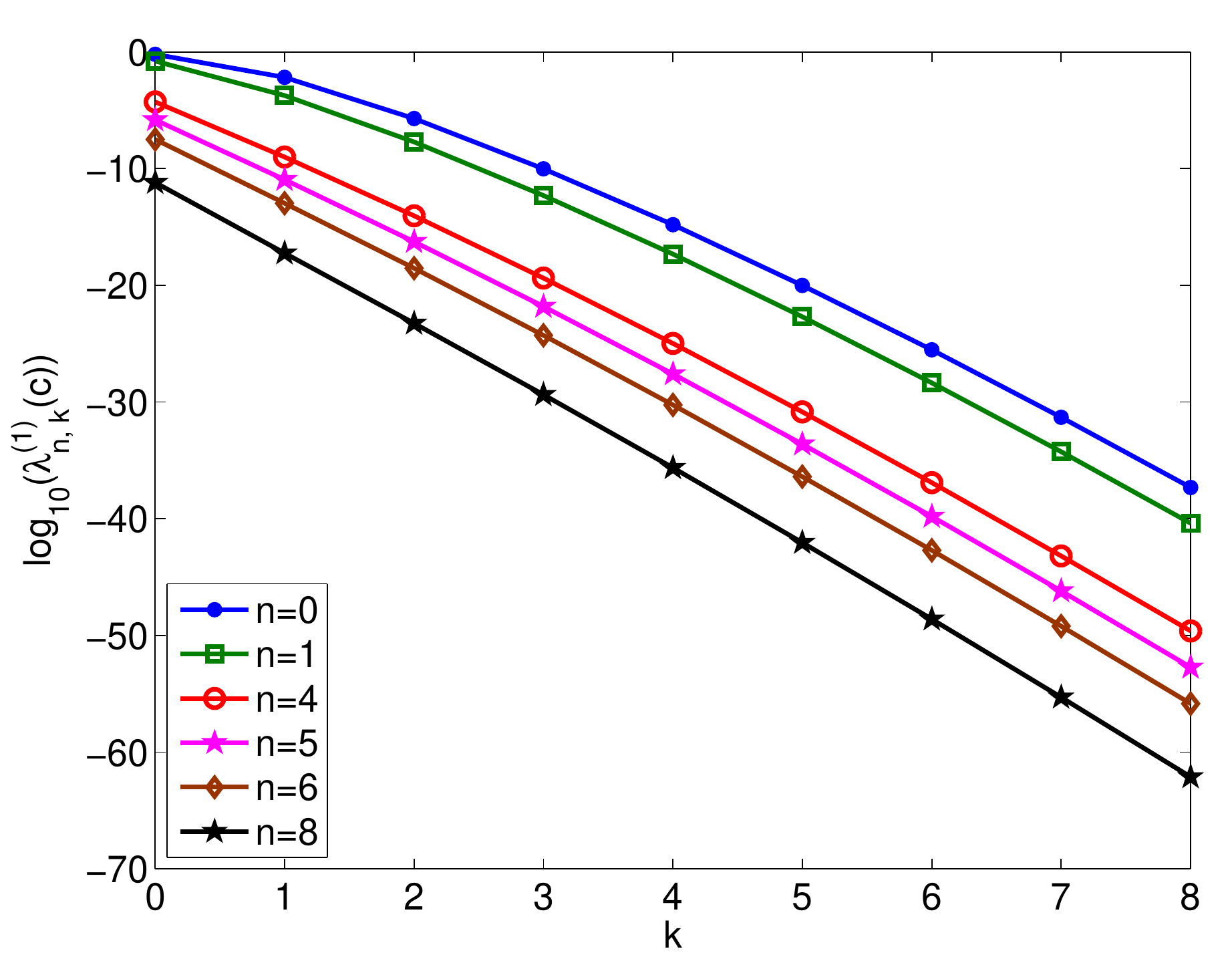}}
  \caption{\small Graphs of $\chi_{n,k}^{(0)}(c)$ and  $\lambda_{n,k}^{(0)}(c).$ }
\label{lamchid2d3}
\end{figure}

In Figure \ref{lamchid2d3} (a)-(b),   we plot $\chi_{n,k}^{(0)}(c)$  and
$\lambda_{n,k}^{(0)}(c)$ versus $k$ in the $2$-dimensional case. It indicates that, for fixed $n$ and $c>0$, $\chi_{n,k}^{(0)}$ becomes larger as $k$ increases,
while $\lambda_{n,k}^{(0)}$ decays exponentially as $k$ grows.  In Figure \ref{psiD2D3} (a)-(b), we depict
the radial component $\psi_{n,k}^{(0)}(r;c)\triangleq r^n\phi_{k}^{0,n}(2r^2-1;c)$ versus $r\in [0,1.5]$ for $n=0,2, k=0,1,2,3$ and $c=2, 10.$ Figures \ref{surfpsi2d1} -\ref{surfpsi2d2}  show  surfaces and contours of $\psi_{k,l}^{\af, n}(x;c)$ with different $c,k,n$ and $l$ with $d=2,\af=0.$


\begin{figure}
  \centering
  \subfigure[Graph of $\psi_{n,k}^{(0)}(r;c)$ with  $c=2$ and $d=2$.]{
    \includegraphics[width=2.7in]{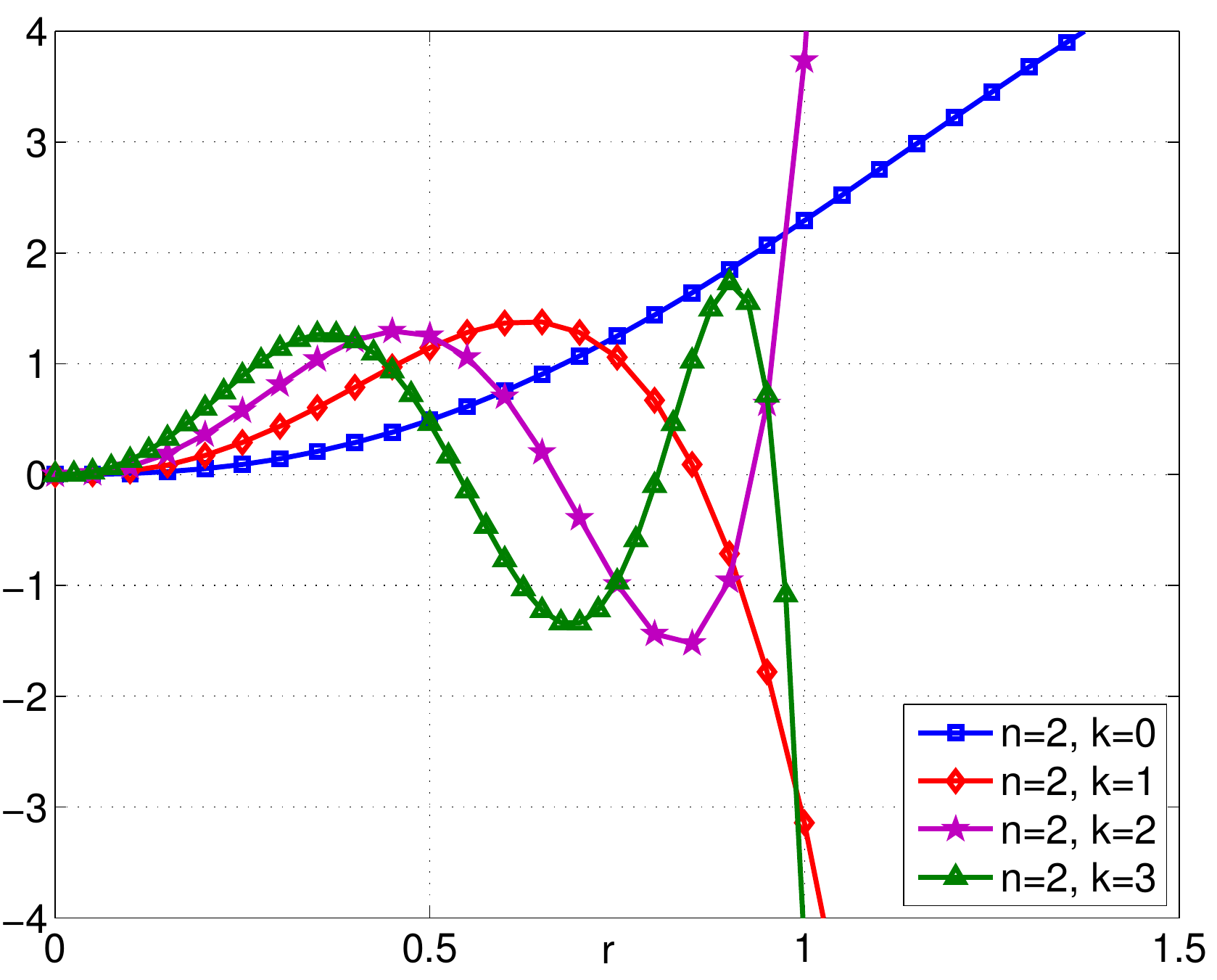}}
  \hspace{0in}
  \subfigure[Graph of $\psi_{n,k}^{(0)}(r;c)$ with  $c=10$ and $d=2$.]{
    \includegraphics[width=2.7in]{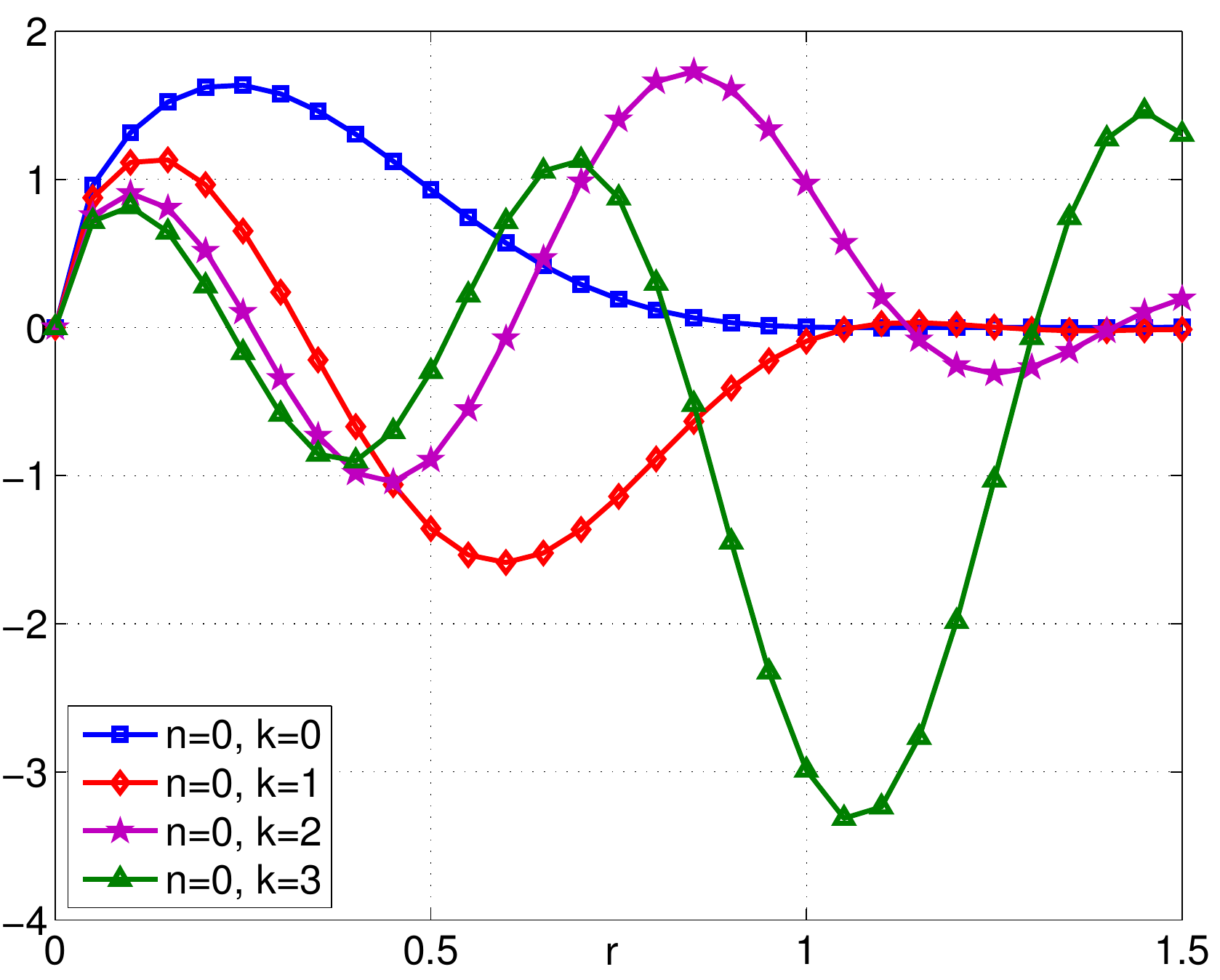}}
  \vfill
  \subfigure[Graph of $\psi_{n,k}^{(0)}(r;c)$ with  $c=2$ and $d=3$.]{
    \includegraphics[width=2.7in]{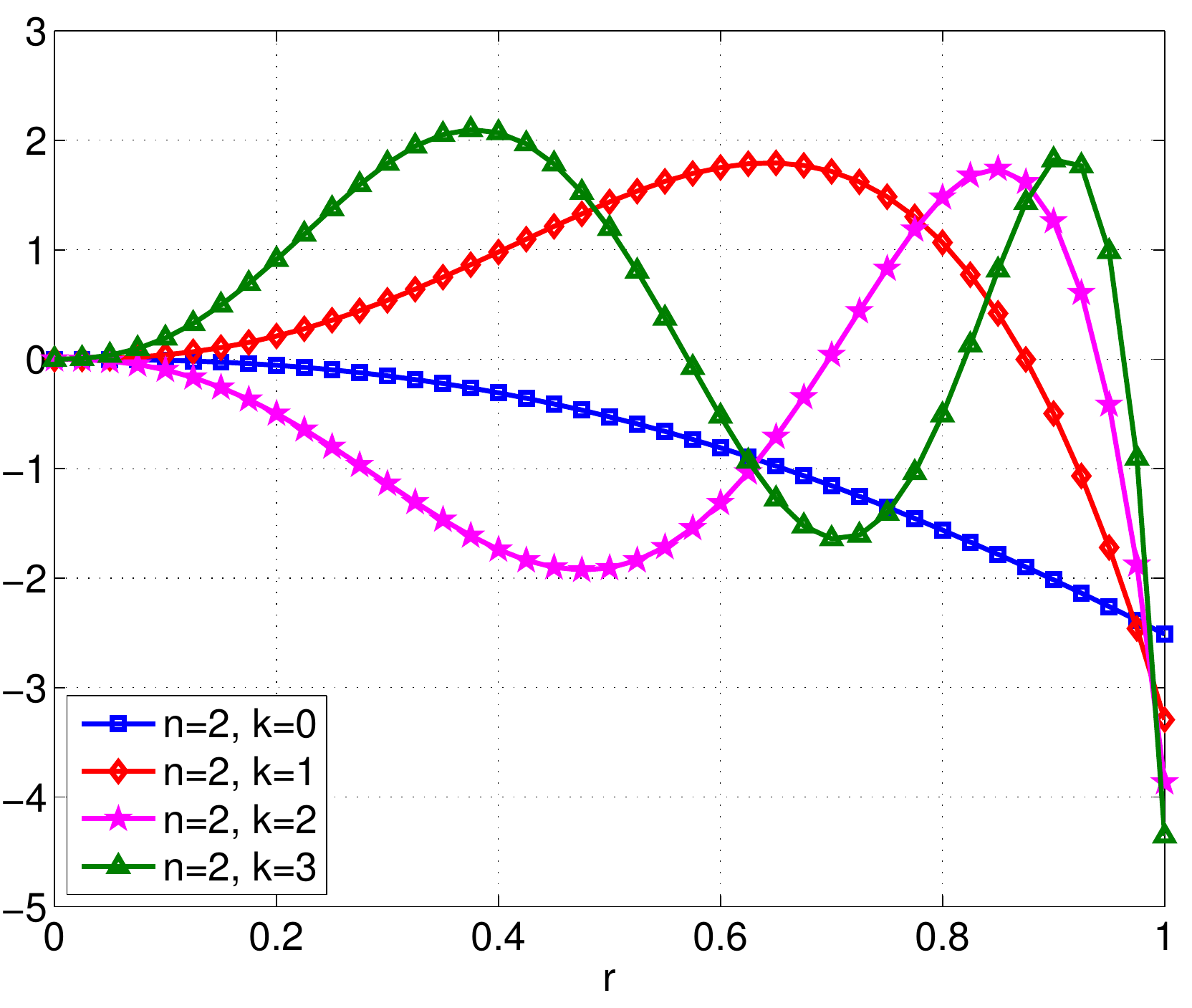}}
  \hspace{0in}
  \subfigure[Graph of $\psi_{n,k}^{(1)}(r;c)$ with  $c=10$ and $d=3$.]{
    \includegraphics[width=2.7in]{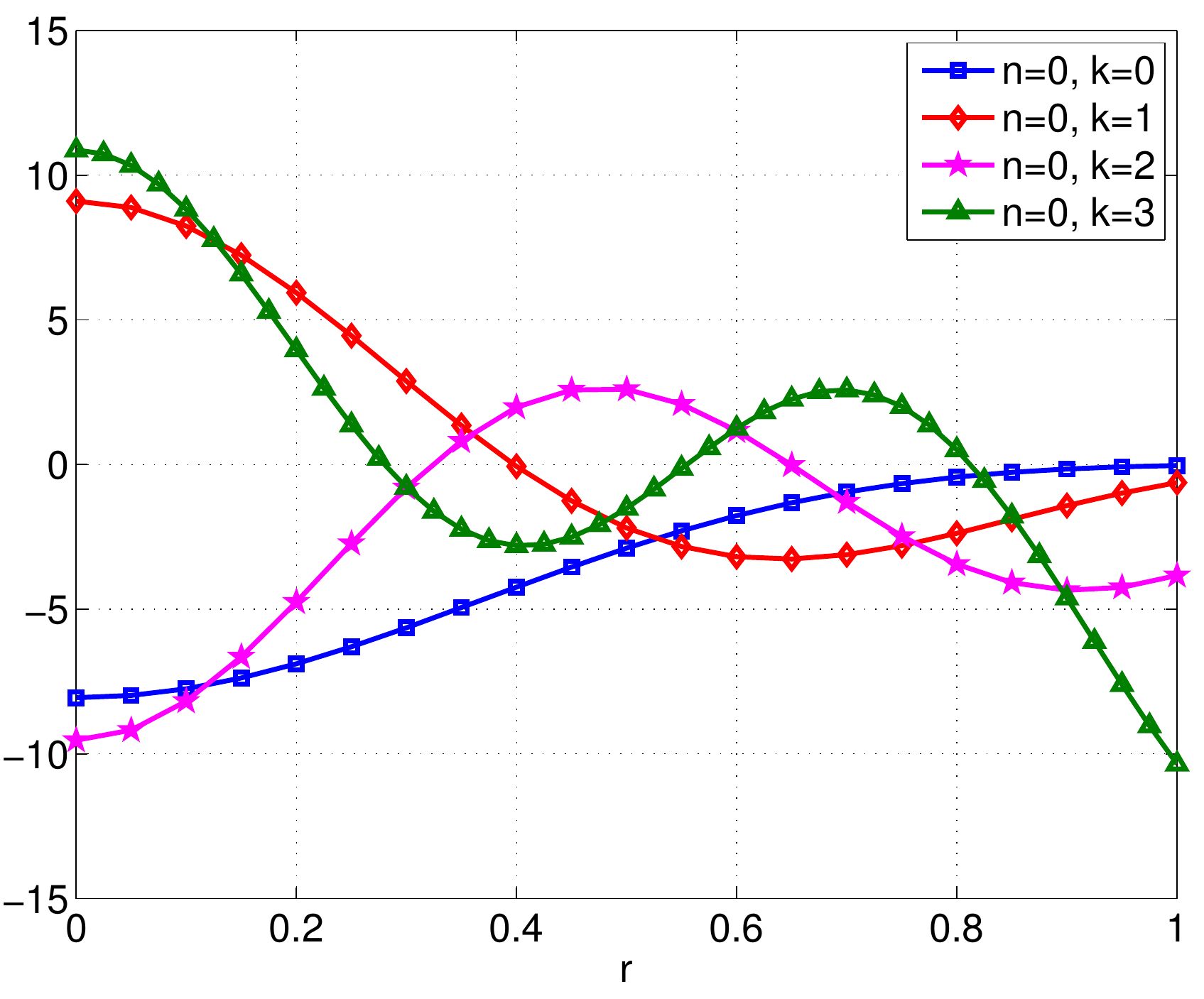}}
  \caption{\small  Graphs of $\psi_{n,k}^{(\af)}(r;c)$ in $2$-dimension and $3$-dimension. }
\label{psiD2D3}
\end{figure}

%
\begin{figure}
  \centering
  \subfigure[$(\af, n, k, l)=(0,1,0,1).$]{
    \includegraphics[width=2.7in]{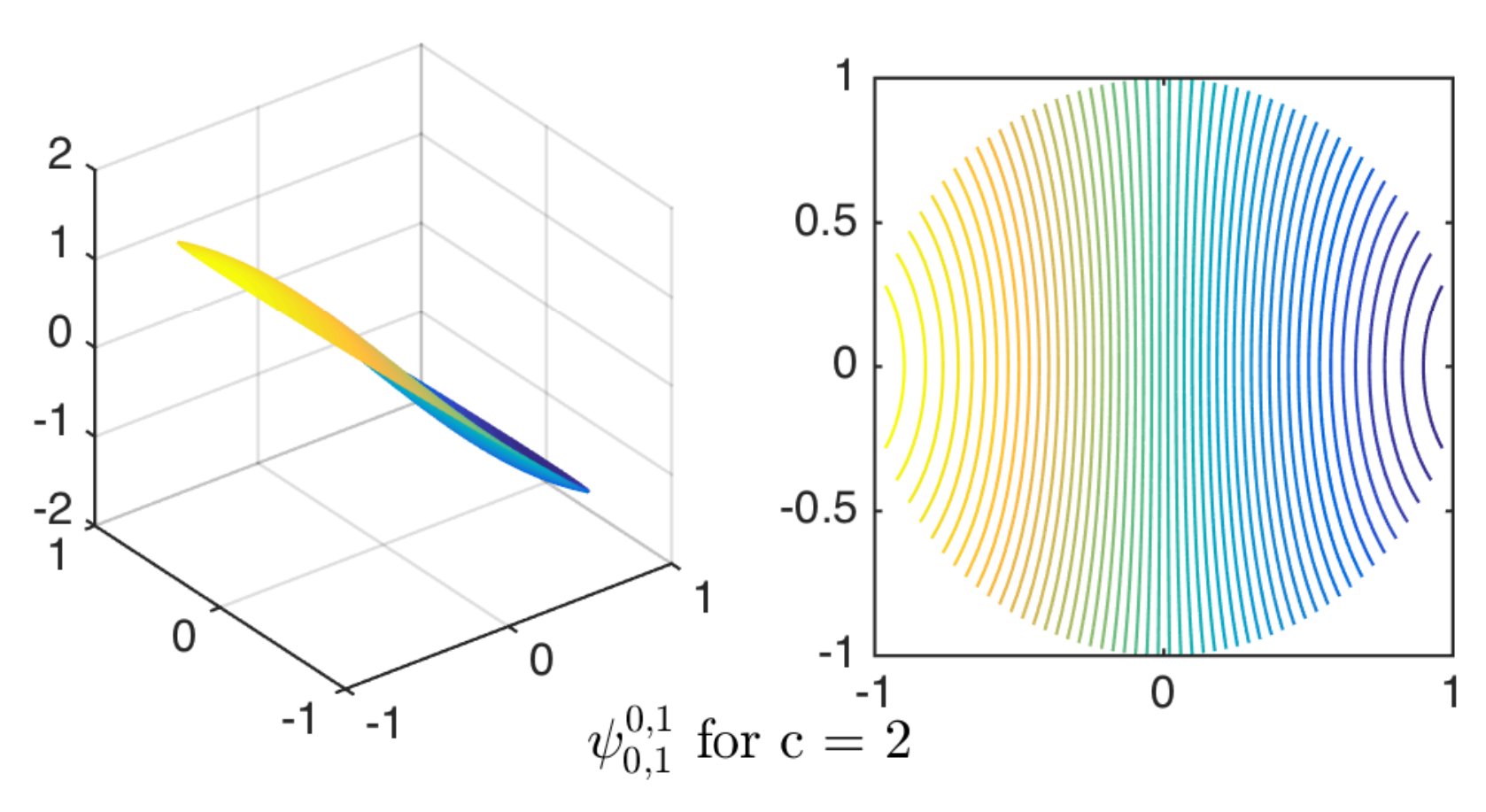}}
  \hspace{0in}
  \subfigure[$(\af, n, k, l)=(0,1,0,2).$]{
    \includegraphics[width=2.7in]{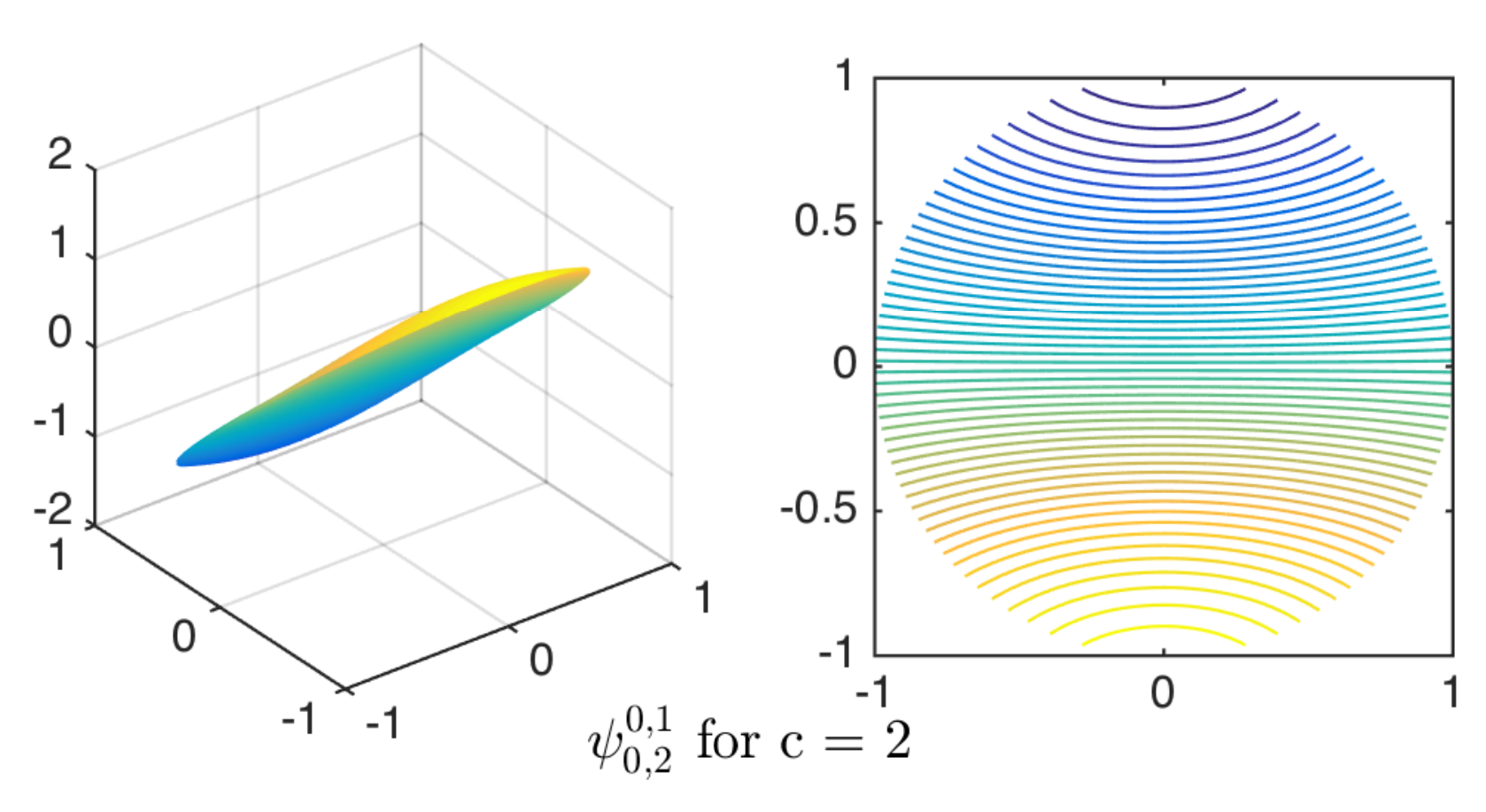}}
  \vfill
  \subfigure[$(\af, n, k, l)=(0,2,0,1).$]{
    \includegraphics[width=2.7in]{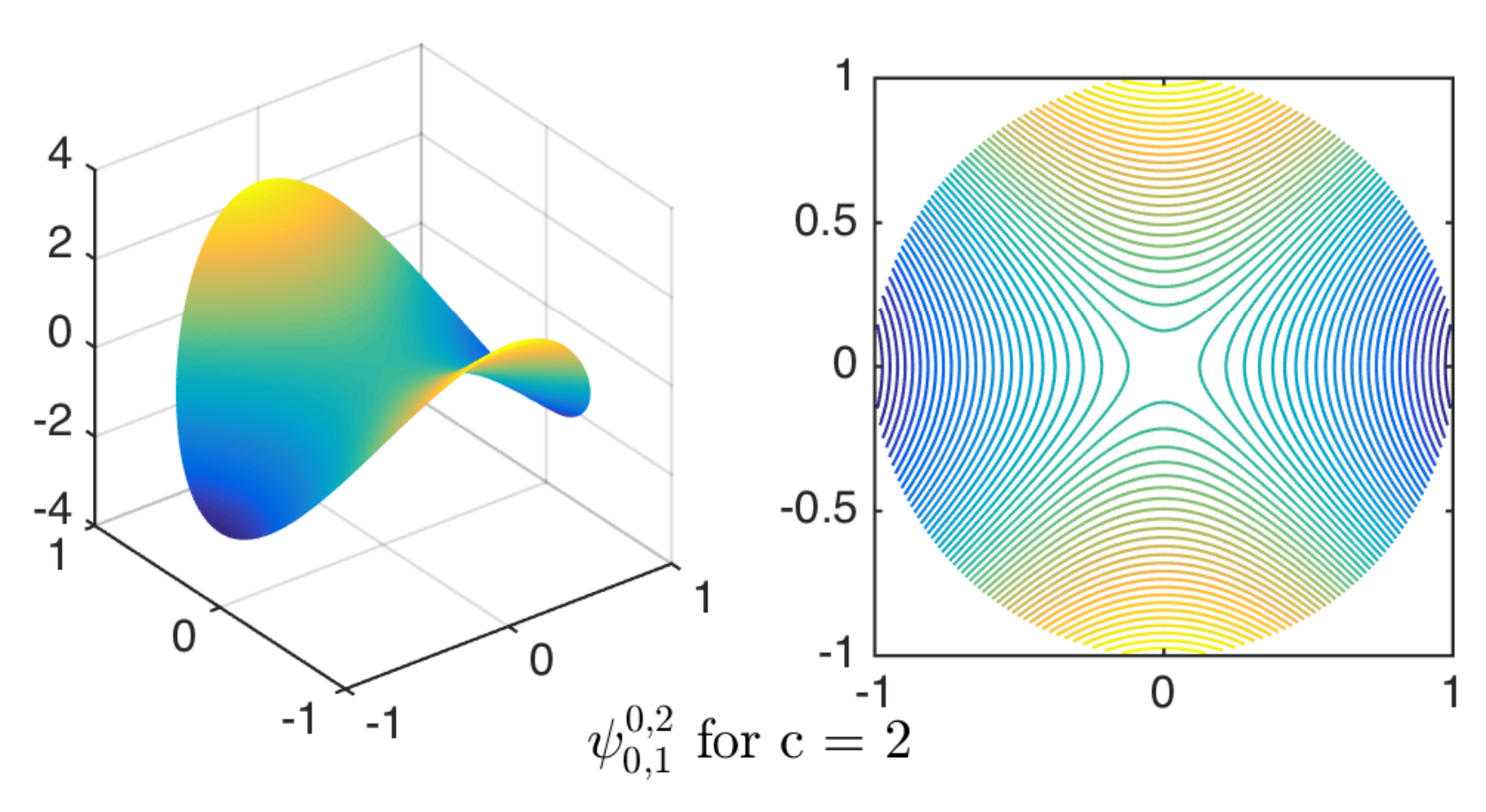}}
  \hspace{0in}
  \subfigure[$(\af, n, k, l)=(0,2,0,2).$]{
    \includegraphics[width=2.7in]{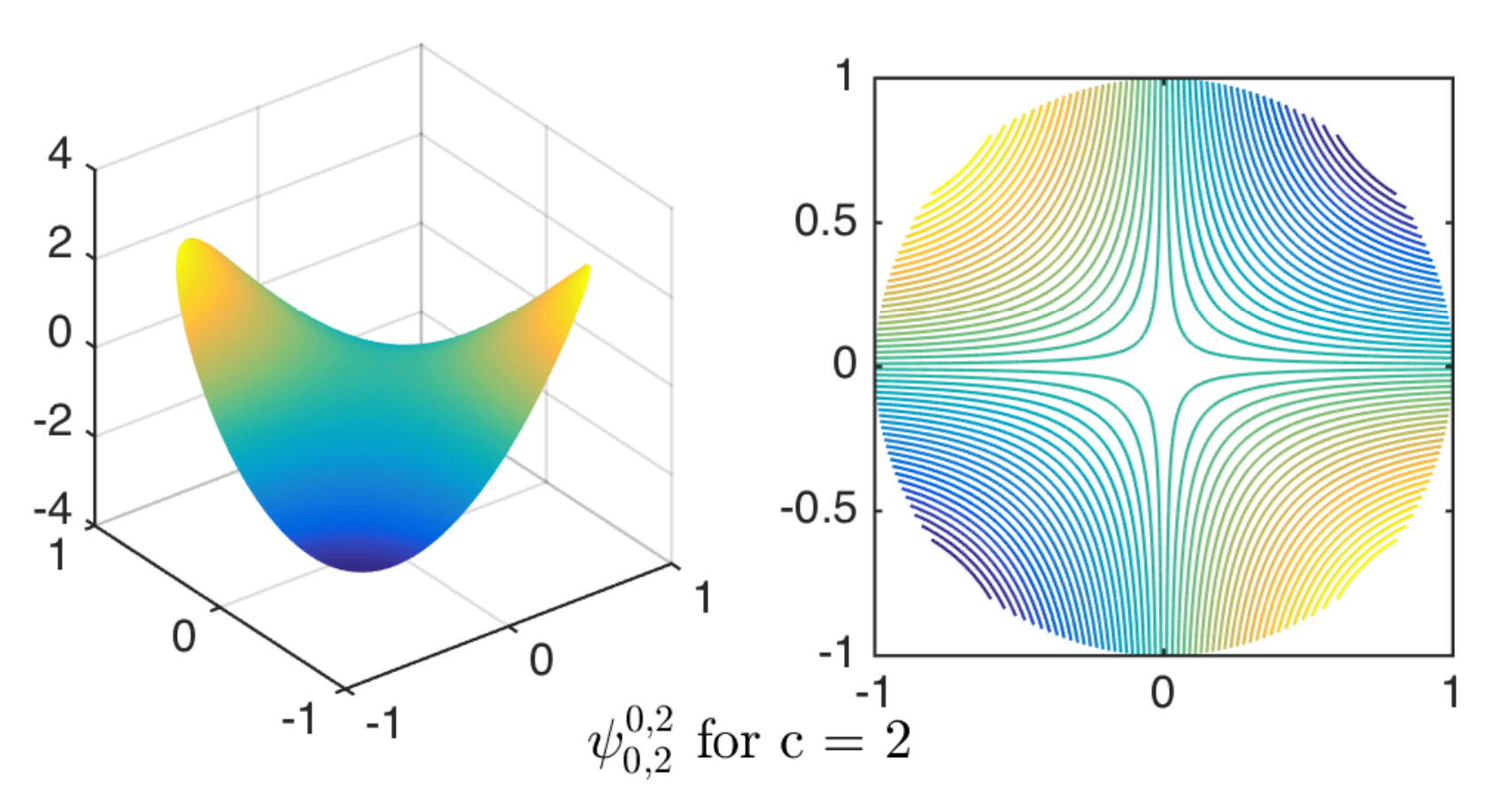}}
  \caption{Eigenfunctions  $\psi_{k,l}^{\af,n}$ with $c=2$ in $2$-dimension.}\label{surfpsi2d1}
\end{figure}

\begin{figure}
  \centering
  \subfigure[$(\af, n, k, l)=(0,0,0,1).$]{
    \includegraphics[width=2.7in]{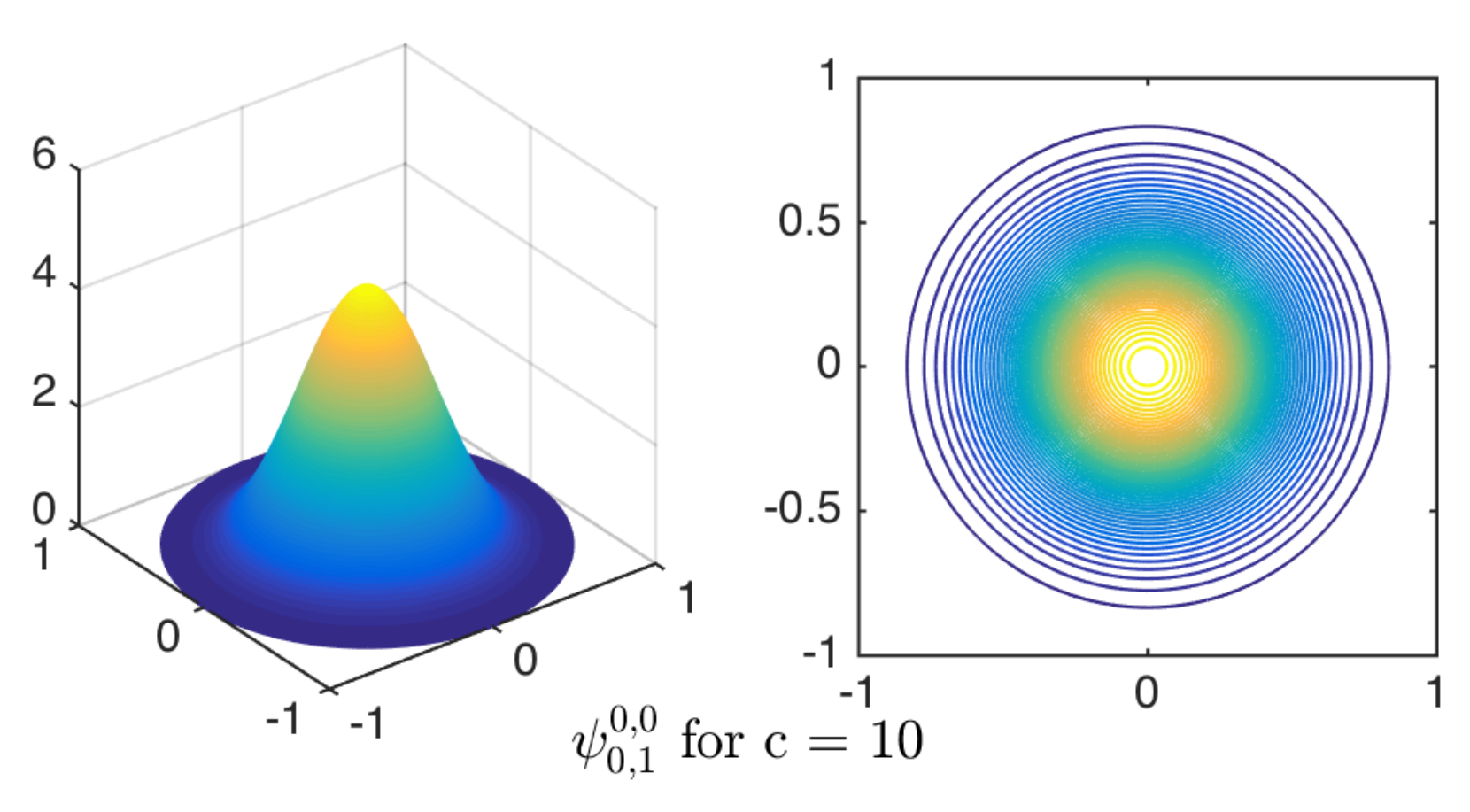}}
  \hspace{0in}
  \subfigure[$(\af, n, k, l)=(0,0,1,1).$]{
    \includegraphics[width=2.7in]{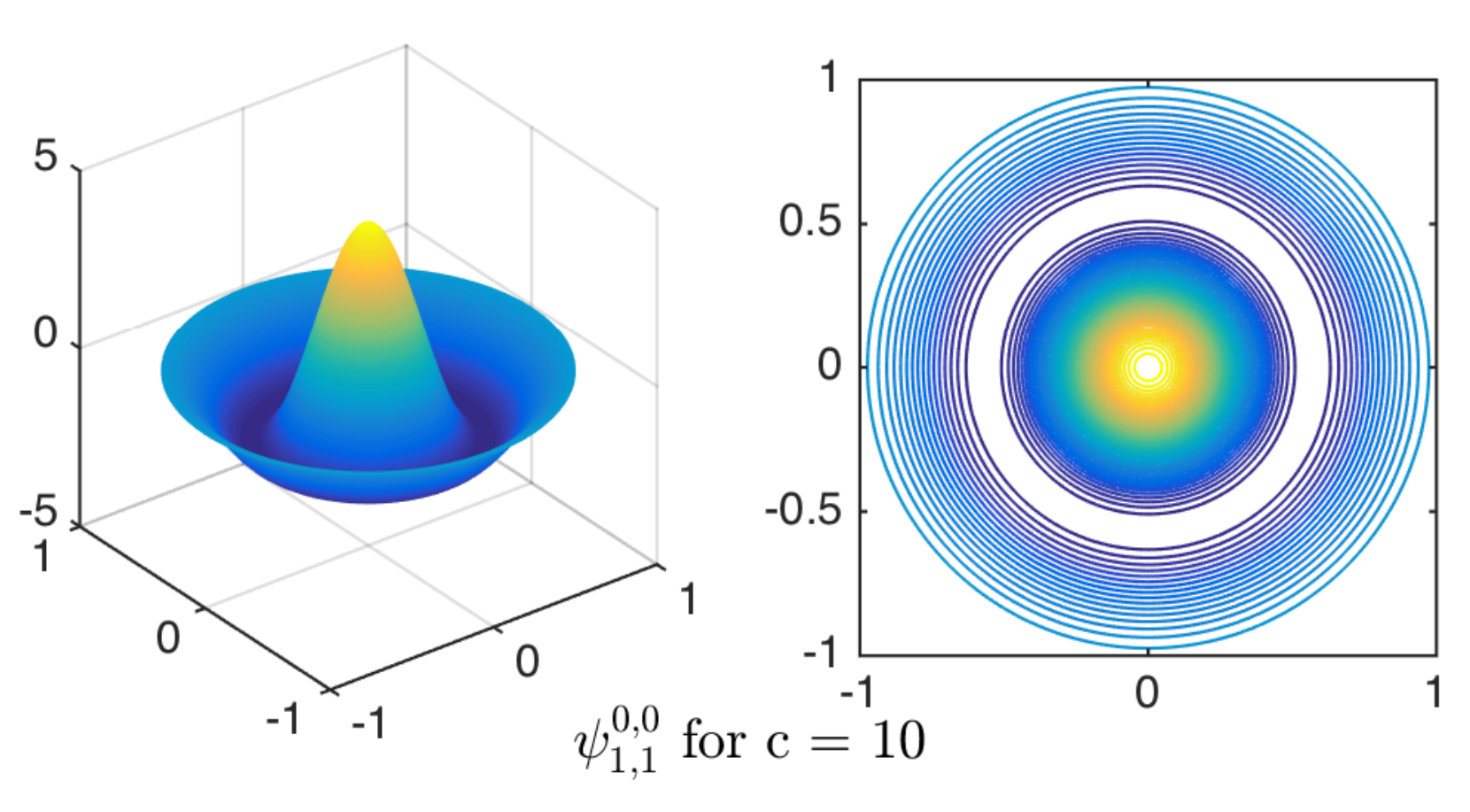}}
  \vfill
  \subfigure[$(\af, n, k, l)=(0,0,2,1).$]{
    \includegraphics[width=2.7in]{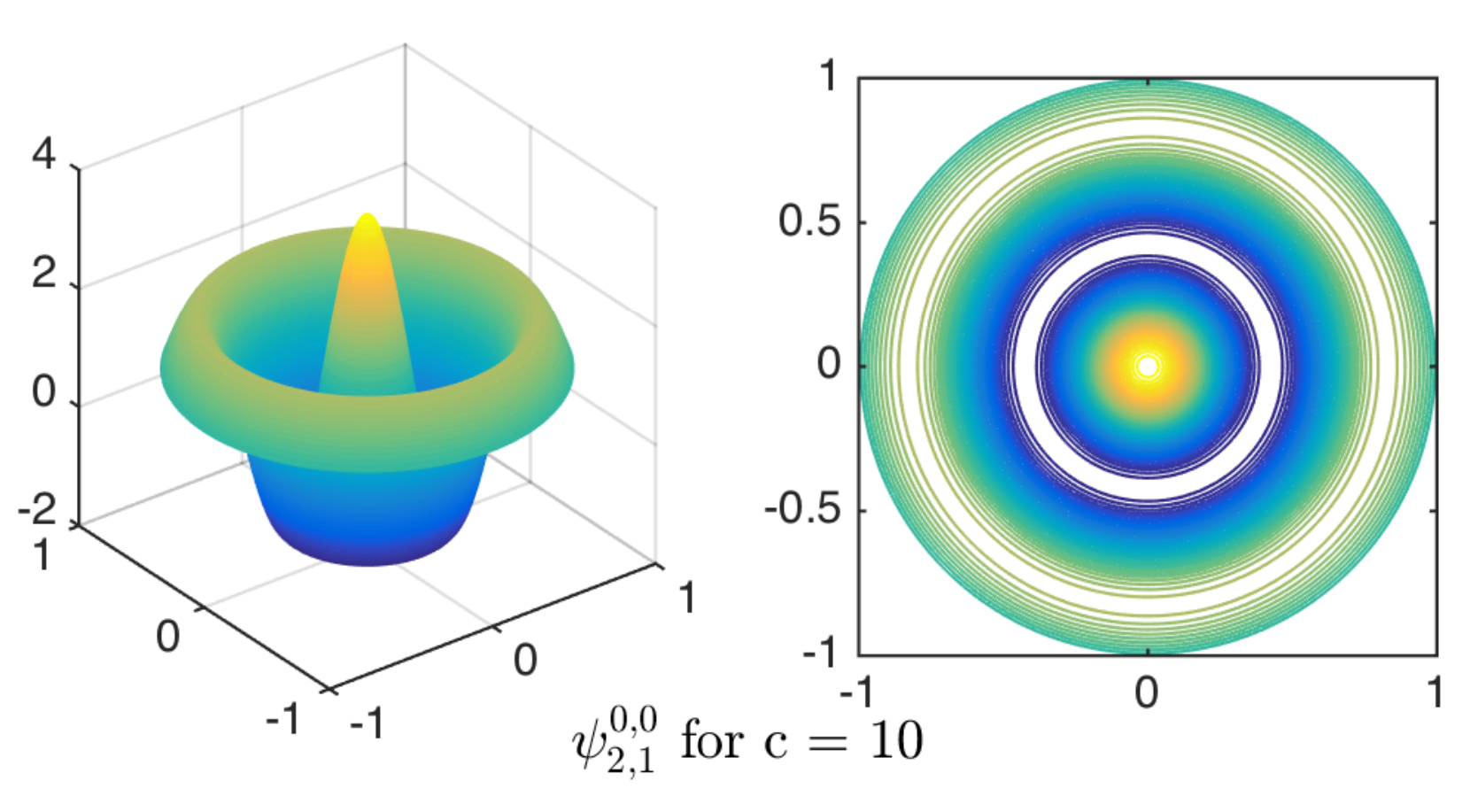}}
  \hspace{0in}
  \subfigure[$(\af, n, k, l)=(0,0,3,1).$]{
    \includegraphics[width=2.7in]{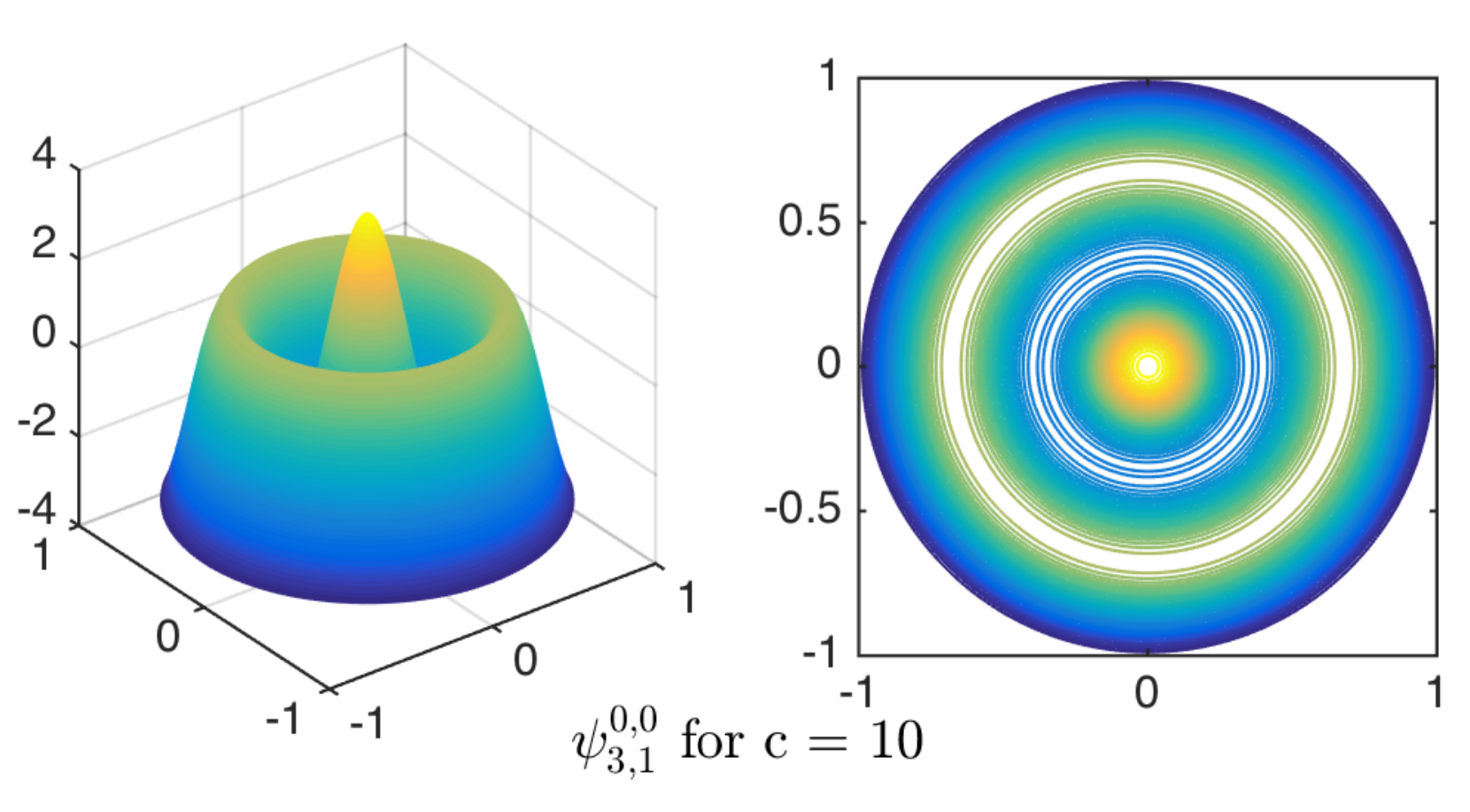}}
  \caption{Eigenfunctions  $\psi_{k,l}^{\af,n}$  with $c=10$ in $2$-dimension.}\label{surfpsi2d2}
\end{figure}

\begin{table}[!th] \centering
\caption{\small The case $d=3$: $\chi_{n,k}^{(1)}(c)$ and $\lambda_{n,k}^{(1)}(c)$ }
\label{tab3}
{\small
\begin{tabular}{llllllll}
  \hline
   $c$ & $n$ &  $k$ &\vs $\chi_{n,k}^{(1)}(c)$ &\vs $\lambda_{n,k}^{(1)}(c)$ \\
  \hline\hline
$0.1$& $0$  &  $0$ & $4.285325573224633e-03$\vs&$ 1.675003294483135e+00$\\
$0.5$& $0$  &   $0$ & $ 1.069001304053325e-01$&$ 1.662771473208847e+00 $\\
$1$& $0$  &  $0$ & $4.246991437751348e-01$&$1.625460618463697e+00$\\
$4$& $0$  &  $0$ & $ 5.948719383823520e+00 $&$ 1.102600593723482e+00$\\
$10$& $0$  &   $0$ & $ 2.333891804161449e+01 $&$4.186593008319554e-01$\\
\hline
$2$& $1$  &   $0$ & $  8.182057327887621e+00$&$4.238871423701353e-01 $\\
$2$& $1$  &   $1$ & $  2.609221756616164e+01$&$ 8.233874011948259e-03$\\
$2$& $1$  &   $2$ & $ 5.205150235186056e+01$&$  6.516928432939569e-05$\\
$2$& $1$  &   $3$ & $8.603255633419086e+01   $&$  2.809367682507114e-07$\\
$2$& $1$  &   $4$ & $ 1.280223716202459e+02  $&$ 7.613268689084687e-10$\\
\hline
\end{tabular} }
\end{table}

\begin{table}[!thb] \centering
\caption{\small The case $d=3:$ $\psi_{n,k}^{(\af)}(r;c)\triangleq r^n\phi_{k}^{\af,n}(2r^2-1;c)$ with
$\af=0, 1, 2.$}
\label{tab4}
{\small
\begin{tabular}{llllllll}
  \hline
$r$ &  $c$ & $n$ &  $k$ &\vs $\psi_{n,k}^{(0)}(r;c)$ &\vs $\psi_{n,k}^{(1)}(r;c)$&\vs $\psi_{n,k}^{(2)}(r;c)$ \\
  \hline\hline
$0.1$ & $1$& $0$  &  $0$ & $\;\;\;  5.805625733654062e-01 $&$\;\;\;  2.820561183868252e+00 $&$\;\;\; 3.687764193662462e+00$\\

$0.2$ & $1$& $0$  &  $0$ & $\;\;\; 8.186066482900428e-01 $&$\;\;\;2.814575764166440e+00 $&$\;\;\;  3.681662607508843e+00 $\\
$0.5$ & $1$& $0$  &  $0$ & $\;\;\;1.267632861585855e+00$&$\;\;\; 2.772954660597707e+00  $&$\;\;\; 3.639182765466543e+00 $\\
$1$ & $1$  & $0$  &  $0$ & $\;\;\; 1.662390750491349e+00 $&$\;\;\;2.628204021066972e+00 $&$\;\;\;3.490731274213273e+00$\\
$1.3$ & $1$& $0$  &  $0$ & $\;\;\;1.765639810965165e+00 $&$\;\;\;2.500277467362624e+00 $&$\;\;\;3.358563656867405e+00 $\\
$2$ & $2$& $0$  &  $0$ & $\;\;\; 3.553627999772212e-01 $&$\;\;\;8.545596995365403e-01 $&$\;\;\;1.510596282792738e+00$\\
\hline
$0.1$ & $1$& $2$  &  $3$ & $-1.893124346916359e-01 $&$-7.943270542522487e-01  $&$ -1.008278981214814e+00 $\\
$0.2$ & $1$& $2$  &  $3$ & $-8.958937078881810e-01 $&$-2.580441975019594e+00 $&$-3.177610574908396e+00 $\\
$0.5$ & $1$& $2$  &  $3$ & $ -1.239366584847178e+00 $&$ -8.701135484764851e-01$&$\;\;\;3.812964021006710e-01 $\\
$1$ & $2$  & $2$  &  $3$ & $\;\;\; 4.355438266567036e+00 $     &$\;\;\; 2.314178264971302e+01 $&$\;\;\; 7.372606028015183e+01 $\\
$1.3$ & $2$& $2$  &  $3$ & $\;\;\;  5.467434735434442e+02 $     &$\;\;\;1.160117778266639e+03 $&$\;\;\;  2.449778131304879e+03$\\
$2$ & $2$& $2$  &  $3$ & $\;\;\;4.569351866698169e+04 $     &$\;\;\; 6.922735069954877e+04 $&$\;\;\; 1.335271987655634e+05  $\\
\hline
\end{tabular} }
\end{table}
In Figure \ref{lamchid2d3} (c)-(d), we depict that $\chi_{n,k}^{(1)}(c)$ and $\lambda_{n,k}^{(1)}(c)$ for various $k$ in the $3$-dimensional case. It is clear that $\chi_{n,k}^{(1)}(c)$  (resp. $\lambda_{n,k}^{(1)}(c)$) become larger (resp. smaller) as $k$ increases. Some values of $\chi_{n,k}^{(1)}(c)$ and $\lambda_{n,k}^{(1)}(c)$ for a large set of parameter values
are given in Table \ref{tab3}. We plot in Figure \ref{psiD2D3} (c)-(d)  some samples of the $\psi_{n,k}^{(\alpha)}(r;c)$ with $d=3.$ We tabulate some values of
$\psi_{n,k}^{(\alpha)}(r;c)$ with $d=3$ in Table \ref{tab4} computed by the aforementioned method. Figures \ref{surfpsi3d1}-\ref{surfpsi3d2} visualize  of $\psi_{k,l}^{\af, n}(\bx;c)$ with different $k,l,n,\af$ and $c$ with $d=3.$

%
%
%
%
%


\begin{figure}
  \centering
  \subfigure[$(\af, n, k, l)=(0,0,0,1).$]{
    \includegraphics[width=1.7in]{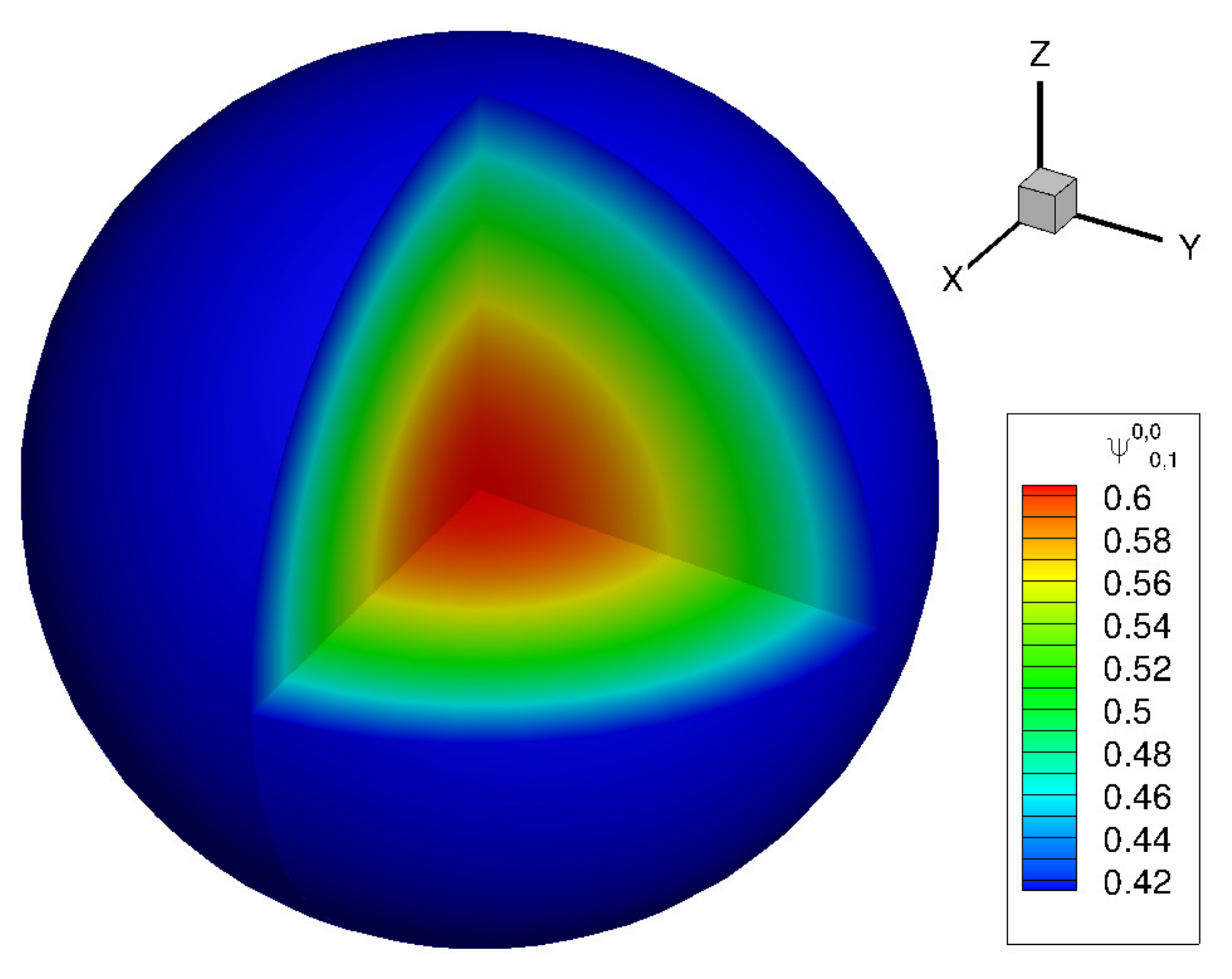}}
  \hspace{0in}
  \subfigure[$(\af, n, k, l)=(0,0,1,1).$]{
    \includegraphics[width=1.7in]{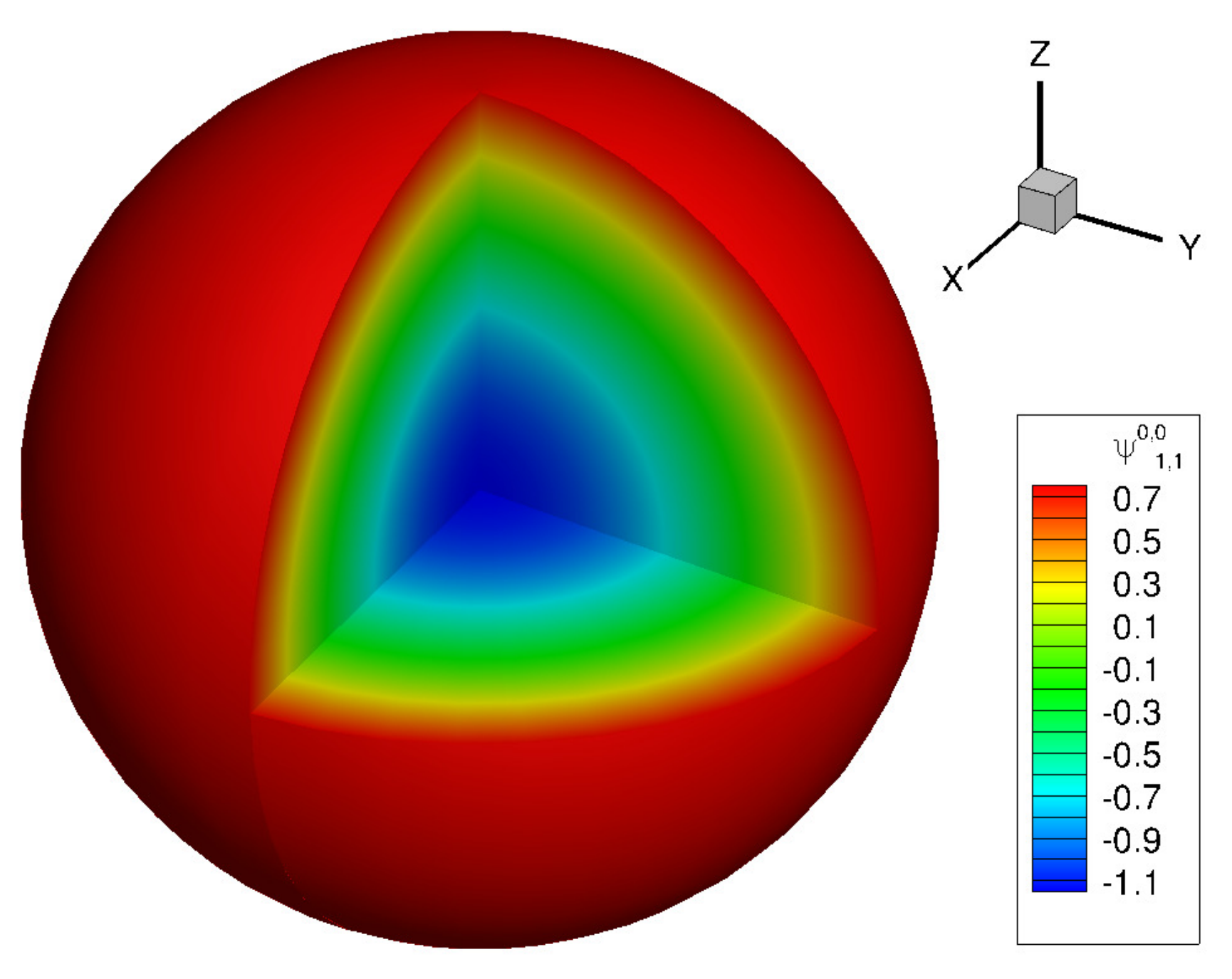}}
  \hspace{0in}
  \subfigure[$(\af, n, k, l)=(0,0,2,1).$]{
    \includegraphics[width=1.7in]{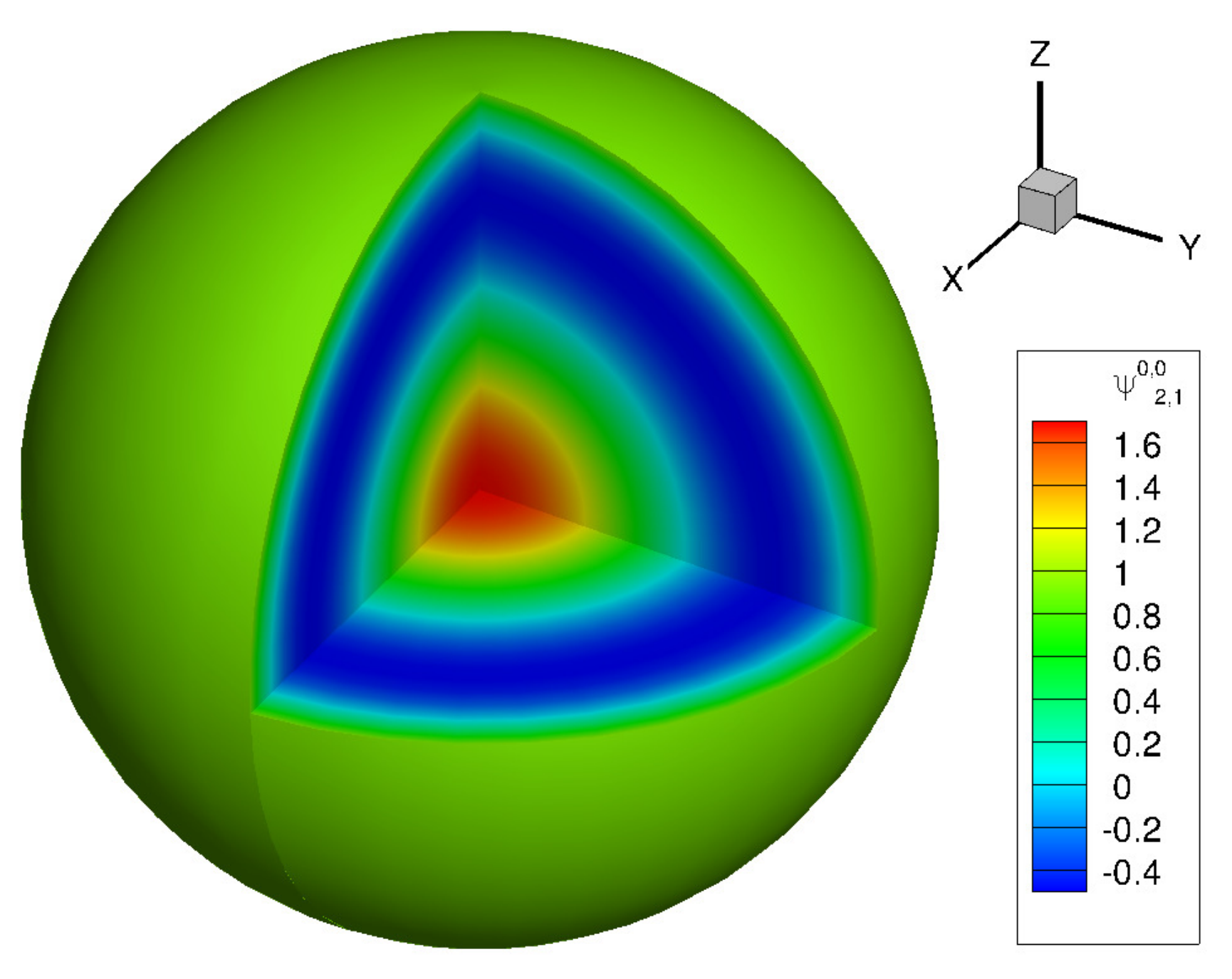}}
  \vfill
  \subfigure[$(\af, n, k, l)=(0,2,0,1).$]{
    \includegraphics[width=1.7in]{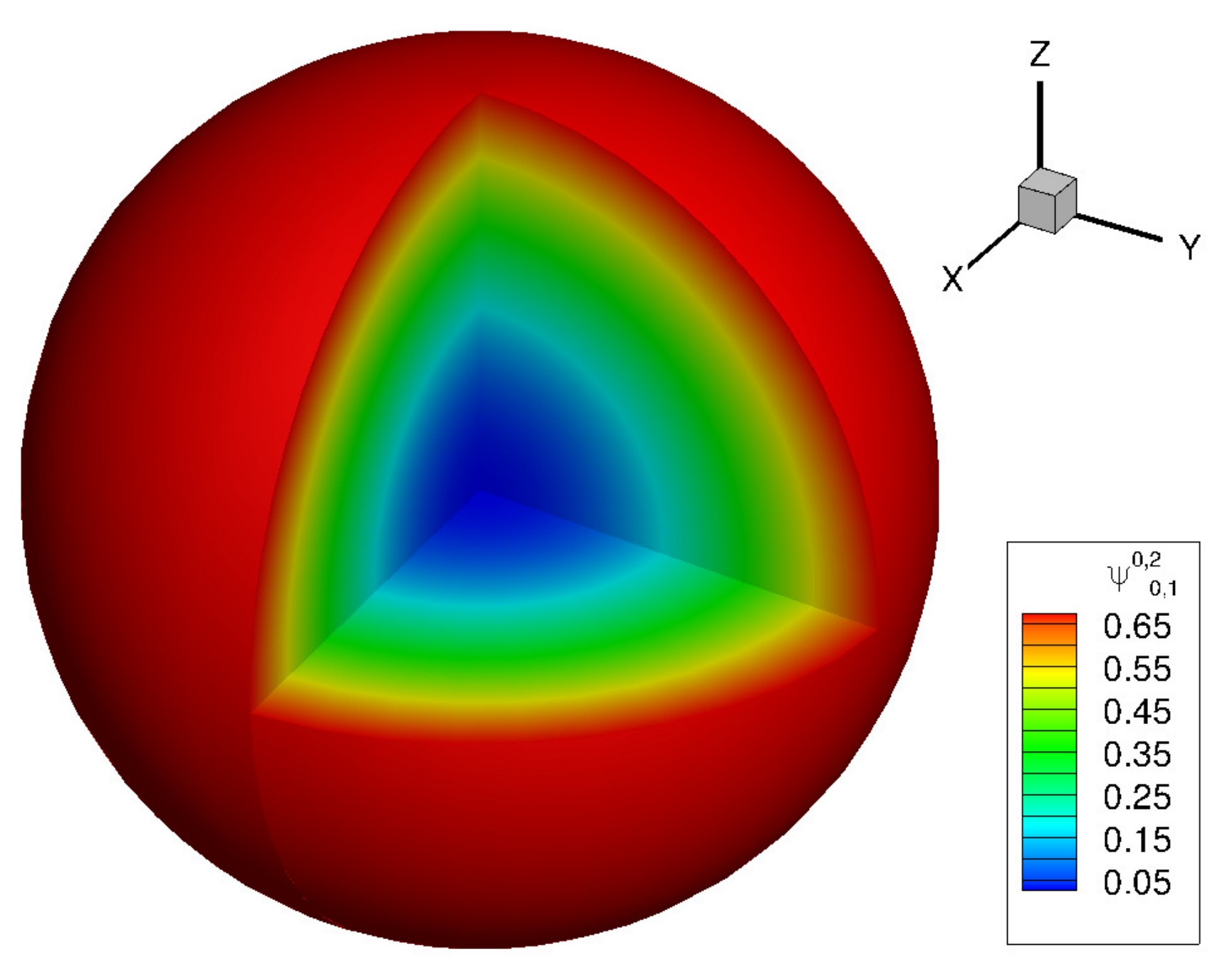}}
  \hspace{0in}
  \subfigure[$(\af, n, k, l)=(0,2,1,1).$]{
    \includegraphics[width=1.7in]{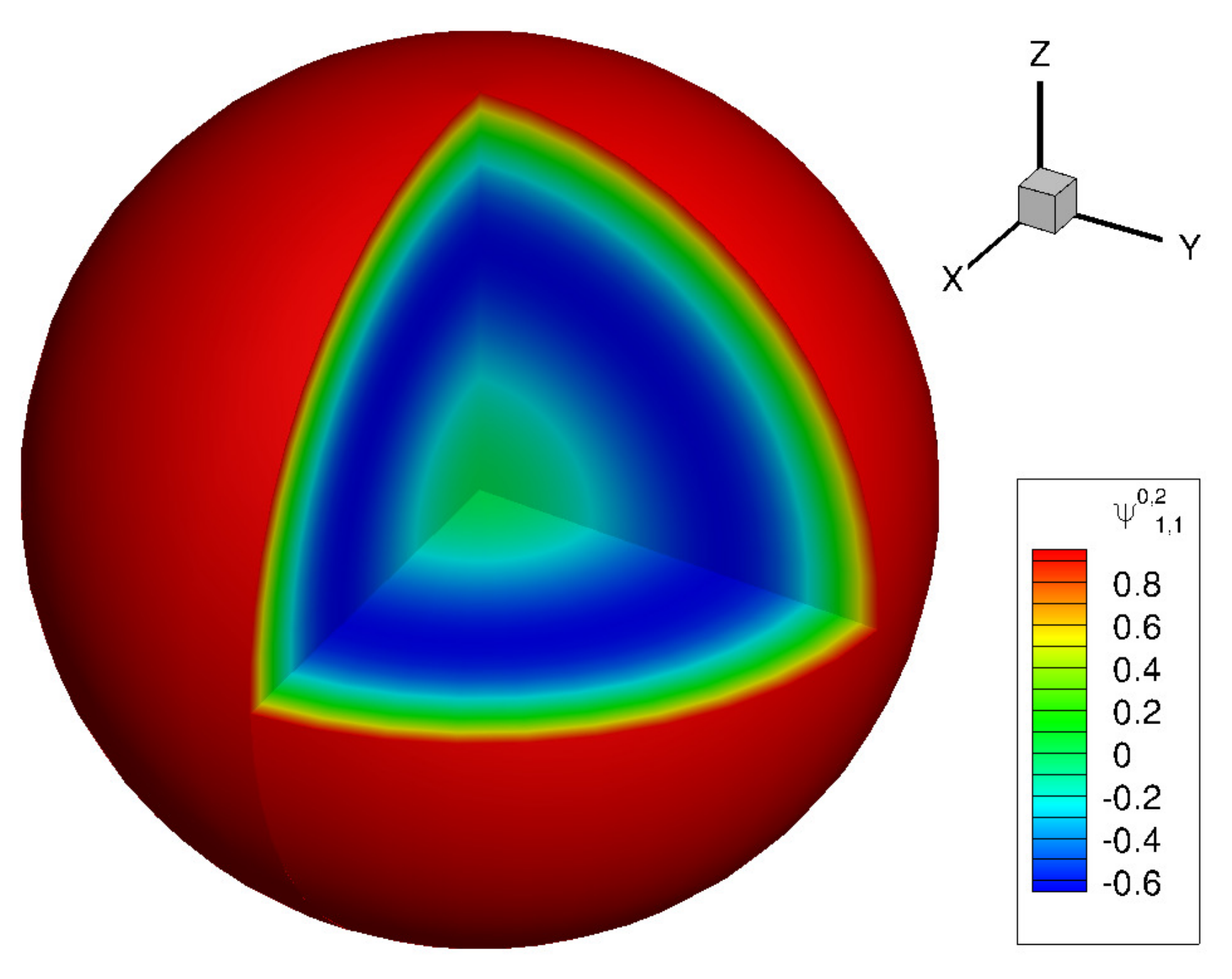}}
    \hspace{0in}
  \subfigure[$(\af, n, k, l)=(0,2,2,1).$]{
    \includegraphics[width=1.7in]{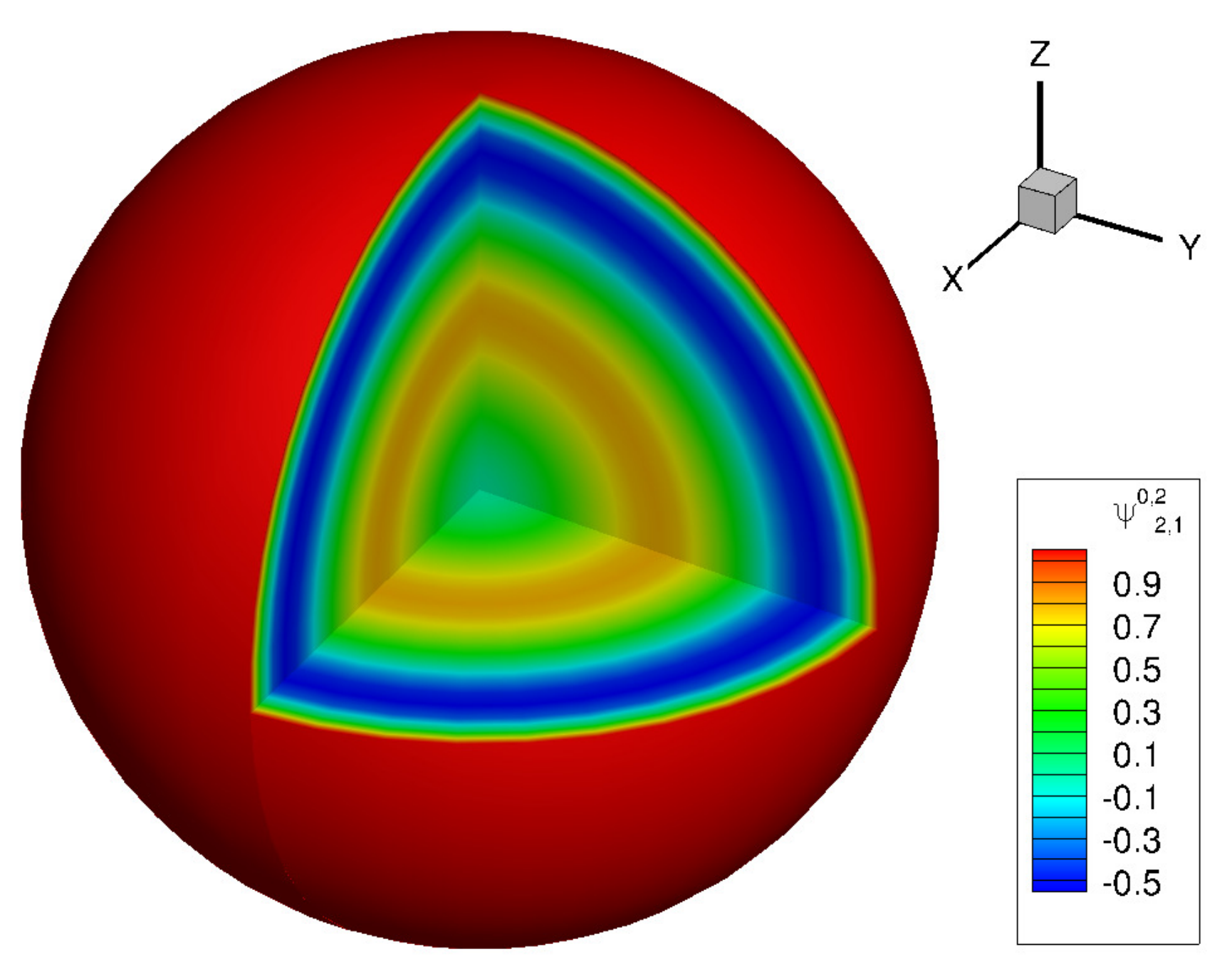}}
  \caption{Eigenfunctions  $\psi_{k,l}^{\af,n}$ with $c=2$ in $3$-dimension.}\label{surfpsi3d1}
\end{figure}

\begin{figure}
  \centering
  \subfigure[$(\af, n, k, l)=(1,0,0,1).$]{
    \includegraphics[width=1.7in]{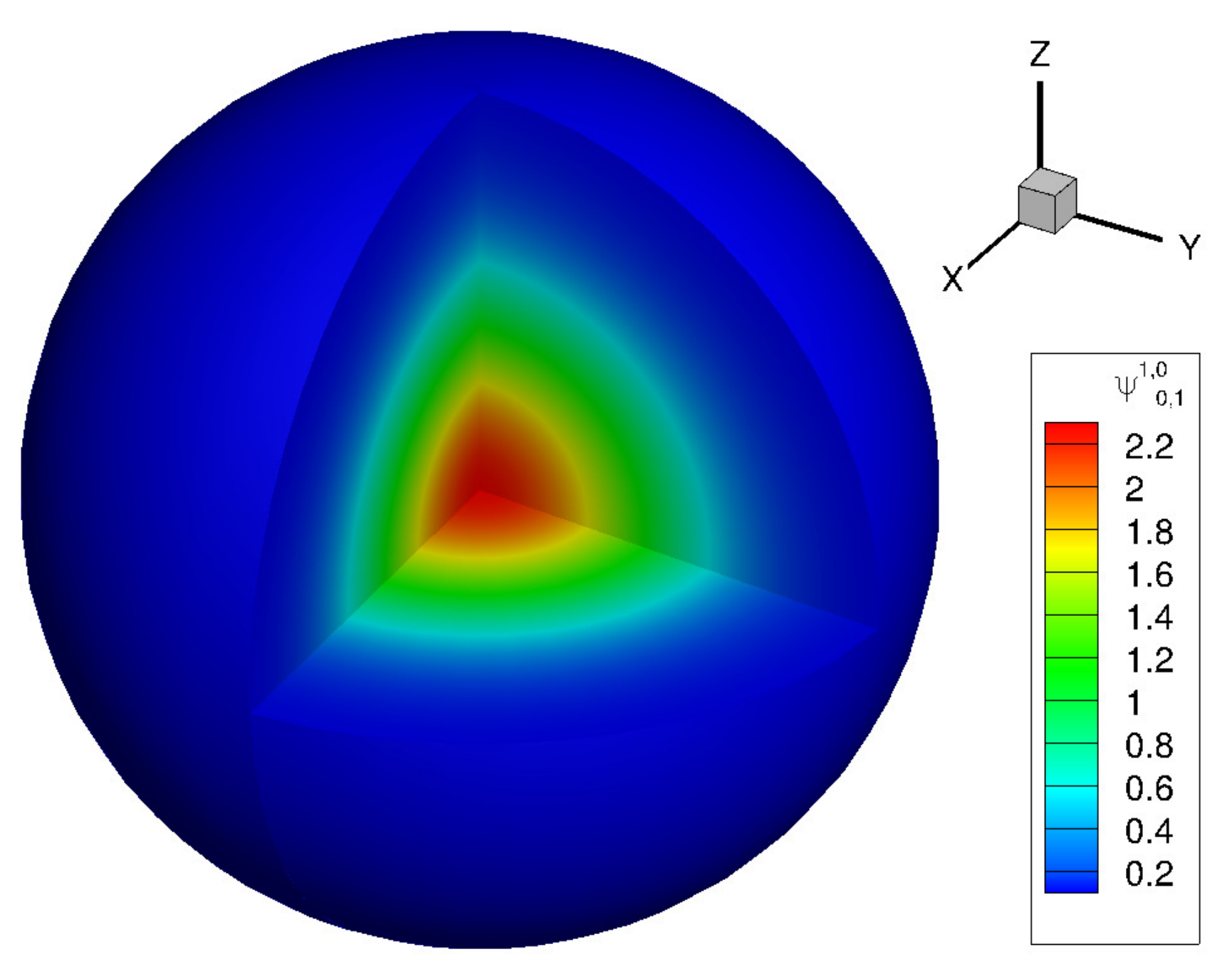}}
  \hspace{0in}
  \subfigure[$(\af, n, k, l)=(1,0,1,1).$]{
    \includegraphics[width=1.7in]{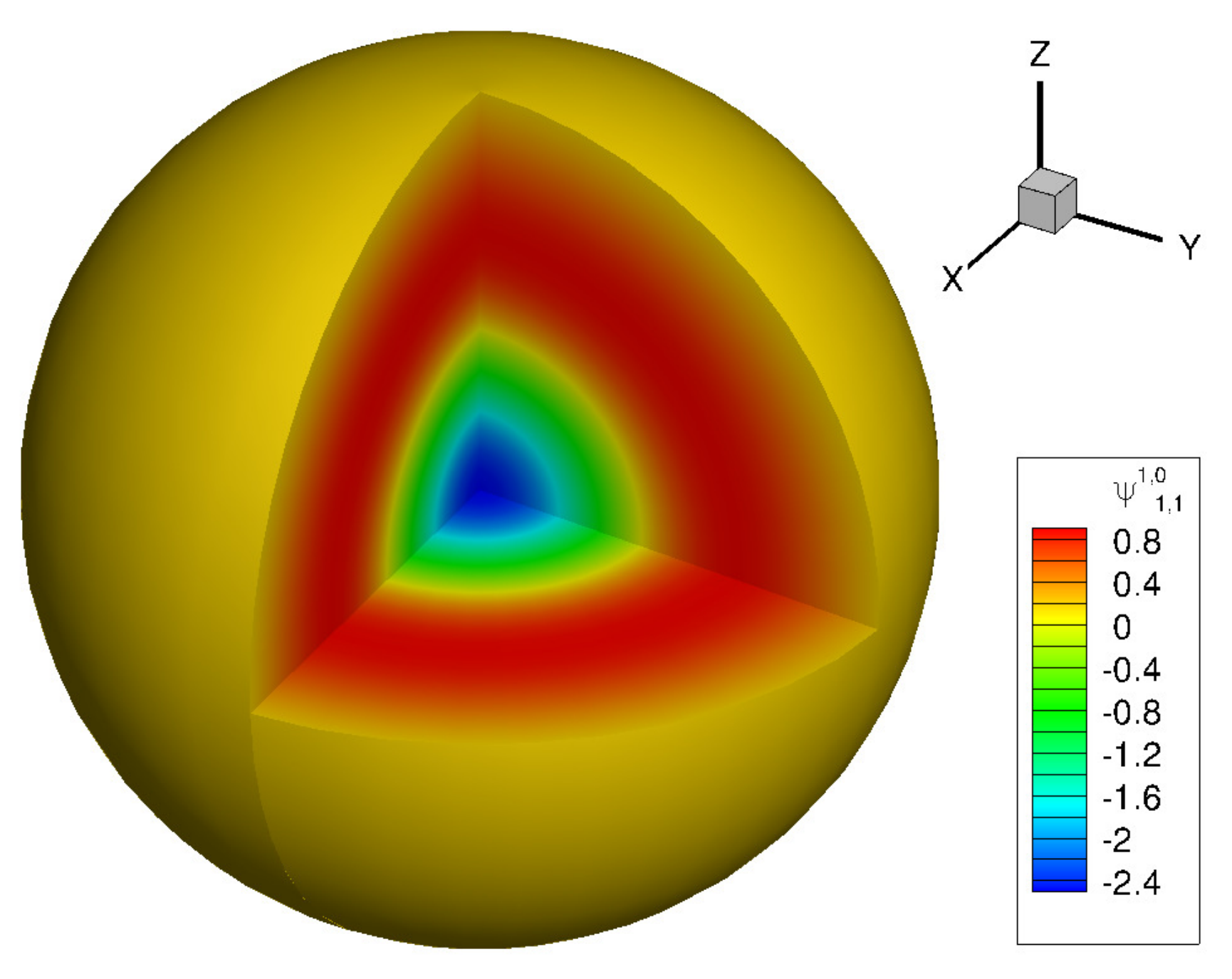}}
  \hspace{0in}
  \subfigure[$(\af, n, k, l)=(1,0,2,1).$]{
    \includegraphics[width=1.7in]{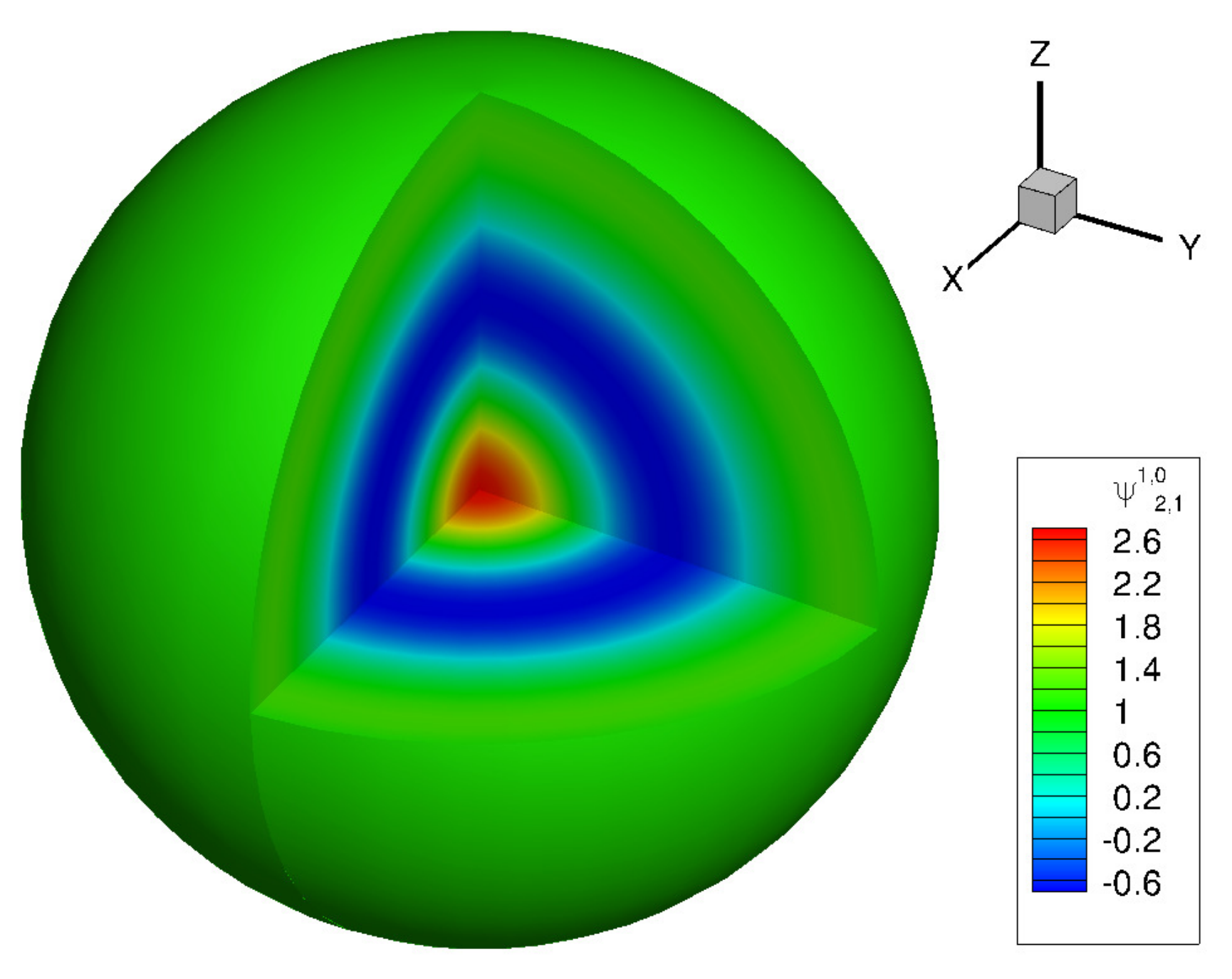}}
  \vfill
  \subfigure[$(\af, n, k, l)=(1,1,0,2).$]{
    \includegraphics[width=1.7in]{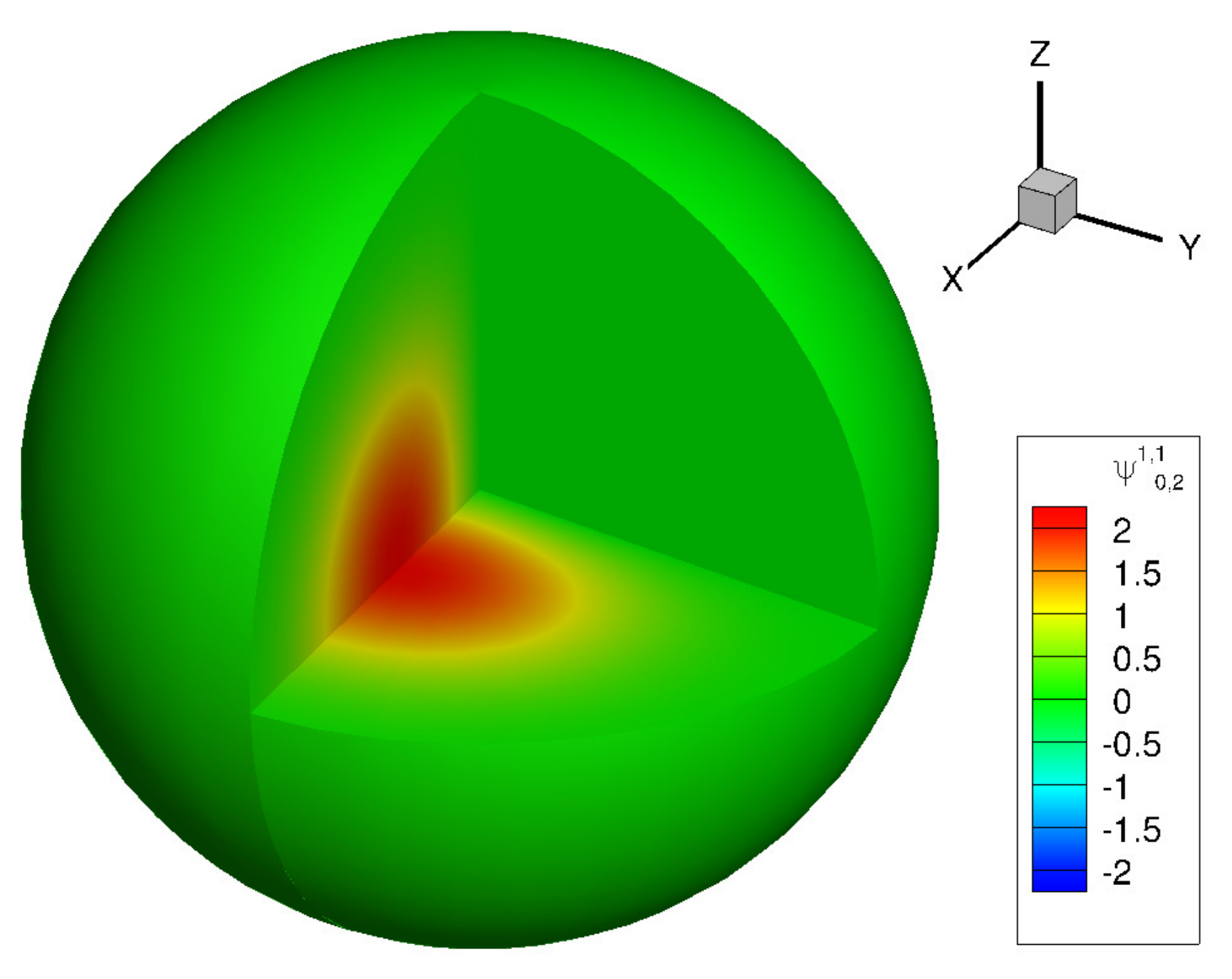}}
  \hspace{0in}
  \subfigure[$(\af, n, k, l)=(1,1,1,2).$]{
    \includegraphics[width=1.7in]{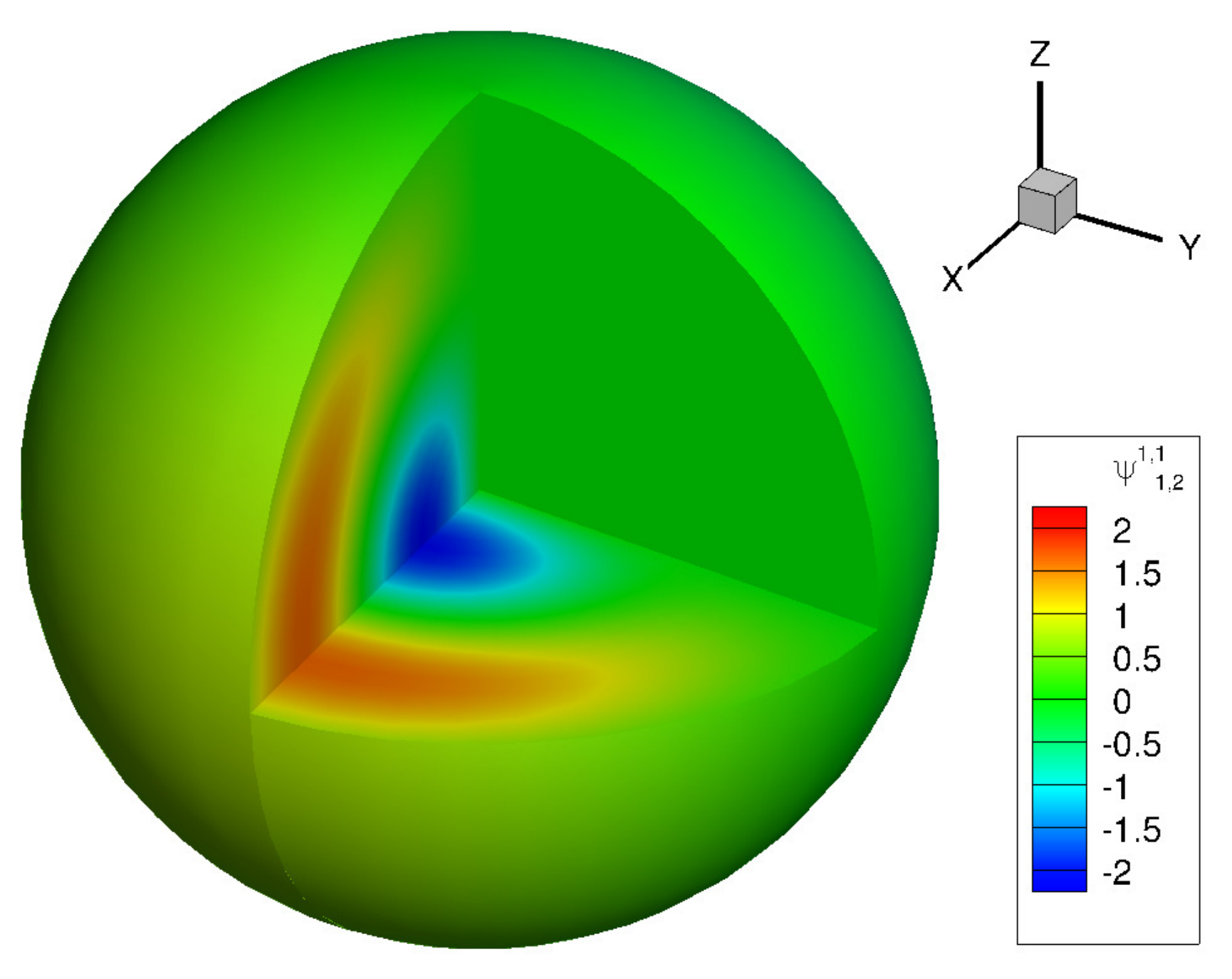}}
    \hspace{0in}
  \subfigure[$(\af, n, k, l)=(1,1,2,2).$]{
    \includegraphics[width=1.7in]{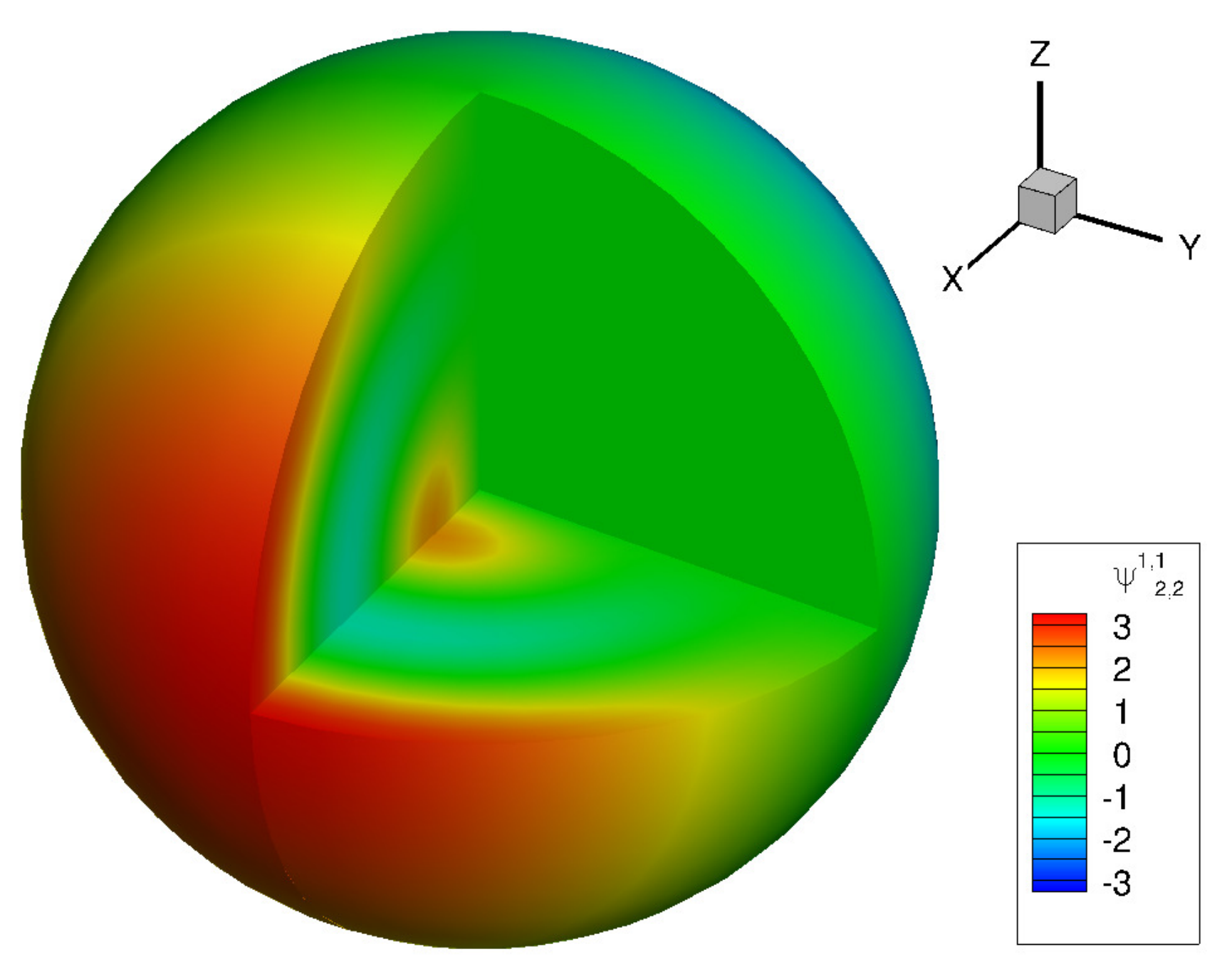}}
  \caption{Eigenfunctions  $\psi_{k,l}^{\af,n}$ with $c=10$ in $3$-dimension. }\label{surfpsi3d2}
\end{figure}

\bibliographystyle{plain}
\bibliography{Thesisreferpapers}

\end{document}